\documentclass{amsart}
\usepackage[english]{babel}
\usepackage{amssymb}
\usepackage{amsmath}
\usepackage{amsfonts}
\usepackage{amsthm}
\usepackage{enumerate}
\usepackage{color}

\newtheorem{theo}{Theorem}[section]
\newtheorem{lemm}[theo]{Lemma}
\newtheorem{defi}[theo]{Definition}

\newtheorem{prop}[theo]{Proposition}
\newtheorem{rema}{Remark}[section]
\numberwithin{equation}{section}

\newtheorem{assumption}{Hypothesis}[section]

\newcommand{\bal}{\begin{align}}
\newcommand{\bbal}{\begin{align*}}
\newcommand{\beq}{\begin{equation}}
\newcommand{\eeq}{\end{equation}}
\newcommand{\bca}{\begin{cases}}
\newcommand{\eca}{\end{cases}}

\newcommand{\pa}{\partial}
\newcommand{\fr}{\frac}

\newcommand{\ep}{\varepsilon}

\newcommand{\R}{\mathbb{R}}
\newcommand{\N}{\mathbb{N}}
\def\om{\omega}
\newcommand{\Nn}{\langle N\rangle}
\newcommand{\Nt}{\langle \tilde{N}\rangle}

\newcommand{\EQ}[1]{\begin{equation}\begin{split} #1 \end{split}\end{equation}}

\newcommand{\de}{\delta}

\newcommand{\T}{\mathbb{T}}
\newcommand{\Z}{\mathbb{Z}}

\newcommand{\al}{\alpha}

\title[Unconditional LWP for the fKdV equation  in $L^2(\R) $]{$L^2(\R) $-Unconditional well-posedness  for  low dispersion fractional KdV equations}
\author[L. Molinet and W. Zhu]{Luc Molinet and Weipeng Zhu}
\address[L. Molinet]{Institut Denis Poisson, Universit\'e de Tours, Universit\'e d'Orl\'eans, CNRS, Parc Grandmont, 37200 Tours, France}
\email[L. Molinet]{Luc.Molinet@univ-tours.fr}
\address[W. Zhu]{ School of Mathematics, Foshan University, Foshan, Guangdong 528000, China}
\email[W. Zhu]{mathzwp2010@163.com}

\keywords{low dispersion  fractional KdV equation;  unconditional uniqueness in $L^2(\R) $}
\subjclass[2010]{35Q53; 35A02}
\begin{document}
\maketitle
\setcounter{page}{001}

\begin{abstract}
 We show that the $ L^2(\R) $-unconditional well-posedness, that is well-known for the KdV equation, is shared by  KdV type equations with weaker dispersion. This is despite the difference in the nature of these equations, which are quasilinear while KdV is semilinear. More precisely we prove that the low dispersion fractional KdV equation 
$$
\partial_t u -D_x^\alpha \partial_x u +\pa_x(u^2)=0
$$
is unconditionally globally  well-posed in $L^2(\R) $ for $\alpha \in ]\frac{55}{38},2] $. Our method of proof combined refined bilinear estimates with the energy method  enhanced with Bourgain's type estimates developed in \cite{MV15}. 
\end{abstract}

\section{Introduction}
In this paper we  consider the unconditional uniqueness in $L^2(\R) $ for the  following family of fractional KdV type equations
\begin{equation}\label{gKdV}
\begin{cases}
\pa_t u+L_{\al+1} u+\fr12 \pa_x(u^2)=0,\quad (t,x)\in \R\times\R, \\
u(0,x)=u_0(x),\quad x\in\R,
\end{cases}
\end{equation}
where $\al>0$ is a real number and the linear operator $L_{\al+1}$ satisfies the following hypothesis:

\begin{assumption}\label{Hy1}
$L_{\al+1}$ is the Fourier multiplier operator by $-ip_{\al+1}$ where $p_{\al+1} \in C^1(\R)\cap C^2(\R\setminus\{0\})$ is a real-valued odd function satisfying, for some $\xi_0>0$, $p'_{\al+1}(\xi)\sim \xi^{\al}$ and $p''_{\al+1}(\xi)\sim \xi^{\al-1}$ for all $\xi\geq \xi_0$.
\end{assumption}
This family contains as special case the famous Korteweg-de Vries (KdV) equation ($\alpha=2$ and $ L_{3} =\partial_x^3 $)  and Benjamin-Ono (BO)  equation($\alpha=1$ and 
$ L_{2} =\partial_x D_x  $ where $ D_x $ is the Fourier multiplier by $|\xi| $) that read respectively 
$$
u_t + u_{xxx} + u u_x =0 \quad \text{and} \quad u_t + D_x u_x + u u_x =0 \; .
$$
The question we focus on in this paper is that of unconditional uniqueness in $ L^2 $ for the  quadratic dispersive equation \eqref{gKdV}. This means that  the uniqueness statement in the $L^2$-well-posedness theory does hold without restricting the solution to a resolution subspace specific to some particular method. It is worth noticing that, in the case of an equation with a quadratic non linear term, unconditional uniqueness in $ L^2 $  is optimal in the scale of classical Sobolev  spaces  since the nonlinear term $ u^2 $ does not make sense for $u\in H^\theta(\R) $ with $ \theta<0 $. Unconditional uniqueness is, in  particular, useful to prove the convergence of solutions to dissipative perturbations of the equation to the dissipative limit equation (see for instance \cite{M}). 

For nonlinear dispersive PDEs, the study of unconditional well-posedness goes back to
the work of Kato \cite{Kato95} who was the first to address this question for the nonlinear Schr\"odinger equation (NLS). Since then the unconditional well-posedness for NLS was further improved, see \cite{FPT03}, \cite{Herr19}, \cite{K21}, \cite{GKO13}. We also mention  \cite{MY20} for the derivative NLS equation,  \cite{KO12}, \cite{KO20}, \cite{MPV18}, \cite{MPV19} for the modified KdV equation and \cite{K19p} for the periodic modified Benjamin-Ono equation.

For the KdV equation global unconditionally well-posedness  in $L^2(\R) $ and $ L^2(\T) $ has been proven in (\cite{Zhou}, \cite{BIT11}, see also \cite{MV15}). As mentioned above, this result is optimal in the scale of classical Sobolev  spaces  since the nonlinear term $ u^2 $ does not make sense for $u\in H^\theta(\R) $ with $ \theta<0 $. It is thus a natural question to ask if this optimal result is shared by fractional KdV type equations with a weaker dispersion, i.e. $ \alpha<2 $.  Recall that, in sharp contrast to the KdV equation,  it is shown in \cite{MST} that for $ \alpha<2 $ these equations cannot be solved in $ H^s(\R) $ by a fixed point argument in any  resolution space. On the other hand, in \cite{HIKK10} it is proven that for $1<\alpha<2 $ and $ L_{\alpha+1}= \partial_x D_x^\alpha$, \eqref{gKdV} is globally well-posed in $ L^2(\R) $. This global well-posedness in $L^2(\R)$, and even below, is also well-known  for the  Benjamin-Ono equation   (see \cite{IK07}, \cite{MoPi12}, \cite{GKT23}). Note also that  this $L^2$-well-posedness has been recently  extended to  the Intermediate Long Wave equation (see \cite{CLOP24}, \cite{IfSa25}) viewed as a perturbation of the BO equation. Other classical  well-posedness results  but also unconditional well-posedness results on \eqref{gKdV} with $ \alpha\in ]0,2[$ can be found in \cite{Guo12}, \cite{MV15}, \cite{MPV25}. 
The aim of this paper is to prove that the unconditional global well-posedness  in $L^2(\R) $ is shared by  \eqref{gKdV} as soon as  $ \frac{55}{38} <\alpha\le 2 $ (We do not treat the case $ \alpha>2 $ that  is simpler and could clearly be handled by the same approach as for instance in \cite{MV15}). Up to our knowledge, this is the first result of this type in the literature for $ \alpha<2 $. Indeed, even in the special case of the BO equation where the equation is integrable, the unconditional uniqueness is only known to hold in $ H^s(\R) $ for $ s> 3-\sqrt{\frac{33}{4}}\sim 0.128 $ in  \cite{MP23} (in a forthcoming paper  \cite{MZ26} we get a slight improvement to $s>\frac{1}{10}$). 
Our strategy to prove the unconditional LWP is to enhance the improved energy method introduced in \cite{MV15} with refined bilinear estimates as in \cite{MT3}. However, since we have to evaluate the difference of two solutions  in Sobolev space with negative index, we will have to use also version of the bilinear estimates in short time Fourier restriction norms. Once the unconditional LWP in $L^2(\R) $  is established, we make use of the conservation of the  $L^2$-norm for smooth solutions to \eqref{gKdV} to obtain the unconditional GWP of  \eqref{gKdV}. Our main result reads as follows : 

\begin{theo}\label{th1} 
Assume that Hypotheses \ref{Hy1} is satisfied with $\fr{55}{38} < \al \leq 2$ then \eqref{gKdV} is unconditionally globally well-posed in $L^2(\R) $ in the following sense : \eqref{gKdV} is globally well-posed in $ L^2(\R) $  with a solution $u\in C(\R;L^2(\R)) $ which  is the unique maximal solution to  \eqref{gKdV} that belongs to $L^{\infty}_{loc} L^2(\R)$.
\end{theo}
\begin{rema}
The approach we used in this work is also applicable in the periodic setting. However, the refined linear and bilinear estimates in the periodic setting will be weaker and we expect to obtain  the unconditional well-posedness for the fKdV equation in $ L^2(\T) $ for a thiner range of the dispersive parameter $ \alpha$. This will be discussed in a future work. 
\end{rema}

Let us give now some details on our approach. For $u\in L^\infty_T L^2_x $ we  start by deriving refined bilinear estimates  by making use of bilinear estimates and Strichartz estimates on small intervals of time. Let us recall that refined bilinear estimates consist, after having projected on space frequencies of order $ N\gg 1$, in substituting the solution $u $ by its integral form on small time intervals $ I_{j,N} $ of length related with $ N $ and then resuming on the $j$. Note that the nonlinear term $ u^2$  that does appear in the Duhamel part can be evaluated in $ L^\infty_T L^1_{I_{j,N}} $ for $u\in L^\infty_T L^2_x $. With this refined bilinear estimates in hands we are able to prove that $ u$ belongs to a Bourgain's space of the form $ X_T^{-\theta,1}+X_T^{-\beta,2/3} $. Following \cite{MV15}  this enables us to bound the $ L^\infty_T L^2_x$-norm of $ u $ by  an energy estimate on the solution $ u$. Then we can handle the difference $ w$ of two solutions $ u$ and $ v$ belonging to $L^\infty_T L^2_x $. We assume that the initial difference $ w_0 $ belongs to the $L^2$-space $ \bar{H}^0(\R) $   that contains a  weight on the low frequencies. According to the properties of  the solutions to \eqref{gKdV} proven in the first part, it follows that $ w\in L^\infty_T \bar{H}^0(\R) \cap (X_T^{-\theta,1}+X_T^{-\beta,2/3}) $. We can thus try to evaluate $ w$ recursively in  $L^\infty_T \bar{H}^r(\R) \cap (X_T^{-\theta+r,1}+X_T^{-\beta+r,2/3}) $ for some $r<0 $. As for the solutions we first establish refined bilinear estimates on $ w$. Here it is worth noticing that since $r<0 $, we can not evaluate the nonlinear term $z w $, with $ z=u+v$, in $L^\infty_T L^1_{I_{j,N} }$ as we did for $ u^2$. Instead we evaluate $ zw $ in some short time Bourgain's space by making  use of the fact that  $ w$ belongs to $X_T^{-\theta+r,1}+X_T^{-\beta+r,2/3}$. With these refined bilinear estimates in hands we proceed as for $ u$ by bounded the $(X_T^{-\theta+r,1}+X_T^{-\beta+r,2/3})$-norm  and the $L^\infty_T \bar{H}^r$-norm of $ w$ in terms of these same norms. Each of these estimates require some   restrictions on $\alpha $ and/or on the parameters $\theta, \beta $ and $ r$. We conclude by checking that for $\alpha\in ]\frac{55}{38},2] $, we can choose 
$\theta, \beta $ and $ r$ such that all these restrictions are satisfied.\vspace*{4mm}

\noindent
{\bf Plan of the paper:} In the next section we introduce the notation and the function spaces and recall some basic estimates. Section \ref{Sect3} is devoted to the proof of  the refined linear and bilinear Strichartz estimates that will be essential tools. In Section \ref{Sect4} we first use the estimates proven in the preceding section to  prove that $ L^2(\R) $- solutions to \eqref{gKdV} do belong to Bourgain's space of the form $ X_T^{-\theta,1}+X_T^{-\beta,2/3} $, for $(\theta,\beta) $ satisfying some restrictions. Then, following \cite{MV15}, we use this to establish an $L^2(\R) $-energy estimate on these solutions. In Section \ref{Sect5} we tackle the difference of two solutions. This part is more delicate since this time we have to work in a negative regularity of index $r $.  In Section \ref{Sect6}  we first prove that for $\alpha\in ]\frac{55}{38},2] $, the set of $ (\alpha, \beta,r) $ satisfying all the restrictions, we established in the preceding sections, is not empty. Then we prove Theorem \ref{th1}. Finally in the Appendix, we prove the refined linear and bilinear Strichartz estimates we used in Section \ref{Sect5} for the difference of two solutions. As explained above these estimates are more involved that the ones on solutions since we have to work at a  negative regularity.

\section{Notation, Function spaces and Basic estimates}\label{Sect2}

\subsection{Notation}
For two nonnegative numbers $A,B$, the notation  $A\lesssim B$ denotes that there exists a positive constant $c>0$ independent of $A$ and $B$, such that $A\leq c B$. The symbol $A\simeq B$ represents $A\lesssim B$ and $B\lesssim A$. We also  denote $A\vee B:=\max\{A,B\}$ and $A\wedge B:=\min\{A,B\}$. 
For  $x\in\R$, $x+$, respectively $x-$ will denote a number slightly greater, respectively lesser, than $x$ .We also write $\langle x\rangle$ for $ (1+x^2)^{\fr12}$.

For $u=u(x,t)\in \mathcal{S}(\R^2),\ \mathcal{F}u=\hat{u}$ will denote its space-time Fourier transform, whereas $\mathcal{F}_x u=(u)^{\wedge_x}$ and $\mathcal{F}_t u=(u)^{\wedge_t}$ will denote its Fourier transform in space and in time, respectively.

Throughout the paper, we fix a smooth cutoff function $\eta$ such that
\begin{equation}\label{defeta}
\eta\in C_0^\infty(\R),\ 0\leq\eta\leq 1,\ \eta|_{[-1,1]}=1,\ \text{and}\ \text{supp}(\eta)\subset [-2,2].
\end{equation}
We set $\phi(\xi):=\eta(\xi)-\eta(2\xi)$. For $\ell\in\Z$, we define
$$\phi_{2^\ell}(\xi):=\phi(2^{-\ell}\xi),$$
and, for $\ell\in \N^*$,
$$\psi_{2^\ell}(\xi,\tau):=\phi_{2^\ell}(\tau-p_{\al+1}(\xi)),$$
where $ip_{\al+1}$ is the Fourier symbol of $L_{\al+1}$. By convention, we also denote
$$\psi_{1}(\xi,\tau):=\eta(2(\tau-p_{\al+1}(\xi))).$$
Any summations over capitalized variables such as $N,L,K$ or $M$ are presumed to be dyadic. Unless
stated otherwise, we work with homogeneous dyadic decomposition for the space-frequency variables and nonhomogeneous decompositions for modulation variables, i.e., these variables range over numbers of the form $2^k:k\in\Z$ and $2^k:k\in\N$, respectively. Then we have that
$$\sum_{N>0}\phi_N(\xi)=1\ \text{for all}\ \xi\in\R^{*}\quad \text{and}\quad \text{supp}(\phi_N)\subset\{\fr12N\leq|\xi|\leq 2N\}\ \text{for} \ N\in\{2^k:k\in\Z\},$$
and
$$\sum_{L\geq 1}\psi_L(\xi,\tau)=1\ \text{for all}\ (\xi,\tau)\in\R^2.$$
Let us define the Littlewood–Paley multipliers by
$$P_N u=\mathcal{F}_x^{-1}(\phi_N \mathcal{F}_x u),\quad Q_L u=\mathcal{F}_x^{-1}(\psi_N \mathcal{F}_x u),$$
$P_{\geq N}:=\sum_{K\geq N}P_K,\ P_{\leq N}:=\sum_{K\leq N}P_K,\ P_{\sim N}:=\sum_{K\sim  N}P_K , \ Q_{\geq L}:=\sum_{K\geq L}Q_K$ and $Q_{\leq L}:=\sum_{K\leq L}Q_K$. For brevity we also write $u_N=P_N u,\ u_{\leq N}=P_{\leq N}u$, etc.

Let $a$ be a (possibly complex-valued) bounded measurable function on $\R^2$
and define the pseudoproduct
operator $\Lambda_{a}$ on $\mathcal{S}(\R)^2$ by
$$\mathcal{F}(\Lambda_{a}(f,g))(\xi)=\int_{\R}\hat{f}(\xi_1)\hat{g}(\xi-\xi_1)a(\xi,\xi_1)d\xi_1.$$
This bilinear operator behaves like a product in the sense that it satisfies the following properties:
\begin{equation}\label{eq2-1}
\begin{cases}
\Lambda_{a}(f,g)=fg \quad \text{if}\ a\equiv 1, \\
\int_{\R}\Lambda_{a}(f,g)h \ dx =\int_{\R}f\Lambda_{a_1}(g,h)\ dx=\int_{\R}\Lambda_{a_2}(f,h)g\ dx,
\end{cases}
\end{equation}
with $a_1(\theta,\theta_1)=\bar{a}(\theta_1,\theta)$ and $a_2(\theta,\theta_1)=\bar{a}(\theta-\theta_1,\theta)$ for any real-valued functions $f,g,h\in \mathcal{S}(\R)$.We easily check from the Bernstein inequality that, if $f_i\in L^2(\R)$ has a Fourier transform
localized in an annulus $\{|\xi|\sim N_i\}, i=1,2,3$, then
\begin{equation}\label{eq2-2}
\Big| \int_{\R} \Lambda_{a}(f_1,f_2)f_3 \ dx\Big|\lesssim N_{\min}^{\fr12}\prod_{i=1}^3\|f_i\|_{L^2},
\end{equation}
where the implicit constant only depends on $\|a\|_{L^\infty(\R^2)}$
and $N_{\min}=\min\{N_1,N_2,N_3\}$. Moreover, in this work we  will require $ a$ to satisfy
\begin{equation}\label{estcom}
\|\Lambda_a(g,h)\|_{L^2_{x}} \lesssim \|g\|_{L^\infty_x} \|h\|_{L^2_x}\quad  , \forall (g,h)\in ( L^\infty(\R) \cap L^2(\R))^2 
\end{equation}
Note that when $a\equiv1$,  $\Lambda_a(u,v)=uv$  obviously satisfies  \eqref{estcom}.
With this notation
in hand, we will be able to systematically estimate terms of the form
$$\int_{\R} P_N(u^2)\pa_x P_N u \ dx \; .$$

\subsection{Function spaces}
For $1\le p\le\infty$, $L^p(\R)$ denotes the usual Lebesgue space and for $s\in\R$, $H^s(\R)$ denotes the usual $L^2$-based Sobolev space.
If $B$ is a space of functions on $\R$, $T>0$ and $1\le p\le\infty$, we define the spaces $L^p_TB_x$ and $L^p_tB_x$ by the norms
$$
\|f\|_{L^p_TB_x} = \left\| \|f\|_{B}\right\|_{L^p([0,T])} \ \textrm{ and } \|f\|_{L^p_tB_x} = \left\| \|f\|_B \right\|_{L^p(\R)}.
$$
For $s,b\in\R$, we introduce the Bourgain spaces $X^{s,b}$ related to the linear part of \eqref{gKdV} as the completion
of the Schwartz space $\mathcal{S}(\R^2)$ under the norm
\begin{equation}\label{eq2-3}
\|v\|_{X^{s,b}}:=\Big( \int_{\R^2}\langle \tau-p_{\al+1}(\xi)\rangle^{2b} \langle\xi\rangle^{2s}|\hat{v}(\xi,\tau)|^2 d\xi d\tau \Big)^{\fr12},
\end{equation}
where $\langle x\rangle:=1+|x|$ and $i p_{\al+1}$ is the Fourier symbol of $L_{\al+1}$. Recall that
$$\|v\|_{X^{s,b}}=\|U_\al(-t)v\|_{H^{s,b}_{x,t}},$$
where $U_\al(t)=\exp(t L_{\al+1})$ is the generator of the free evolution associated with \eqref{gKdV}.
We also use the Besov type $X^{s,b}$ spaces: for $b\in\R$,
$$\|v\|_{X^{0,b,1}}=\sum_{L\geq 1}L^b\|Q_L v\|_{L^2_{t,x}}.$$

In this paper, we use the frequency envelope method (see for instance [33] and [16]) in
order to show the continuity result with respect to initial data. To this aim, we first
introduce the following:
\begin{defi}\label{def1}
Let $1<\delta\le 2 $. An acceptable frequency weight $\{\omega_N^{(\delta)}\}_{N\in 2^{\Z}}$ is defined as a
dyadic sequence satisfying $\omega_N=1$ for $N<1$ and $\omega_N \leq \omega_{2N} \leq \delta\omega_N$ for $N\geq 1$. We simply write $\{\omega_N\}$ when there is no confusion.
\end{defi}
With an acceptable frequency weight $\{\omega_N\}$, we slightly modulate the classical Sobolev
spaces in the following way: for $s\geq 0$, we define $H^s_\omega(\R)$ with the norm
$$\|v\|_{H^s_\omega}:=\Big( \sum_{N\in 2^{\Z}}\omega_N^2 \langle N\rangle^{2s} \|P_N v\|_{L^2_{x}}^2 \Big)^{\fr12}.$$
Note that $H^s_\omega(\R)=H^s(\R)$ when we choose $\omega_N \equiv 1$, $L^2_\omega(\R)=H^0_\omega(\R)$.

Finally, we will use restriction-in-time versions of these spaces. Let $T>0$ be a positive time and
let $Y$ be a normed space of space-time functions. The restriction space $Y_T$ will be the space of functions
$v:\R\times]0,T[\rightarrow\R$ satisfying
$$\|v\|_{Y_T}:=\inf\{\|\tilde{v}\|_Y \ |\ \tilde{v}:\R\times\R\rightarrow\R,\ \tilde{v}|_{\R\times]0,T[}=v\}<\infty.$$
\subsection{Preliminary lemmas}
\begin{lemm}{\rm (see \cite{MV15})}\label{le2-1}
Let $L\geq 1,\ 1\leq p\leq \infty$ and $s\in\R$. The operator $Q_{\leq L}$ is bounded in $L^{p}_t H^s$ uniformly in $L\geq 1$.
\end{lemm}

\begin{lemm}{\rm (see \cite{MV15})}\label{le2-2}
For any $R>0$ and $T>0$,
\begin{equation*}
\|1^{\text{high}}_{T,R}\|_{L^1}\lesssim T\wedge R^{-1}
\end{equation*}
and for $ 1\le p\le \infty $, 
\begin{equation*}
\|1^{\text{low}}_{T,R}\|_{L^p}\lesssim T^{1/p}.
\end{equation*}.
\end{lemm}

\begin{lemm}{\rm (see \cite{MV15})}
Let $u\in L^2(\R^2)$. Then, for any $T>0,\ R>0$ and $L\gg R$,
\begin{equation*}
\|Q_L(1^{\text{low}}_{T,R}u)\|_{L^2_{t,x}}\lesssim \|Q_{\sim L} u\|_{L^2_{t,x}}.
\end{equation*}
\end{lemm}
We recall now the following standard  linear estimates in Bourgain's spaces (see \cite{IKT}).
\begin{prop}
For any time interval $I=[t_0, t_1]$, we have
\begin{equation}\label{Bour1}
 \left\|\int_{t_0}^t e^{(t-t') L_{\alpha+1}} f(t')dt'\right\|_{X^{0,\frac 12,1}_I}\lesssim \|f\|_{X^{0,-\frac 12,1}_I}\wedge |I|^\frac{1}{2} \|f\|_{L^2_I L^2_x}\; .
\end{equation}
\end{prop}

\subsection{Strichartz estimates}
We will need the following Strichartz estimates that are essentially proven in Theorem 2.1 of \cite{KPV3}.
\begin{lemm}\label{str}
Let  $L_{\al+1}$ satisfying Hypothesis \ref{Hy1}, $T>0$, $N>0 $, $ \varphi\in L^2(\R) $ and $ f\in L^\frac{4}{3}_T L^1_x $. Then
\begin{eqnarray}
\|e^{-tL_{\al+1}}P_{N}\varphi\|_{L^4_T L^\infty_{x}}
&\lesssim & N^{\fr{1-\al}{4}}||P_{N}\varphi||_{L^2_x} \label{stri1} \\
\text{ and } \quad \quad \Bigr\|\int_{\R} e^{-(t-t')L_{\al+1}}P_{N}f dt'\Bigl\|_{L^4_T L^\infty_x}
&\lesssim &N^{\fr{1-\al}{2}}||P_{N}f||_{L^{\fr43}_T L^1_{x}}. \label{stri2}
\end{eqnarray}
\end{lemm}

\section{Refined linear and  bilinear Strichartz estimates}\label{Sect3}
\subsection{A refined linear Strichartz estimate}
\begin{lemm}\label{Cu2}
	Let $0<T\le 2,\,  s\geq 0$ and $u \in C([0,T];H^s(\R))$ satisfying
	\eqref{gKdV} on $]0,T[\times \R$  with $L_{\al+1}$ satisfying Hypothesis \ref{Hy1}. Then for any $ N\ge 1$ it holds
	
	\begin{align*}
		\langle N \rangle^s\|u_{N}\|_{L^2_T L^\infty_x}
		\lesssim T^{\fr14} N^{\fr{2-\al}{3}}
		\|u\|_{L^\infty_T H^s} (1+\|u\|_{L^2_T L^2_x}). 
	\end{align*}
\end{lemm}
\begin{proof}
We follow \cite{KT03}. We chop the time interval $[0,T]$ into small pieces of length $\sim N^{-\delta}T$, $\de=\fr{5-\al}{3}$, i.e., we define $\{I_{j,N}\}_{j\in J_{N}}$ with $\# J_{N}\sim N^\delta$ so that $\bigcup_{j\in J_{N}}=[0,T],\ |I_{j,N}|\sim N^{-\delta}T$. Then substituting $u $ by its integral formulation on each small interval, applying Lemma \ref{str} and making use of the classical Christ-Kiselev lemma (cf. \cite{CK01}), we eventually get
	\begin{align*}
		\langle N \rangle^{2s}&\|u_{N}\|_{L^2_T L^\infty_x}^2=\langle N \rangle^{2s}\sum_{j\in J_{N}}\|u_{N}\|_{L^2({I_{j,N}};L^\infty_x)}^2
		\\& \lesssim \langle N \rangle^{2s}\sum_{j\in J_{N}} |I_{j,N}|^{\fr12} \|u_{N}\|_{L^4({I_{j,N}};L^\infty_x)}^2
		\\& \lesssim \langle N \rangle^{2s}\sum_{j\in J_{N}} |I_{j,N}|^{\fr12}N^{\fr{1-\al}{2}}
		\big(\|u_{N}\|_{L^\infty({I_{j,N}};L^2_x)}^2
		+N^{\fr{1-\al}{2}}  \|\pa_x P_N(u^2)\|_{L^{\fr43}({I_{j,N}};L^1_x)}^2 \big)
		\\
		& \lesssim \langle N \rangle^{2s}\sum_{j\in J_{N}} |I_{j,N}|^{\fr12}N^{\fr{1-\al}{2}}
		\big(\|u_{N}\|_{L^\infty_T L^2_x}^2
		+N^{\fr{1-\al}{2}} N^2  |I_{j,N}|^{\fr32}\| P_N(u^2)\|_{L^\infty_T L^1_x)}^2 \big) \\
		& \lesssim \sum_{j\in J_{N}} |I_{j,N}|^{\fr12}N^{\fr{1-\al}{2}}
		\big(\|u\|_{L^\infty_T H^s}^2
		+N^{\fr{1-\al}{2}} N^2  |I_{j,N}|^{\fr32}\| u\|_{L^\infty_T H^s}^2\| u\|_{L^\infty_T L^2_x}^2 \big) ,
		\end{align*}
		where we use the fact that $\langle N \rangle^{s}\|P_{N}(u^2) \|_{L^{\infty}_{T} L^1_{x}}\lesssim  \|u\|_{L^\infty_T H^s}\|u\|_{L^\infty_T L^2_{x}}$  for $ s\ge 0 $.
		Now we optimize by requiring the expression  in the above parentheses to be bounded uniformly in $ N $, i.e. $N^{\fr{1-\al}{2}} N^2  N^{-\fr32 \delta} \sim 1 $. This leads to $ \delta=\frac{5-\alpha}{3} $. With this choice of $ \delta $, recalling that  $\# J_{N}\sim N^\delta$, we get 
		\begin{align*}
\langle N \rangle^{2s}\|u_{N}\|_{L^2_T L^\infty_x}^2& \lesssim \sum_{j\in J_{N}}  (TN^{-\delta })^{\fr12}N^{\fr{1-\al}{2}}
		\|u\|_{L^\infty_T H^s}^2 (1+\|u\|_{L^2_T L^2_x}^2)
				\\
				& \lesssim   T^\frac{1}{2} N^{\frac{\delta}{2}} N^{\fr{1-\al}{2}}
		\|u\|_{L^\infty_T H^s}^2 (1+\|u\|_{L^2_T L^2_x}^2)
				\\& \lesssim T^{\fr12} N^{\fr{4-2\al}{3}} \|u\|_{L^\infty_T H^s}^2 (1+\|u\|_{L^2_T L^2_x}^2)\, .
	\end{align*}
	This completes the proof.
\end{proof}

\subsection{Refined bilinear Strichartz estimates}
The following bilinear estimates on small interval of times can be found for instance in \cite{MT25}.
\begin{prop}\label{prov}
	Let $a\in L^\infty(\R^2)$ be such that $\|a\|_{L^\infty}\lesssim 1, \delta\in  [1,\infty[, \al\ge 1, N_1\gtrsim N_2>0 $ with $N_1\geq 1$, $0<T\le 2 $ and $v_1,v_2\in X^{0,\fr12,1}$. Suppose that $I\subset \R$ satisfies $|I|\sim N_1^{-\delta}T$. Then it holds
	\begin{align}\label{prov1}
		\|\Lambda_a(P_{N_1} v_1 P_{\lesssim N_2}v_2)\|_{L^2(I;L^2)}
		\lesssim T^{\fr14} N_1^{-\fr{\al+\delta-1}{4}}\|P_{ N_1}v_1\|_{X^{0,\fr12,1}}\|P_{\lesssim N_2}v_2\|_{X^{0,\fr12,1}}.
	\end{align}
	Moreover, when $N_1\ \gg N_2$, \eqref{prov1} can be improved to
	\begin{align}\label{prov2}
		\|\Lambda_a(P_{N_1}v_1,P_{\lesssim N_2}v_2)\|_{L^2(I;L^2)}
		\lesssim N_1^{-\fr{\al}{2}}\|P_{ N_1}v_1\|_{X^{0,\fr12,1}}\|P_{\lesssim N_2}v_2\|_{X^{0,\fr12,1}}.
	\end{align}
\end{prop}

\begin{prop}\label{pro}
	Let $a\in L^\infty(\R^2)$ be such that $\|a\|_{L^\infty}\lesssim 1, \delta\in [1,\infty[, \al\ge 1, N_1\gtrsim N_2>0 $ with $  N_1\geq 1, \, 0<T\le 2 $ and $\varphi_1,\varphi_2 \in L^2(\R)$. Suppose that $I\subset \R$ satisfies $|I|\sim N_1^{-\delta}T$. Then it holds
	\begin{align}\label{prop1}
		\|\Lambda_a(e^{-tL_{\al+1}}P_{N_1}\varphi_1,e^{-tL_{\al+1}}P_{\lesssim N_2}\varphi_2)\|_{L^2(I;L^2)}
		\lesssim T^{\fr14} N_1^{-\fr{\al+\delta-1}{4}}\|P_{N_1}\varphi_1\|_{L^2}\|P_{\lesssim N_2}\varphi_2\|_{L^2}.
	\end{align}
	Moreover, when $N_1\ \gg N_2$, \eqref{prop1} can be improved to
	\begin{align}\label{prop2}
		\|\Lambda_a(e^{-tL_{\al+1}}P_{N_1}\varphi_1,e^{-tL_{\al+1}}P_{\lesssim N_2}\varphi_2)\|_{L^2(I;L^2)}
		\lesssim N_1^{-\fr{\al}{2}}\|P_{N_1}\varphi_1\|_{L^2}\|P_{\lesssim N_2}\varphi_2\|_{L^2}.
	\end{align}
\end{prop}
With these last estimates in hands we establish the following bilinear estimates on small intervals for solutions to \eqref{gKdV}.
\begin{prop}\label{sti}
	Let $a\in L^\infty(\R^2)$ be such that $\|a\|_{L^\infty}\lesssim 1, \delta=\fr{5-\al}{3}, \al\ge 1 , N_1\gtrsim N_2>0 $ with $ N_1\geq 1$ and $0<T\le 2$. Suppose that $I\subset \R$ satisfies $|I|\sim N_1^{-\delta}T$. Let $u_1,u_2 \in C(I;L^2(\R))$ satisfying
	\begin{equation*}
		\begin{cases}
			\pa_t u_k+L_{\al+1} u_k+ \pa_x f_k=0,\\
			u_k(t_0,x)=\varphi_k(x),
		\end{cases}
	\end{equation*}
	on $I\times \R$ for $k=1,2$, with $L_{\al+1}$ satisfying Hypothesis \ref{Hy1} and $ t_0\in I$. Then
	\begin{align}\label{pros1}
		\|\Lambda_a(P_{N_1}u_1,P_{\lesssim N_2}u_2)\|_{L^2(I;L^2)}^2
		&\lesssim T^{\fr12}N_1^{-\fr{\al+\delta-1}{2}}(\|P_{N_1}\varphi_1\|_{ L^2_{x}}^2+\|P_{N_1}f_1 \|_{L^{\infty}_I L^1_{x}}^2)\\
		&\quad\times(\|P_{\lesssim N_2}\varphi_2\|_{ L^2_{x}}^2+\|P_{\lesssim N_2}f_2\|_{L^{\infty}_I L^1_{x}}^2). \nonumber
	\end{align}
	Moreover,
	\begin{align}\label{pros2}
		\|\Lambda_a(P_{N_1}u_1,P_{\lesssim N_2}u_2)\|_{L^2(I;L^2)}^2
		&\lesssim N_1^{-\al}(\|P_{N_1}\varphi_1\|_{ L^2_{x}}^2+\|P_{N_1}f_1 \|_{L^{\infty}_I L^1_{x}}^2)\\
		&\quad\times(\|P_{\lesssim N_2}\varphi_2\|_{ L^2_{x}}^2+\|P_{\lesssim N_2}f_2\|_{L^{\infty}_I L^1_{x}}^2), \nonumber
	\end{align}
	whenever $N_1\gg N_2$.
\end{prop}
\begin{proof}
Following \cite{MT3}, we substitute $ u_i $ by its integral formulation on $ I $. By Proposition \ref{pro}, \eqref{stri2}  and an application of Christ-Kiselev lemma
	(see for instance Corollaries 3.11 and 3.12 in \cite{MT3}) , we have
	\begin{align*}
		\|\Lambda_a(P_{N_1}u_1,P_{\lesssim N_2}u_2)\|_{L^2(I;L^2)}^2
		&\lesssim T^{\fr12}N_1^{-\fr{\al+\delta-1}{2}}(\|P_{N_1}\varphi_1\|_{ L^2_{x}}^2+N_1^2 N_1^{\fr{1-\al}{2}}\|P_{N_1}f_1 \|_{L^{\fr43}_I L^1_{x}}^2)\\
		&\quad\times(\|P_{\lesssim N_2}\varphi_2\|_{ L^2_{x}}^2
		+\sum_{0<\tilde{N}_2\lesssim N_2} \tilde{N}_2^2 \tilde{N}_2^{\fr{1-\al}{2}}\|P_{\tilde{N}_2}f_2\|_{L^{\fr43}_I L^1_{x}}^2
		)
		\\&\lesssim T^{\fr12} N_1^{-\fr{\al+\delta-1}{2}}(\|P_{N_1}\varphi_1\|_{ L^2_{x}}^2+N_1^{\fr{5-\al}{2}}|I|^{\fr32}\|P_{N_1}f_1 \|_{L^{\infty}_I L^1_{x}}^2)\\
		&\quad\times\Bigl(\|P_{\lesssim N_2}\varphi_2\|_{ L^2_{x}}^2+N_2^{\fr{5-\al}{2}}|I|^{\fr32}\|P_{N_2}f_2\|_{L^{\infty}_I L^1_{x}}^2\Bigr)
		\\&\lesssim T^{\fr12} N_1^{-\fr{\al+\delta-1}{2}}(\|P_{N_1}\varphi_1\|_{ L^2_{x}}^2+T^{\fr32}\|P_{N_1}f_1 \|_{L^{\infty}_I L^1_{x}}^2)\\
		&\quad\times(\|P_{\lesssim N_2}\varphi_2\|_{ L^2_{x}}^2+T^{\fr32}\|P_{\lesssim N_2}f_2\|_{L^{\infty}_I L^1_{x}}^2)
		\\&\lesssim T^{\fr12}N_1^{-\fr{\al+\delta-1}{2}}(\|P_{N_1}\varphi_1\|_{ L^2_{x}}^2+\|P_{N_1}f_1 \|_{L^{\infty}_I L^1_{x}}^2)\\
		&\quad\times(\|P_{\lesssim N_2}\varphi_2\|_{ L^2_{x}}^2+\|P_{\lesssim N_2}f_2\|_{L^{\infty}_I L^1_{x}}^2),
	\end{align*}
	where we used that for our choice of $\delta$, it holds  $ N_1^{\fr{5-\al}{2}}|I|^{\fr32}\sim   T^\frac{3}{2} $.
	Moreover, whenever $N_1 \gg  N_2$, by Proposition \ref{pro}, $T^{\fr12}N_1^{-\fr{\al+\delta-1}{2}}$ is replaced by $N_1^{-\al}$, which leads to the desired result.
\end{proof}

We are now in position to state our refined bilinear estimates on $ [0,T] $ for solutions to \eqref{gKdV}.
\begin{prop}\label{bl0}
	Let $a\in L^\infty(\R^2)$ be such that $\|a\|_{L^\infty}\lesssim 1, \al\ge 1, 0<T\le 2, N_1\gtrsim N_2,N_1\geq 1, s\geq 0$ and $u_1,u_2 \in L^\infty(]0,T[;H^s(\R))$ satisfying
	\eqref{gKdV} on $]0,T[\times \R$ for $k=1,2$, with $L_{\al+1}$ satisfying Hypothesis \ref{Hy1}. Then
	\begin{align}
		\langle N_1 \rangle^{s}\|\Lambda_a(P_{N_1}u_1,P_{\lesssim N_2}u_2)\|_{L^2_{T,x}}
		&\lesssim T^{\fr14} N_1^{\fr{2-\al}{3}} 
		  \|u_1\|_{L^\infty_T H^{s}}(1+\|u_1\|_{L^\infty_T L^2_{x}}) \nonumber \\
		  & \hspace*{15mm} \times\|u_2\|_{L^\infty_T L^2_{x}}(1+\|u_2\|_{L^\infty_T L^2_{x}}),\label{bili1} \; .
	\end{align}
	Moreover,
	\begin{align}
		\langle N_1 \rangle^{s}\|\Lambda_a(P_{N_1}u_1,P_{\lesssim N_2}u_2)\|_{L^2_{T,x}}
		&\lesssim N_1^{\fr{5-4\al}{6}}  \|u_1\|_{L^\infty_T H^{s}}(1+\|u_1\|_{L^\infty_T L^2_{x}}) \nonumber \\
		 & \hspace*{15mm} \times\|u_2\|_{L^\infty_T L^2_{x}}(1+\|u_2\|_{L^\infty_T L^2_{x}}),\label{bili2} \; ,	\end{align}
	whenever $N_1\gg N_2$.
\end{prop}
\begin{proof}
	We chop the time interval $[0,T]$ into small pieces of length $\sim N_1^{-\delta}T$, i.e., we define $\{I_{j,N_1}\}_{j\in J_{N_1}}$ with $\# J_{N_1}\sim N_1^\delta$ so that $\bigcup_{j\in J_{N_1}}=[0,T],\ |I_{j,N_1}|\sim N_1^{-\delta}T$. For $j\in J_{N_1}$, we denote by $c_{j,N_1}$
	 the infimum of the interval $ I_{j,N_1} $. 
	For simplicity, we write $c_j=c_{j,N_1}$.
	Then applying Proposition \ref{sti}, on $I_{j,N_1}$ it holds
	\begin{align*}
		\|\Lambda_a(P_{N_1}u_1,P_{N_2}u_2)\|_{L^2(I_{j,N_1};L^2)}^2
		&\lesssim T^{\fr12}N_1^{-\fr{\al+\delta-1}{2}}(\|P_{N_1}u_1(c_j)\|_{ L^2_{x}}^2+\|P_{N_1}(u_1^2) \|_{L^{\infty}_{I_{j,N_1}} L^1_{x}}^2)\\
		&\quad\times(\|P_{\lesssim N_2}u_2(c_j)\|_{ L^2_{x}}^2+\|P_{\lesssim N_2}(u_2^2)\|_{L^{\infty}_{I_{j,N_1}} L^1_{x}}^2).
	\end{align*}
	This leads to 
	\begin{align*}
		\langle N_1 \rangle^{2s} &\|\Lambda_a(P_{N_1}u_1,P_{ N_2}u_2)\|_{L^2_{T,x}}^2=\langle N_1 \rangle^{s}\sum_{j\in J_{N_1}}\|\Lambda_a(P_{N_1}u_1,P_{N_2}u_2)\|_{L^2(I_{j,N_1};L^2)}^2
		\nonumber \\
		&\lesssim T^{\fr12}N_1^{-\fr{\al-\delta-1}{2}} 
		 \|u_1\|_{L^\infty_T H^{s}}^2 (1+\|u_1\|_{L^\infty_T L^2_x}^2)  \|u_2\|_{L^\infty_T L^2_x}^2 (1+\|u_2\|_{L^\infty_T L^2_x}^2)\nonumber \\	
		 	&\lesssim T^{\fr12}N_1^{\fr{4-2\al}{3}}  \|u_1\|_{L^\infty_T H^{s}}^2 (1+\|u_1\|_{L^\infty_T L^2_x}^2)  \|u_2\|_{L^\infty_T L^2_x}^2 (1+\|u_2\|_{L^\infty_T L^2_x}^2) ,
	\end{align*}
	where we use the fact that $\langle N_1 \rangle^{s}\|P_{N_1}(u_1^2) \|_{L^{2}_{T} L^1_{x}}\leq T^{\fr12} \|u_1\|_{L^\infty_T H^s}\|u_1\|_{L^\infty_T L^2_{x}}$ for $ s\ge 0$.
	Finally, when $N_1\gg  N_2$, it can be obtained in the same way by making use of \eqref{pros2} instead of \eqref{pros1}.
\end{proof}

\begin{lemm}\label{Cuv}Let $0<T\le 2 , s\geq 0$, $v\in X^{0,\fr12,1}_T $and $u \in C([0,T];H^s(\R))$ satisfying
	\eqref{gKdV} on $]0,T[\times \R$  with $L_{\al+1}$ satisfying Hypothesis \ref{Hy1}. Then
	for $N_1 \gg N_2,\ N_1\ge 1 $ it holds
	\begin{align*}
		\langle N_1 \rangle^s \|u_{N_1}v_{N_2}\|_{L^2_T L^2_x} &
		\lesssim  N_1^{\fr{5-4\al}{6}}  \|u\|_{L^\infty_T H^s} (1+\|u\|_{L^\infty_T L^2_x}) \|v_{N_2}\|_{X^{0,\fr12,1}}.
	\end{align*}
\end{lemm}
\begin{proof}
	Similar with the proof of Proposition \ref{bl0}, by Proposition \ref{prov}, we have
	\begin{align*}
		\langle N_1 \rangle^{2s} &\|u_{N_1}v_{N_2}\|_{L^2_{T,x}}^2=\langle N_1 \rangle^{2s}\sum_{j\in J_{N_1}}\|u_{N_1}v_{N_2}\|_{L^2(I_{j,N_1};L^2)}^2
		\\ &\lesssim  \sum_{j\in J_{N_1}} N_1^{-\al}\langle N_1 \rangle^s (\|P_{N_1}u\|_{L^\infty_T L^2_{x}}^2 +\|P_{N_1}(u^2) \|_{L^{\infty}_{T} L^1_{x}}^2) \|v_{N_2}\|_{X^{0,\fr12,1}}^2\\
		&\lesssim   N_1^{\fr{5-4\al}{3}}\|u\|_{L^\infty_T H^s}^2 (1+\|u\|_{L^2_T L^2_x})^2\|v_{N_2}\|_{X^{0,\fr12,1}}^2,
	\end{align*}
	which leads to the desired result.
\end{proof}

\begin{lemm}\label{le3-1}Let $0<T<1,s\geq 0$, $v\in X^{0,\fr13}_T $and $u \in C([0,T];H^s(\R))$ satisfying
	\eqref{gKdV} on $]0,T[\times \R$  with $L_{\al+1}$ satisfying Hypothesis \ref{Hy1}. Then
	for $N_1 \geq N_2$ with $ N_1\ge 1 $ it holds
	\begin{align}\label{tr1}
		\langle N_1 \rangle^{s}\|u_{N_1}v_{N_2}\|_{L^2_T L^2_x} &
		\lesssim T^{\fr16-} N_1^{\fr{4-2\al}{9}-} N_2^{\fr16+} \|u\|_{L^\infty_T H^s} (1+\|u\|_{L^\infty_T L^2_x}) \|v_{N_2}\|_{X_T^{0,\fr13}}.
	\end{align}
	Moreover when $N_1 \gg N_2>0 $ it holds
	\begin{align}\label{tr2}
		\langle N_1 \rangle^{s}\|u_{N_1}v_{N_2}\|_{L^2_T L^2_x} &
		\lesssim  N_1^{\fr{5-4\al}{9}-} N_2^{\fr16+} \|u\|_{L^\infty_T H^s} (1+\|u\|_{L^\infty_T L^2_x}) \|v_{N_2}\|_{X_T^{0,\fr13}}.
	\end{align}
\end{lemm}
\begin{proof}
	In view of Lemma \ref{Cu2} and the continuous embedding $ X^{0,\frac12+}_T \hookrightarrow L^\infty_T L^2_x $ we have
	\begin{align*}
		\langle N_1 \rangle^{s}\|u_{N_1}v_{N_2}\|_{L^2_T L^2_x} &\lesssim \langle N_1 \rangle^{s}\|u_{N_1}\|_{L^2_T L^\infty_x} \|v_{N_2}\|_{L^\infty_T L^2_x}
		\\ &\lesssim T^{\fr14} N_1^{\fr{2-\al}{3}} \|u\|_{L^\infty_T H^s} (1+\|u\|_{L^\infty_T L^2_x}) \|v_{N_2}\|_{X_T^{0,\fr12+}} \, .
	\end{align*}
	Interpolating with the crude estimate
	\begin{align*}
		\langle N_1 \rangle^{s}\|u_{N_1}v_{N_2}\|_{L^2_T L^2_x} \lesssim \langle N_1 \rangle^{s}\|u_{N_1}\|_{L^\infty_T L^2_x} \|v_{N_2}\|_{L^2_T L^\infty_x}
		\lesssim N_2^{\fr12}\|u\|_{L^\infty_T H^s} \|v_{N_2}\|_{L^2_T L^2_x}.
	\end{align*}
	we obtain \eqref{tr1}. Interpolating the crude estimate with Lemma \ref{Cuv} and the continuous embedding $ X^{0,\frac12+}_T \hookrightarrow X^{0,\frac12,1}_T $, we get \eqref{tr2}.
\end{proof}
\section{A priori estimates on solutions to \eqref{gKdV}}\label{Sect4}
We define the space
\begin{align} \label{defFs}
	F^{s,b}=X^{s-\theta,b+\fr12}+X^{s-\beta,b+\fr16},
\end{align}
endowed with the usual norm, and we define
\begin{align}\label{defYr}
	Y^s=L^\infty_t H^s\cap F^{s,\fr12}=L^\infty_t H^s\cap (X^{s-\theta,1}+X^{s-\beta,\fr23}).
\end{align}
For $ (\xi_1,\xi_2,\xi_3)\in \R^3 $ with $ \xi_1+\xi_2+\xi_3=0 $, let us define the resonance function of order 2 associated with \eqref{gKdV} by
\begin{equation} \label{Omega2}
	\Omega^{(2)}(\xi_1, \xi_2,\xi_3) =  \omega(\xi_1)+\omega(\xi_2)+\omega(\xi_3)
\end{equation}
where $\omega$ is the dispersive symbol. For $\xi_1, \xi_2, \xi_3\in\R$, it will be convenient to define the quantities $|\xi_{max}|\ge |\xi_{med}|\ge |\xi_{min}|$ to be the maximum, median and minimum of $|\xi_1|, |\xi_2|$ and $|\xi_3|$ respectively.

In the sequel we will need  the following result  proved in \cite{MV15}.

\begin{lemm}[\cite{MV15}, Lemma 2.1]\label{res2}
	Let $\alpha>0$. Let
	$ (\xi_1,\xi_2,\xi_3)\in \R^3 $ with $ \xi_1+\xi_2+\xi_3=0 $. Then
	\begin{equation}\label{resonance}
		|\Omega^{(2)}(\xi_1,\xi_2,\xi_3)| \sim |\xi_{\min}| |\xi_{\max}|^\alpha.
	\end{equation}
\end{lemm}
\begin{defi}[Time extension operator] For $0<T<1 $ and $ u\in C([0,T]; {\mathcal D}'(\T)) $ we introduce  the extension operator $ \rho_T $ defined by
	\begin{equation}\label{defrho}
		\rho_T(u)(t):= U_\alpha(t)\eta(t) U_\alpha(-\mu_T(t)) u(\mu_T(t))\; ,
	\end{equation}
	where $ \eta $ is the smooth cut-off function defined in  \eqref{defeta} and $\mu_T $ is the  continuous piecewise affine
	function defined  by
	$$
		\mu_T(t)=\left\{\begin{array}{rcl}
			0  &\text{for } &  t \le 0 \\
			t  &\text {for }& t\in [0,T] \\
			T & \text {for } & t \ge T
		\end{array}
		\right. .
		$$
\end{defi}
\begin{lemm}\label{lem32}
	For any $ s,s_1,s_2\in \R $ and any $1/2<b_1,b_2\le 1 $,
	$\rho_T $ is a continuous linear operator from $  C([0,T]; H^s(\T)) $ into $ C(\R;H^s(\T)) $  and from $ (X^{s_1,b_1}_T+X^{s_2,b_2}_T) $ into $ (X^{s_1,b_1}+X^{s_2,b_2}) $ with a norm that is uniformly bounded in $T \in ]0,1]$.
\end{lemm}
We would like to prove  that a $L^\infty_T L^2_x$-solution $ u $ to \eqref{gKdV} belongs to $ F^{0,\frac12}_T $ under suitable conditions on $\alpha $ and $ (\theta,\beta) $. To do this we proceed in two steps. First, in the next lemma,  we check that such solution $ u $ belongs to $X^{s-(\frac{3}{2}+),1}_T$. Then in Propositon \ref{y0}, with this new information in hands, we are able to prove that 
$u\in Y^0_T $ whenever $ \alpha>\frac{37}{26} $.
\begin{rema}
	Actually it is possible to  relax the condition to $ \alpha>7/5 $ by adding a supplementary step. This step consists in using Lemma \ref{first} below to prove in a second time that $u\in  X^{s-\theta,1}_T+X^{s-\gamma,\frac12+}_T $ for $ \theta>\frac{11-4\alpha}{6} $ and $\gamma>\frac{10-5\alpha}{6}$. Then in a third step we could prove Proposition \ref{y0} below with a relaxing condition on $\alpha$. However, to close the estimates on the differences in the next section we will have a stronger requirement on $\alpha $ that is $\alpha>\frac{55}{38} $. Since 
	anyway $\frac{55}{38} >\frac{37}{26} $ we decide not to present this small improvement here. 
\end{rema}
\begin{lemm}\label{first}
	Let $0<T<1,\alpha\ge 1, s\geq 0$ and $u\in L^\infty_T H^s$ be a solution to \eqref{gKdV} associated with an initial datum $u_0 \in H^s(\R)$. Then $u\in X^{s-(\frac{3}{2}+),1}_T$ and it holds
	\begin{align*}
		\|u\|_{X^{s-(\frac{3}{2}+),1}_T}&\lesssim \|u\|_{L^\infty_T H^s}(1+\|u\|_{L^\infty_T L^2_x}).
	\end{align*}
	
\end{lemm}
\begin{proof} We set $ \tilde{u}=\rho_T(u) $
	and  we look at the following extension $\tilde{\tilde{u}}$ of $u $
	$$
	\tilde{\tilde{u}}(t)= \eta(t) e^{-tL_{\al+1}} u_0 + \eta(t) \int_0^t e^{-(t-t')L_{\al+1}}\partial_x ( \tilde{u}^2)(t') \, dt'
	$$
	Classical  linear estimates in Bourgain's spaces (see for instance \cite{Gi}) lead to
	\begin{align*}
	\|u\|_{X^{s-(\frac{3}{2}+),1}_T}\lesssim \|\tilde{\tilde{u}} \|_{X^{s-(\frac{3}{2}+),1}}
	&\lesssim \|u\|_{L^\infty_T H^{s-(\frac{3}{2}+)}}+ \|\tilde{u}^2\|_{L^\infty_T H^{s+1-(\frac{3}{2}+)}}
	\\& \lesssim  \|u\|_{L^\infty_T H^s}(1+\|u\|_{L^\infty_T L^2_x})\, ,
    \end{align*}
	where we used  that $ L^1(\R) \hookrightarrow H^{-(\frac12+)} $ in the last step.
\end{proof}

\begin{prop}\label{y0}
	Let $0<T<1,\, \alpha\in ]\frac{37}{26},2], \,  s\geq 0 $ and $u\in L^\infty_T H^s$ be a solution to \eqref{gKdV} associated with an initial datum $u_0 \in H^s(\R)$. 
	Assume moreover that  the following conditions are fulfilled :
	\begin{equation}\label{toto1}
		\theta>\frac{11-4\alpha}{6} \text{ and } \beta>\frac{4-2\alpha}{3}\; .
	\end{equation}
	Then $u\in Y^s_T$ and it holds
	\begin{align}\label{eqy0}
		\|u\|_{Y_T^{s}} &\lesssim \|u\|_{L^\infty_T H^s} (1+\|u\|_{L^\infty_T L^2_x})^3.
	\end{align}
\end{prop}

\begin{proof}
	We will work with the extension $\tilde{u}=\rho_T( u)$ of $u$.
	Recall that $\text{supp}\ \tilde{u} \subset [-2,2]\times\R$ and that, according to Lemmas \ref{lem32},
	\begin{align*}
		||\tilde{u}||_{L^\infty_t H^s}\lesssim ||u||_{L^\infty_T H^s}\quad \text{and}\quad ||\tilde{u}||_{X^{\theta,b}}\lesssim ||u||_{X^{\theta,b}_T}
	\end{align*}
	for any $(s,\theta,b)\in \R\times\R\times]1/2,1]$.
	It thus remains to control the $F^{s,\fr12}_T$ -norm of $u$.
	We first notice that linear estimates in Bourgain's spaces (see for instance \cite{Gi}) lead to
	\begin{align*}
		||u||_{F^{s,\fr12}_T}\lesssim ||u_0||_{H^s}+||\partial_x(u^2)||_{F^{s,-\fr12}_T}
	\end{align*}
	and then decompose $u^2$ as
	\begin{align}\label{zq}
		u^2=P_{\lesssim 1}u^2+\sum_{N\gg 1}\Big( P_N(P_{\ll N}u \, u_{\sim N})+\sum_{N'_1\sim N_1\gtrsim N}P_N(u_{N_1}u_{N'_1}) \Big).
	\end{align}
	The contribution of the first term in the above right-hand side is easily controlled by 
	$$
	\| \partial_x P_{\le 1}(u^2) \|_{F^{s,-\frac12}_T} \lesssim \| P_{\le 1} (u^2) \|_{L^\infty_T L^1_x} \lesssim ||u||_{L^\infty_T H^s_x} ||u||_{L^\infty_T L^2_x}
	$$
	 By \eqref{bili2} of Proposition \ref{bl0},  the contribution of the second term in the RHS of \eqref{zq}  is estimated as follows
	
	\begin{align*}
		\Big\| \sum_{N\gg 1}\pa_x P_N(P_{\ll N}u u_{\sim N}) \Big\|_{F^{s,-\fr12}_T} &\lesssim \Big\| \sum_{N\gg 1}\pa_x P_N(P_{\ll N}u \, u_{\sim N}) \Big\|_{X^{s-\theta,0}_T}
		\\&\lesssim \sum_{N\gg 1}\Big\|  \pa_x P_N(P_{\ll N}u \, u_{\sim N}) \Big\|_{X^{s-\theta,0}_T}
		\\&\lesssim \sum_{N\gg 1}N^{-\theta+1}N^{\fr{5-4\al}{6}} \|u\|_{L^\infty_T H^s_x} \|u\|_{L^\infty_T L^2_x} (1+\|u\|_{L^\infty_T L^2_x})^2
		\\&\lesssim \|u\|_{L^\infty_T H^s_x} \|u\|_{L^\infty_T L^2_x} (1+\|u\|_{L^\infty_T L^2_x})^2,
	\end{align*}
	where we use the fact that $\theta>\fr{11-4\al}{6}$.
	To estimate the third term, we take advantage of the $X^{s-\beta,-\fr13}_T$ -part of $F^{s,-\fr12}_T$.
	For $N\gg 1$ and $ N_1\sim N_1'\gtrsim N$, we have
	\begin{align}
		\|\pa_x P_N(u_{N_1}u_{N'_1})\|_{F^{s,-\fr12}_T}&\lesssim \|\pa_x P_N(u_{N_1}u_{N'_1})\|_{X^{s-\beta,-\fr13}_T}\nonumber 
		\\&\lesssim N \Nn^{s-\beta} \|\eta(\frac{\cdot}{2T})P_N(\tilde{u}_{N_1}\tilde{u}_{N'_1})\|_{X^{0,-\fr13}}\nonumber
		\\&\lesssim  N \Nn^{s-\beta} \sup_{\|\varphi\|_{X^{0,\fr13}} \leq 1\atop \text{Supp} \, \varphi\subset [-2T,2T]} \Bigl| \int_{\R}\int_{\R} P_N(\tilde{u}_{N_1}\tilde{u}_{N'_1}) P_N\varphi dxdt \Bigr|\, . \label{zqa}
	\end{align}
	Note that, in view of the resonance relation \eqref{resonance} and on the time support of $\varphi$,
	\begin{align}
		\int_{\R}\int_{\R} P_N(\tilde{u}_{N_1}\tilde{u}_{N'_1}) P_N\varphi dxdt
		&=\int_{-2T}^{2T}\int_{\R} Q_{\gtrsim NN_1^{\al}}\tilde{u}_{N_1}\tilde{u}_{N'_1} P_N\varphi dxdt\nonumber \\
		&\quad+\int_{-2T}^{2T}\int_{\R} Q_{\ll NN_1^{\al}}\tilde{u}_{N_1}Q_{\gtrsim NN_1^{\al}}\tilde{u}_{N'_1} P_N\varphi dxdt
		\nonumber \\
		&\quad\quad+\int_{\R}\int_{\R} Q_{\ll NN_1^{\al}}\tilde{u}_{N_1}Q_{\ll NN_1^{\al}}\tilde{u}_{N'_1} Q_{\sim NN_1^{\al}}P_N\varphi dxdt.
		\label{s4est}
	\end{align}
	To estimate  the first term, we  separate the contributions of $ N_1\sim N_1'\gg N $ and $ N_1\sim N_1'\sim N$. When
	$ N_1\sim N_1'\gg N $, Lemma \ref{first} and \eqref{tr2} of Lemma \ref{le3-1}  lead to
	\begin{align}
		&\Bigl|\int_{-2T}^{2T}\int_{\R} Q_{\gtrsim NN_1^{\al}}\tilde{u}_{N_1}\tilde{u}_{N'_1} P_N\varphi dxdt\Bigr|
		\lesssim \|Q_{\gtrsim NN_1^{\al}}\tilde{u}_{N_1}\|_{L^2_t L^2_x}
		\|\tilde{u}_{N'_1} P_N\varphi\|_{L^2_{2T} L^2_x} \nonumber
		\\& \lesssim  N_1^{\fr{5-4\al}{9}-} N^{\fr16+} N_1^{-s}\|\tilde{u}\|_{L^\infty_t H^s_x} (1+\|\tilde{u}\|_{L^\infty_t L^2_x}) \|P_N\varphi\|_{X^{0,\fr13}}N_1^{\frac{3}{2}+}(NN_1^\al)^{-1}\|\tilde{u}_{N_1}\|_{X^{-(\frac{3}{2}+),1}}\nonumber\\
		& \lesssim  N_1^{\fr{37-26\al}{18}} N^{-\fr56+} N_1^{-s}\|u\|_{L^\infty_T H^s_x} \|u\|_{L^\infty_T L^2_x} (1+\|u\|_{L^\infty_T L^2_x})^2
		\label{zq1} \; .
	\end{align}
	Note that for $ s=0 $ this will be summable over $ N_1\gg N $ as soon as $ \alpha>\frac{37}{26} $. 
	Now,  when $ N_1\sim N_1'\sim N$ we get
	\begin{align}
		&\Bigl|\int_{-2T}^{2T}\int_{\R} Q_{\gtrsim NN_1^{\al}}\tilde{u}_{N_1}\tilde{u}_{N'_1} P_N\varphi dxdt\Bigr|
		\lesssim \|Q_{\gtrsim NN_1^{\al}}\tilde{u}_{N_1}\|_{L^2_t L^2_x}
		\|\tilde{u}_{N'_1} P_N\varphi\|_{L^2_{2T} L^2_x}\nonumber
		\\& \lesssim T^{\fr16} N^{\fr{4-2\al}{9}} N^{\fr16+} N^{-s}\|\tilde{u}\|_{L^\infty_t H^s_x} (1+\|\tilde{u}\|_{L^\infty_t L^2_x}) \|P_N\varphi\|_{X^{0,\fr13}}N^{\frac{3}{2}+}(N^{1+\al})^{-1}\|\tilde{u}_{N_1}\|_{X^{-(\frac{3}{2}+),1}}\nonumber 
		\\& \lesssim T^{\fr16} N^{\fr{10-11\al}{9}+} N^{-s}\|u\|_{L^\infty_T H^s_x}\|u\|_{L^\infty_T L^2_x} (1+\|u\|_{L^\infty_T L^2_x})^2\; .
		\label{zq2}
	\end{align}
	The second term in the RHS of \eqref{s4est} can be treated in the same way as the first one by using that, according to Lemma \ref{le2-1},
	$
	\|Q_{\ll NN_1^{\al}}\tilde{u}_{N_1}\|_{L^p_t L^2_x}\lesssim \|\tilde{u}_{N_1}\|_{L^p_t L^2_x}$ for $1\leq p\leq \infty$.
	To treat the  third term of \eqref{s4est}, we first rewrite
	$Q_{\ll NN_1^{\al}}\tilde{u}_{N_1}$ as
	$$
	Q_{\ll NN_1^{\al}}\tilde{u}_{N_1}=\tilde{u}_{N_1}-Q_{\gtrsim NN_1^{\al}}\tilde{u}_{N_1}\, .
	$$
	In view of the time support of $ \tilde{u} $, the contribution of $\tilde{u}_{N_1}$ in the third term of \eqref{s4est} is estimated  as follows
	\begin{align}
		\Bigl|\int_{\R}\int_{\R}  \tilde{u}_{N_1}Q_{\ll NN_1^{\al}} & \tilde{u}_{N'_1} Q_{\sim NN_1^{\al}}P_N\varphi dxdt\Bigr|\nonumber \\
		&\lesssim \|u_{N_1} \|_{L^2_{[-2,2]} L^\infty_x} \|Q_{\ll NN_1^{\al}}\tilde{u}_{N_1}\|_{L^\infty_t L^2_x}
		\|Q_{\sim NN_1^{\al}}P_N\varphi\|_{L^2_t L^2_x}
		\nonumber \\
		& \lesssim N_1^{\fr{2-\al}{3}} N_1^{-s} \|\tilde{u}\|_{L^\infty_t H^s_x} (1+\|\tilde{u}\|_{L^2_t L^2_x}) \|\tilde{u}\|_{L^\infty_t L^2_x} (NN_1^\al)^{-\fr13} \|P_N\varphi\|_{X^{0,\fr13}} \nonumber \\
		& \lesssim N^{-\fr13} N_1^{\fr{2-2\al}{3}} N_1^{-s} \|u\|_{L^\infty_T H^s_x} \|u\|_{L^\infty_T L^2_x} (1+\|u\|_{L^2_T L^2_x})  \, .
		\label{zq3}
	\end{align}
	Finally we bound the contribution of $Q_{\gtrsim NN_1^{\al}} \tilde{u}_{N_1}$ in the third term of \eqref{s4est} by
	\begin{align}
		\Bigl|\int_{\R}\int_{\R} Q_{\gtrsim NN_1^{\al}}\tilde{u}_{N_1} & Q_{\ll NN_1^{\al}}\tilde{u}_{N'_1} Q_{\sim NN_1^{\al}}P_N\varphi dxdt\Bigr| \nonumber \\
		& \lesssim \| Q_{\gtrsim NN_1^{\al}} \tilde{u}_{N_1} \|_{L^2_t L^2_x}\|Q_{\ll NN_1^{\al}}\tilde{u}_{N_1'}\|_{L^\infty_t L^2_x}
		\|Q_{\sim NN_1^{\al}}P_N\varphi\|_{L^2_t L^\infty_x}
		\nonumber \\
		& \lesssim N_1^{\frac{3}{2}+}(NN_1^\al)^{-1}\|\tilde{u}_{N_1}\|_{X^{-(\frac{3}{2}+),1}} \|\tilde{u}_{N_1}\|_{L^\infty_t L^2_x} N^{\fr12}(NN_1^\al)^{-\fr13} \|P_N\varphi\|_{X^{0,\fr13}} \nonumber \\
		& \lesssim  N ^{-\fr56} N_1^{\frac{9-8\al}{6}+} N_1^{-s} (1+\|u\|_{L^\infty_T L^2_x}) \|u\|_{L^\infty_T H^s_x} \|u\|_{L^\infty_T L^2_x}\, .
		\label{zq4}
	\end{align}
	Gathering \eqref{zqa}-\eqref{zq4}, we obtain that for $s\ge 0 $, $ \alpha>\frac{37}{26}$  and $ N\gg 1 $
	\begin{align*}
		& \sum_{N\gg 1} \sum_{N'_1\sim N_1\gtrsim N}  \|\pa_x P_N(\tilde{u}_{N_1}\tilde{u}_{N'_1})\|_{F^{s,-\fr12}_T} \\
		 &\lesssim  \|u\|_{L^\infty_T H^s_x} \|u\|_{L^\infty_T L^2_x} (1+\|u\|_{L^\infty_T L^2_x})^2
		N^{1-\beta}  \max(N^{\frac{11-13\alpha}{9}+}, N^{\frac{10-11\alpha}{9}+}, N^\frac{1-2\alpha}{3}, N^{\frac{2-4\alpha}{3}+})
		\\&\lesssim  \|u\|_{L^\infty_T H^s_x} \|u\|_{L^\infty_T L^2_x} (1+\|u\|_{L^\infty_T L^2_x})^2
		N^{1-\beta}  N^\frac{1-2\alpha}{3} \lesssim \|u\|_{L^\infty_T H^s_x} \|u\|_{L^\infty_T L^2_x} (1+\|u\|_{L^\infty_T L^2_x})^2,
	\end{align*}
	as soon as $ \alpha\in ]\frac{37}{26},2] $  and $ \beta> \frac{4-2\alpha}{3}$.
	Thus, the proof is completed.
\end{proof}


\begin{prop}\label{h0}
	let $0<T\le 2,  \;  \alpha\in ]\frac{37}{26},2],\;s\geq 0$ and  $\{w_n\} $ be an acceptable weight in the sense of  Definition \ref{def1} with  $ \omega_N \lesssim N^{0+} $ for $ N\ge 1$. Let, moreover,  $u\in L^\infty_T H^s$ be a solution to \eqref{gKdV} associated with an initial datum $u_0 \in H^s(\R)$. Assume a dyadic sequence $\{\omega_N\}$ satisfies  $\omega_N\lesssim \langle N \rangle^{0+}$. Then $u\in L^\infty_T H^s_{\omega}$ and it holds
	\begin{align}\label{eqh0}
		&\|u\|_{L^\infty_T H^s_{\omega}}^2  \lesssim \|u_0\|_{H^s_{\omega}}^2 +T^{0+}\|u\|_{L^\infty_T H^s}(\|u\|_{Y^{s}_T}\|u\|_{L^\infty_T L^2_x}+\|u\|_{Y^{0}_T}\|u\|_{L^\infty_T H^s}) (1+\|u\|_{L^\infty_T L^2_x}),
	\end{align}
	for  any couple 
	\begin{equation}\label{coco}
		(\theta,\beta) \in  ]\frac{11-4\alpha}{6}, \frac{4\alpha-2}{3}[ \times \ ]  \frac{4-2\alpha}{3},\alpha-1[ \;.
	\end{equation}
\end{prop}

\begin{proof}
	First, we notice that Proposition \ref{y0} ensures that $u\in Y^s_T$. Applying the operator $P_N$ with $N>0$ dyadic to \eqref{gKdV} and taking $L^2_x$-scalar product of the resulting equation with $P_N u$, multiplying by $\omega_N^2 \langle N \rangle^{2s}$ and integrating on $]0,t[$ with $0<t<T$, we obtain
	\begin{align*}
		\|P_N u\|_{L^\infty_T H^s_{\omega}}^2 &\lesssim \|P_N u_0\|_{L^\infty_T H^s_{\omega}}^2 + \omega_N^2 \langle N \rangle^{2s} \sup_{t\in]0,T[}\Big|\int_0^t\int_{\R}P_N(u^2)\pa_x P_N u dxdt'\Big|.
	\end{align*}
	Thus it remains to estimate
	\begin{align*}
		\sum_{N>0} \omega_N^2 \langle N \rangle^{2s} \sup_{t\in]0,T[}\Big|\int_0^t\int_{\R}P_N(u^2)\pa_x P_N u dxdt'\Big|.
	\end{align*}
	The contribution of the sum over $N\lesssim 1$ is easily estimated by $ T \|u\|_{L^\infty_T L^2_x}^3$ so that we only have to consider the sum over $N\gg 1$.
	We work with the  extension $\tilde{u}=\rho_T(u) $ of $u$ supported in time in $[-2,2]$ (see Lemma \ref{lem32}). To simplify the notation we drop the tilde in the sequel.
	For $ N\gg 1 $ we have the two following contributions denoted by $A$ and $B$, where
\EQ{
A=& \sum_{N \gg  1}\omega_N^2 \langle  N\rangle^{2s}\int_0^t  \int_{\R^2} \partial_x P_N(u P_{\ll N} u ) P_N u \, dsdxdy\\
 =&  \sum_{N \gg  1}\sum_{0<N_3\ll N}\omega_N^2  \langle  N\rangle^{2s}\Bigl( -\int_0^t \int_{\R^2} \partial_x P_{N_3}u  (P_N u)^2\, dsdxdy\\
 &+\int_0^t \int_{\R^2} [\partial_x P_N, P_{N_3}u] P_{\sim N} u \cdot P_N u  dsdxdy\Bigr) \\
 = &  \sum_{N \gg  1}\sum_{0<N_3\ll N} \omega_N^2  \langle  N\rangle^{2s} N_3 \int_0^t \int_{\R^2} \Lambda_{a_1}(P_{N_3}u, P_{\sim N} u) P_N u  \, dsdxdy \, ,\label{intep}
}
with 
\begin{equation}\label{a1}
a_1(\xi_1,\xi_2)=N_3^{-1}\phi_{N_3}(\xi_1)\tilde \phi_N(\xi_2)[\phi_N(\xi_1+\xi_2)(\xi_1+\xi_2)-\phi_N(\xi_2)\xi_2] \; ,
\end{equation}
 and
\EQ{
B= &  \sum_{N \gg 1}\sum_{N_1\gtrsim N}\omega_N^2  \langle  N\rangle^{2s}\int_0^t  \int_{\R^2} \partial_x P_N( P_{\sim N_1}u P_{N_1} u) P_N u dsdxdy\\
= &  \sum_{N_1 \gg 1}\sum_{1\ll N \lesssim N_1} \om_N^2 \langle  N\rangle^{2s} N \int_0^t  \int_{\R^2}\Lambda_{a_2}(P_N u,P_{\sim N_1}  u) P_{N_1} u  dsdxdy\\
\lesssim & \sum_{N_1 \gg 1}\sum_{1\ll N \lesssim N_1}\om_{N_1}^2  \langle  N_1\rangle^{2s} N \int_0^t  \int_{\R^2}    \Lambda_{a_2}(P_N u, P_{\sim N_1} u) P_{N_1} u \, dsdxdy \label{intep1},
}
with $a_2(\xi,\xi_1)=\frac{\xi}{N}\phi_{N}(\xi)$.  It is direct to check that $\|a_2\|_{L^\infty(\R^2)} \lesssim 1 $ and that $ a_2 $ satisfies \eqref{estcom}. Moreover, it is well-known (see for instance \cite{KT03}) that, for $ N_3\ll N$, the following commutator estimate holds
$$
\|\Lambda_{a_1}(g,h)\|_{L^2_x}=N_3^{-1}  \|[\partial_x P_N, P_{N_3}g] P_{\sim N} h\|_{L^2_x}\lesssim \|g\|_{L^\infty_x} \|h\|_{L^2_x} \;, \quad \forall (g,h)\in L^\infty(\R) \times L^2(\R) ,
$$
and by the mean value theorem, 
$$
\|a_1(\xi_1,\xi_2)\|_{L^\infty(\R^2}\lesssim N_3^{-1} \phi_{N_3}(\xi_1) |\xi_1| \sup_{\R} \Bigl(\phi_N+|\phi'(\frac{\cdot}{N})(\frac{\cdot}{N})| \Bigr)\lesssim 1 \; .
$$
Therefore $a_1 $ also satisfies  \eqref{estcom} and   it thus  suffices to bound
\begin{equation}\label{DefC}
I=\sum_{N \gg  1}\sum_{0<N_3\lesssim N}\om_N^2 \langle  N\rangle^{2s} N_3 \int_0^t \int_{\R^2} \Lambda_a( P_{N_3}u,  P_{\sim N} u) P_N u dsdxdy
\end{equation}
for $ a\in L^\infty(\R^2) $ with $ \|a\|_{L^\infty(\R^2)}\lesssim 1$ satisfying \eqref{estcom}. 
For any $0<t<T$, we set
	\begin{align*}
		I_t(u_1,u_2,u_3)=\int_0^t\int_\R \Lambda_a(u_1,u_2)u_3 dxdt',
	\end{align*}
	so that \eqref{DefC} may be rewritten as 
	\begin{align*}
		I&\lesssim  \sum_{N_1 \gg 1} \sum_{0<N \lesssim N_1} \omega_{N_1}^2  N \sup_{t\in]0,T[}\langle N_1 \rangle^{2s}|I_t(u_N,u_{N_1},u_{\sim N_1})|.
	\end{align*}

	We set $R=N_1^{-(0+)} N N_1^\alpha \ll N N_1^{\al}$ since  $N_1\gg 1$. We split $I_t$ as
	\begin{align*}
		I_t(u_N,u_{N_1},u_{\sim N_1})=I_\infty(1^{\text{high}}_{t,R}u_N,u_{N_1},u_{\sim N_1})+I_\infty(1^{\text{low}}_{t,R}u_N,u_{N_1},u_{\sim N_1})=:I^{\text{high}}_t+I^{\text{low}}_t,
	\end{align*}
	where $I_\infty(u_N,u_{N_1},u_{\sim N_1})=\int_{\R^2}\Lambda_a(u_N,u_{N_1})u_{\sim N_1}dxdt'$. The contribution of $I^{\text{high}}_t$ is estimated, thanks to \eqref{eq2-2}, Lemma \ref{le2-2} and H\"{o}lder's inequality, by
	\begin{align*}
		\langle N_1 \rangle^{2s}|I_t^{\text{high}}| & \lesssim \langle N_1 \rangle^{2s} N^{\fr12}\|1^{\text{high}}_{t,R}\|_{L^1}  \|u_N\|_{L^\infty_t L^2_x} \|u_{N_1}\|_{L^\infty_t L^2_x}\|u_{\sim N_1}\|_{L^\infty_t L^2_x}\\
		&\lesssim   N^{\fr12}  T^{0+}  (N N_1^{\al-})^{-1+} \|u\|_{L^\infty_T L^2_x} \|u\|_{L^\infty_T H^s}^2
		\\
		& \lesssim T^{0+} N_1^{\fr{2-\al}{3}} N_1^{\theta}(N N_1^\al)^{-1+} \|u\|_{L^\infty_T L^2_x} \|u\|_{L^\infty_T H^s}^2,
	\end{align*}
	where in the last step we used that $ \theta \ge \frac{2\alpha-1}{6} $ since $\frac{2\alpha-1}{6} \leq \frac{11\alpha-4}{6} $ for $\alpha\le 2$.\\
	Since $R\ll N N_1^{\al}$,  in view of the resonance relation \eqref{resonance}, we can use the following decomposition to evaluate the contribution $I_t^{\text{low}}$,
	\begin{align*}
		I_\infty(1^{\text{low}}_{t,R} u_N,u_{N_1},u_{\sim N_1})&=I_\infty(1^{\text{low}}_{t,R} Q_{\gtrsim N N_1^\al}u_N,u_{N_1},u_{\sim N_1})
		\\& \quad +I_\infty(1^{\text{low}}_{t,R} Q_{\ll N N_1^\al}u_N,Q_{\gtrsim N N_1^\al}u_{N_1},u_{\sim N_1})
		\\& \quad\quad +I_\infty(1^{\text{low}}_{t,R} Q_{\ll N N_1^\al}u_N,Q_{\ll N N_1^\al}u_{N_1},Q_{\sim N N_1^\al}u_{\sim N_1})
		\\&=: I^{1,\text{low}}_t+I^{2,\text{low}}_t+I^{3,\text{low}}_t.
	\end{align*}
	To get a positive power of $ T $ in the estimates on each of this terms we will make use of Lemma \ref{le2-2} together with the following trivial  Holder inequality in time  
	\begin{equation}\label{triv}
		\|u\|_{L^{2+}_t  L^2_x}\lesssim \|u\|_{L^2_{tx}}^{1-}  \|u\|_{L^\infty_t L^2_{x}}^{0+} 
	\end{equation} 
	Indeed , by \eqref{triv}, Lemma \ref{le2-2}, Lemma \ref{Cu2} and  \eqref{estcom}, we have
	\begin{align*}
		\langle N_1 \rangle^{2s} &| I^{1,\text{low}}_t|
		\lesssim  \langle N_1 \rangle^{2s}\|Q_{\gtrsim N N_1^\al} u_N\|_{L^{2+}_t L^2_x} \|u_{N_1}\|_{L^2_t L^\infty_x}  \|u_{\sim N_1}\|_{L^\infty_t L^2_x} \|1^{\text{low}}_{t,R}\|_{L^{\infty-}} \\
		&\lesssim  T^{0+} N_1^{\fr{2-\al}{3}} \max\Big\{\langle N\rangle^{\theta}(N N_1^\al)^{-1},\langle N\rangle^{\beta}(N N_1^\al)^{-\fr23}\Big\}^{1-} \|u\|_{Y^{0}_T}
		\|u\|_{L^\infty_T H^s}^2 (1+\|u\|_{L^2_T L^2_x}) ,
	\end{align*}
	where we used that $ u = \rho_T(u) $ is supported in time in $[-2,2] $.
	Similarly, by  \eqref{triv}, Lemma \ref{le2-1}, Lemma \ref{le2-2}, Lemma \ref{Cu2} and  \eqref{estcom}, we have
	\begin{align*}
		\langle N_1  &\rangle^{2s}|  I^{2,\text{low}}_t|\lesssim  \langle N_1 \rangle^{2s}
		\|Q_{\gtrsim N N_1^\al} u_{N_1}\|_{L^{2+}_t L^2_x} \|Q_{\ll N N_1^\al} u_{N}\|_{L^\infty_t L^2_x} \| u_{\sim N_1}\|_{L^2_t L^\infty_x} \|1^{\text{low}}_{t,R}\|_{L^{\infty-}} 
		\\&\lesssim T^{0+} N_1^{\fr{2-\al}{3}} \max\Big\{N_1^{\theta}(N N_1^\al)^{-1},N_1^{\beta}(N N_1^\al)^{-\fr23}\Big\}^{1-} \|u\|_{Y^{s}_T}  \|u\|_{L^\infty_T L^2_x} \|u\|_{L^\infty_T H^s} (1+\|u\|_{L^2_T L^2_x}).
	\end{align*}
	To treat the  third term $I^{3,\text{low}}_t$, we rewrite
	$Q_{\ll NN_1^{\al}}u_{N_1}$ as
	$$
	Q_{\ll NN_1^{\al}}u_{N_1}=u_{N_1}-Q_{\gtrsim NN_1^{\al}}u_{N_1}\, .
	$$
	 The contribution of $u_{N_1}$ in the third term is estimated  as above by 
	\begin{align*}
		&\quad \langle N_1 \rangle^{2s}|I_\infty(1^{\text{low}}_{t,R} Q_{\ll N N_1^\al}u_N,u_{N_1},Q_{\sim N N_1^\al}u_{\sim N_1})|
		\\&\lesssim  \langle N_1 \rangle^{2s}
		\|Q_{\sim  N N_1^\al} u_{\sim N_1}\|_{L^{2+}_t L^2_x} \|Q_{\ll N N_1^\al} u_{N}\|_{L^\infty_t L^2_x} \|u_{N_1}\|_{L^2_t L^\infty_x} \|1^{\text{low}}_{t,R}\|_{L^{\infty-}}
		\\&\lesssim T^{0+} N_1^{\fr{2-\al}{3}} \max\Big\{N_1^{\theta}(N N_1^\al)^{-1},N_1^{\beta}(N N_1^\al)^{-\fr23}\Big\}^{1-} \|u\|_{Y^{s}_T} \|u\|_{L^\infty_T L^2_x} \|u\|_{L^\infty_T H^s} (1+\|u\|_{L^2_T L^2_x}).
	\end{align*}
	Finally to bound the contribution of $Q_{\gtrsim NN_1^{\al}}u_{N_1}$ in the third term we separate two cases. In the case $N\leq N_1^{\fr{4-2\al}{3}}$,  Lemmas \ref{le2-1} and \ref{le2-2} leads to 
	\begin{align*}
		&\quad \langle N_1 \rangle^{2s}|I_\infty(1^{\text{low}}_{t,R} Q_{\ll N N_1^\al}u_N,Q_{\gtrsim N N_1^\al}u_{N_1},Q_{\sim N N_1^\al}u_{\sim N_1})|
		\\&\lesssim \langle N_1 \rangle^{2s} N^{\fr12}\max\Big\{\langle N_1\rangle^{\theta}(N N_1^\al)^{-1},\langle N_1\rangle^{\beta}(N N_1^\al)^{-\fr23}\Big\}^{1-} \|u_{N_1}\|_{Y^0}  \\
		 & \qquad\qquad  \times  \|Q_{\ll N N_1^\al}u_N\|_{L^\infty_t L^2_x} \|Q_{\sim N N_1^\al} u_{\sim N_1}\|_{L^\infty_t L^2_x} \|1^{\text{low}}_{t,R}\|_{L^{2-}_t}
		\\&\lesssim T^{\fr 12+}
		N_1^{\fr{2-\al}{3}}
		\max\Big\{N_1^{\theta}(N N_1^\al)^{-1},N_1^{\beta}(N N_1^\al)^{-\fr23}\Big\}^{1-} \|u\|_{Y^{s}_T} \|u\|_{L^\infty_T H^s} \|u\|_{L^\infty_T L^2_x}.
	\end{align*}
	Whereas in the case $N> N_1^{\fr{4-2\al}{3}}$, Lemmas  \ref{le2-1}, \ref{le2-2} and \ref{first} lead to
	\begin{align*}
		&\quad \langle N_1 \rangle^{2s}|I_\infty(1^{\text{low}}_{t,R} Q_{\ll N N_1^\al}u_N,Q_{\gtrsim N N_1^\al}u_{N_1},Q_{\sim N N_1^\al}u_{\sim N_1})|
		\\&\lesssim \langle N_1 \rangle^{2s} N^{\fr12}\max\Big\{\langle N_1\rangle^{\theta}(N N_1^\al)^{-1},\langle N_1\rangle^{\beta}(N N_1^\al)^{-\fr23}\Big\}^{1-} \|u_{\sim N_1}\|_{Y^0} 
		\\
		 & \qquad\qquad  \times \|Q_{\ll N N_1^\al}u_N\|_{L^\infty_t L^2_x} \|Q_{\gtrsim N N_1^\al} u_{N_1}\|_{L^2_t L^2_x} \|1^{\text{low}}_{t,R}\|_{L^{\infty-}}
		\\&\lesssim T^{0+} \langle N_1 \rangle^{2s}
		N^{\fr12} N_1^{\fr32+}(N N_1^\al)^{-1}
		\max\Big\{N_1^{\theta}(N N_1^\al)^{-1},N_1^{\beta}(N N_1^\al)^{-\fr23}\Big\}^{1-} \\
		& \qquad\qquad  \times \|u_{\sim N_1}\|_{Y^{0}}
		\|u_N\|_{L^\infty_t L^2_x} \|u_{N_1}\|_{X^{-(\fr32+),1}}
		\\&\lesssim T^{0+}
		N_1^{\fr{\al-2}{3}} N_1^{\fr32+} N_1^{-\al}
		\max\Big\{N_1^{\theta}(N N_1^\al)^{-1},N_1^{\beta}(N N_1^\al)^{-\fr23}\Big\}^{1-} 
		\\& \qquad\qquad  \times \|u_{\sim N_1}\|_{Y^{s}}
		\|u_N\|_{L^\infty_t L^2_x} \|u_{N_1}\|_{X^{s-(\fr32+),1}}
		\\&\lesssim
		T^{0+}N_1^{\fr{2-\al}{3}}
		\max\Big\{N_1^{\theta}(N N_1^\al)^{-1},N_1^{\beta}(N N_1^\al)^{-\fr23}\Big\}^{1-} 
		\\& \qquad\qquad  \times \|u\|_{Y^{s}_T}
		\|u\|_{L^\infty_T L^2_x} \|u\|_{L^\infty_T H^s} (1+\|u\|_{L^\infty_T L^2_x}),
	\end{align*}
	where we use that $\fr32-\al<\fr23 (2-\al)$ since $\al>\fr12$. Therefore, for $ N_1\gg 1 $ it holds
	\begin{align*}
		\quad \langle N_1 \rangle^{2s} & |I_t(u_N,u_{N_1},u_{\sim N_1})|
		 \lesssim T^{0+} N_1^{\fr{2-\al}{3}} \max\Big\{N_1^{\theta}(N N_1^\al)^{-1},N_1^{\beta}(N N_1^\al)^{-\fr23}\Big\}^{1-} \\
		 & \hspace*{32mm}  \times
		\|u\|_{L^\infty_T H^s}(\|u\|_{Y^{s}_T}\|u\|_{L^\infty_T L^2_x}+\|u\|_{Y^{0}_T}\|u\|_{L^\infty_T H^s}) (1+\|u\|_{L^\infty_T L^2_x})
	\end{align*}
	and we have the following estimate on the contribution of the sum over $N_1\gg 1$ :
	\begin{align*}
		&\quad\sum_{N_1\gg 1} \sum_{0< N \lesssim N_1} \omega_{N_1}^2 N \sup_{t\in]0,T[} \langle N_1 \rangle^{2s}|I_t(u_N,u_{N_1},u_{\sim N_1})|
		\\ &\lesssim  T^{0+} \sum_{N_1\gg 1} \sum_{0<N \lesssim N_1} \omega_{N_1}^2 N N_1^{\fr{2-\al}{3}} \max\Big\{N_1^{\theta}(N N_1^{\al})^{-1},
		N_1^{\beta}(N N_1^{\al})^{-\fr23}\Big\}^{1-}
		\\ & \qquad\qquad \times \|u\|_{L^\infty_T H^s}(\|u\|_{Y^{s}_T}\|u\|_{L^\infty_T L^2_x}+\|u\|_{Y^{0}_T}\|u\|_{L^\infty_T H^s}) (1+\|u\|_{L^\infty_T L^2_x})
		\\&\lesssim T^{0+}\sum_{N_1\gg 1}   \max\Big\{N_1^{\theta-\fr{4\al-2}{3}+},\
		N_1^{\beta-\al+1+} \Big\}
		\\& \qquad\qquad \times\|u\|_{L^\infty_T H^s}(\|u\|_{Y^{s}_T}\|u\|_{L^\infty_T L^2_x}+\|u\|_{Y^{0}_T}\|u\|_{L^\infty_T H^s}) (1+\|u\|_{L^\infty_T L^2_x})
		\\& \lesssim  T^{0+}\|u\|_{L^\infty_T H^s}(\|u\|_{Y^{s}_T}\|u\|_{L^\infty_T L^2_x}+\|u\|_{Y^{0}_T}\|u\|_{L^\infty_T H^s}) (1+\|u\|_{L^\infty_T L^2_x}),
	\end{align*}
	where we use  that, $\omega_{N_1} \lesssim N_1^{0+}$, $\theta<\fr{4\al-2}{3}$ and $\beta<\al-1$.
	
	Thus, we have
	\begin{align*}
		&\|u\|_{L^\infty_T H^s_{\omega}}^2  \lesssim \|u_0\|_{L^\infty_T H^s_{\omega}}^2 + T^{0+}\|u\|_{L^\infty_T H^s}(\|u\|_{Y^{s}_T}\|u\|_{L^\infty_T L^2_x}+\|u\|_{Y^{0}_T}\|u\|_{L^\infty_T H^s}) (1+\|u\|_{L^\infty_T L^2_x}),
	\end{align*}
	which completes the proof.
\end{proof}

\section{Estimates on the difference of two solutions to \eqref{gKdV}}
\label{Sect5}
We define the Banach space
\begin{align*}
\bar{H}^s(\R)=\{\varphi\in H^s(\R)\ |\ \|\varphi\|_{\bar{H}^s}<\infty\}
\end{align*}
with
\begin{align*}
\|\varphi\|_{\bar{H}^s}:=\|\langle|\xi|^{-1}\rangle \langle\xi\rangle^{s}\hat{\varphi}\|_{L^2},
\end{align*}
equipped with the norm $\|\cdot\|_{\bar{H}^s}$.
Finally, we define the function spaces $Z^s$, $s\in\R$, by
\begin{align*}
Z^s=L^\infty_t \bar{H}^s\cap F^{s,\fr12}.
\end{align*}

\begin{prop}\label{yr}
Let $0<T\le 2,\ \frac{37}{26}<\al\leq 2,\ r<0$, and $u,v\in L^\infty_T L^2$ be a solution to \eqref{gKdV} associated with an initial datum $u_0,v_0 \in L^2(\R)$ such that $u_0-v_0 \in \bar{L}^2(\R)$. Then, for any $\theta>\frac{11-4\alpha}{6} $ and $\beta> \frac{4-2\alpha}{3} $, $u-v\in Z^r_T$ and it holds
\begin{align} \label{estpropyr}
\|u-v\|_{Z_T^{r}} &\lesssim \|u-v\|_{L^\infty_T \bar{H}^r_x}+T^{0+}\|u-v\|_{Z^{r}_T} \|u+v\|_{Y^0_T}(1+\|u+v\|_{Y^0_T}^2).
\end{align}
provided
\begin{equation}\label{condprop51}
\begin{cases}
\theta-r<\al \\
\beta-r<\fr23\al \\
\theta-\beta<\frac{22\alpha-11}{18} \\
r>\frac{2-2\alpha}{3},
\end{cases}
\end{equation}
\end{prop}

\begin{proof}

Note that $w=u-v$ satisfies
\begin{equation}\label{diff}
\pa_t w+L_{\al+1} w+\fr12 \pa_x(w z)=0,
\end{equation}
with $z=u+v$. By Proposition \ref{y0},  $u,\ v \in Y^0_T$ for any  $\theta>\frac{11-4\alpha}{6} $ and $\beta>  \frac{4-2\alpha}{3}  $. Since $u_0-v_0 \in \bar{L}^2(\R)$, we claim that the difference $w=u-v$ belongs to $Z^0_T$. Indeed, according to the definition of $Z^s$, it suffices to check that $P_{\leq 1} w$ belongs to $L^\infty_T \bar{H}^0$. By the Duhamel formula, for any dyadic integer $0<N\leq 1$, we have
\begin{align*}
\|P_N w\|_{L^\infty_T \bar{H}^0} & \lesssim \|P_N(u_0-v_0)\|_{\bar{H}^0}+N^{\fr12}\|w\|_{L^\infty_T L^2}\|z\|_{L^\infty_T L^2} \\
& \lesssim \|P_N(u_0-v_0)\|_{\bar{H}^0}+N^{\fr12}(\|u\|_{L^\infty_T L^2}^2+\|v\|_{L^\infty_T L^2}^2)
,
\end{align*}
which implies that $P_{\leq 1} w \in L^\infty_T \bar{H}^0$. Thus, $w=u-v \in Z^0_T \hookrightarrow Z^r_T$.

We extend $w$ from $(0,T)$ to $\R$ by using the extension operator $\rho_T$. On account of the uniform bounds on $\rho_T$, it remains to estimate the $F^{r,\fr12}_T$ -norm of $w$. From classical linear estimate in the framework of Bourgain's spaces (see \cite{Gi}), the Duhamel formulation associated with \eqref{diff} leads to
\begin{align*}
\|w\|_{F^{r,\fr12}_T} \lesssim \|w_0\|_{H^{r}}+\|\pa_x(zw)\|_{F^{r,-\fr12}_T}.
\end{align*}
 It thus suffices to estimate
\begin{align*}
\|\pa_x(zw)\|_{F^{r,-\fr12}_T} \lesssim \Big( \sum_{N>0} \|P_N \pa_x(zw)\|_{F^{r,-\fr12}_T}^2\Big)^{\fr12}.
\end{align*}

We first estimate the low-high contribution $P_N(P_{\ll N}z P_{\sim N}w)$. The contribution of the sum over $N\lesssim 1$ is easily estimated by
\begin{align*}
\|\pa_x P_N(P_{\ll N}z P_{\sim N}w)\|_{F^{r,-\fr12}_T} &\lesssim  N\|P_N(P_{\ll N}z P_{\sim N}w)\|_{L^2_T H^{r-\theta}}
\lesssim  T^{\frac12} N^{0+} \|w\|_{L^\infty_{T} H^r_{x}}\|z\|_{L^\infty_{T} L^2_{x}}.
\end{align*}
It remains to estimate the sum over $N\gg 1$.  We notice that Bernstein inequalities lead,  for any $0<\epsilon\ll 1 $, to
$$
\|P_N(P_{\ll N}z P_{\sim N}w)\|_{L^2_T L^2_x} \lesssim T^{\fr\epsilon2} N^{\frac\epsilon2} \|P_N(P_{\ll N}z P_{\sim N}w)\|_{L^2_T L^2_x}^{1-\ep}
(\|P_{\ll N}z\|_{L^\infty_T L^2_x} \|P_{\sim N}w\|_{L^\infty_T L^2_x})^{\ep}
$$
Applying \eqref{Est3} of Proposition  \ref{bl2}, we then obtain
\begin{align*}
\|\pa_x P_N(P_{\ll N} & z P_{\sim N}w)\|_{F^{r,-\fr12}_T} \lesssim  N\|P_N(P_{\ll N}z P_{\sim N}w)\|_{L^2_T H^{r-\theta}}
\\&\lesssim  T^{0+}N^{1+} N^{r-\theta} N^{\fr{5-4\al}{6}}N^{-r} (\|P_{N}w\|_{L^\infty_{T} H^r_{x}}+\|w\|_{Z^{r}_T}
\|z\|_{L^\infty_T L^2_{x}}^{1-}+\|w\|_{L^\infty_T \bar{H}^r_{x}}\|z\|_{Y^{0}_T}^{1-})\\
&\quad\times(\|z\|_{L^\infty_T L^2_{x}}+\|z\|_{L^\infty_T L^2_{x}}^{2-})
\\&\lesssim   T^{0+} N^{\fr{11-4\al}{6}-\theta+} (\|w\|_{L^\infty_{T} H^r_{x}}+\|w\|_{Z^{r}_T}
\|z\|_{L^\infty_T L^2_{x}}^{1-}+\|w\|_{L^\infty_T \bar{H}^r_{x}}\|z\|_{Y^{0}_T}^{1-})\\
&\quad\times(\|z\|_{L^\infty_T L^2_{x}}+\|z\|_{L^\infty_T L^2_{x}}^{2-})
\\&\lesssim T^{0+}\|w\|_{Z^{r}_T} \|z\|_{Y^0_T}(1+\|z\|_{Y^0_T}^2).
\end{align*}
that is acceptable since $\theta>\fr{11-4\al}{6}$. 
Similarly, we can estimate the high-low contribution $P_N(P_{\sim N}z P_{\ll N}w)$.
The contribution of the sum over $N\lesssim 1$ is easily estimated by
\begin{align*}
\|\pa_x P_N(P_{\sim N}z P_{\ll N}w)\|_{F^{r,-\fr12}_T} &\lesssim  N\|P_N(P_{\sim N}z P_{\ll N}w)\|_{L^2_T H^{r-\theta}}
\lesssim T^{\frac12}N^{0+} \|w\|_{L^\infty_{T} H^r_{x}}\|z\|_{L^\infty_{T} L^2_{x}}.
\end{align*}
It remains to estimate the sum over $N\gg 1$.
Applying \eqref{Est2}  of Proposition \ref{bl2}, we obtain in the same way as above,
\begin{align*}
\|\pa_x & P_N(P_{\sim N}z P_{\ll N}w)\|_{F^{r,-\fr12}_T} \lesssim  N\|P_N(P_{\sim N}z P_{\ll N }w)\|_{L^2_T H^{r-\theta}}
\\&\lesssim T^{0+} N^{1+} N^{r-\theta} N^{\fr{5-4\al}{6}} N^{-r} (\|z\|_{L^\infty_T L^2_{x}}+\|z\|_{L^\infty_T L^2_{x}}^{2-})\\
&\qquad\qquad\times  (\|w\|_{L^\infty_{T} H^r_{x}}+\|w\|_{Z^{r}_T}
\|z\|_{L^\infty_T L^2_{x}}^{1-}+\|w\|_{L^\infty_T \bar{H}^r_{x}}\|z\|_{Y^{0}_T}^{1-})
\\&\lesssim    T^{0+} N^{\fr{11-4\al}{6}-\theta+} \|w\|_{Z^{r}_T} \|z\|_{Y^0_T}(1+\|z\|_{Y^0_T}^2).
\end{align*}
That is again acceptable since $r<0$ and $\theta>\fr{11-4\al}{6}$.  Now we deal with the high-high interactions term:
\begin{align*}
\|\pa_x P_N(P_{\gtrsim N}z P_{\gtrsim N}w)\|_{F^{r,-\fr12}_T} &\lesssim \sum_{N'_1\sim N_1\gtrsim N} N\| P_N(P_{N_1}z P_{N'_1}w)\|_{F^{r,-\fr12}_T}.
\end{align*}
We separate the contributions of $N\lesssim N'_1\sim N_1 \lesssim 1$, $N \lesssim 1 \ll N'_1\sim N_1$ and $1 \ll N \lesssim N'_1\sim N_1$.

In the case $N\lesssim N'_1\sim N_1 \lesssim 1$, it is easily estimated by
\begin{align*}
\sum_{1\gtrsim N'_1\sim N_1\gtrsim N} N\| P_N(P_{N_1}z P_{N'_1}w)\|_{F^{r,-\fr12}_T} &\lesssim  \sum_{1\gtrsim N'_1\sim N_1\gtrsim N} N\| P_N(P_{N_1}z P_{N'_1}w)\|_{L^2_T H^{r-\theta}} \\
 & \lesssim T^{\frac12} \|w\|_{L^\infty_{T} H^r_{x}}\|z\|_{L^\infty_{T} L^2_{x}}.
\end{align*}

In the case $N \lesssim 1 \ll N'_1\sim N_1$, we work with the extensions $ \tilde{w}=\rho_T(w) $ and $ \tilde{z}=\rho_T(z) $ of $w $ and $ z$ on $ \R $. It holds 
\begin{align*}
\sum_{N'_1\sim N_1 \gg 1 \gtrsim N} N\| P_N &(P_{N_1}z P_{N'_1}w)\|_{F^{r,-\fr12}_T} \lesssim T^{0+} \sum_{N'_1\sim N_1 \gg 1 \gtrsim N} N\|P_N(P_{N_1}\tilde{z} P_{N'_1}\tilde{w})\|_{X^{r-\beta,-\fr13+}}
\\&\lesssim  T^{0+} \sum_{N'_1\sim N_1 \gg 1 \gtrsim N} N\|\sum_{L,L_1,L_2} Q_L P_N(Q_{L_1} P_{N_1}\tilde{z} \  Q_{L_2}P_{N'_1}\tilde{w})\|_{X^{0,-\fr13+}}.
\end{align*}
Note that, in the above estimates,  we gain a factor $ T^{0+} $ by taking in the right-hand side member the norm $X^{r-\beta,-\fr13+}$ instead of $X^{r-\beta,-\fr13}$.
By the resonance relation \eqref{resonance}, this contribution can be separated in three parts.
The contribution of the sum over $L\gtrsim N N_1^{\al}$ can be controlled by
\begin{align*}
&T^{0+}\sum_{N'_1\sim N_1 \gg 1 \gtrsim N} N\|Q_{\gtrsim N N_1^{\al}} P_N(P_{N_1}\tilde{z} \  P_{N'_1}\tilde{w})\|_{X^{0,-\fr13+}}
\\&\lesssim T^{0+}\sum_{N'_1\sim N_1 \gg 1 \gtrsim N} N (N N_1^{\al})^{-\fr13+} \|P_N(P_{N_1}\tilde{z} \  P_{N'_1}\tilde{w})\|_{L^2_t L^2_x}
\\&\lesssim T^{0+}\sum_{N'_1\sim N_1 \gg 1 \gtrsim N} N (N N_1^{\al})^{-\fr13+} N^{\fr12} N_1^{-r} \|P_{N'_1}\tilde{w}\|_{L^\infty_{t} H^r_{x}}\|P_{N_1}\tilde{z}\|_{L^\infty_{t} L^2_{x}}
\\&\lesssim  T^{0+} \|w\|_{L^\infty_{T} H^r_{x}}\|z\|_{L^\infty_{T} L^2_{x}},
\end{align*}
where we use  that $ \frac{2-2\alpha}{3}\ge -\alpha/3 $ for $ \alpha\le 2 $ and thus the condition  $ r>\frac{2-2\alpha}{3} $ forces  $r>-\al/3$. The contribution of the sum over $L\ll N N_1^{\al}$ and $L_1\gtrsim N N_1^{\al}$ can be controlled by
\begin{align*}
& T^{0+} \sum_{N'_1\sim N_1 \gg 1 \gtrsim N} N\|Q_{\ll N N_1^{\al}} P_N(Q_{\gtrsim N N_1^{\al}} P_{N_1}\tilde{z} \  P_{N'_1}\tilde{w})\|_{X^{0,-\fr13+}}
\\&\lesssim  T^{0+} \sum_{N'_1\sim N_1 \gg 1 \gtrsim N} N \|P_N(Q_{\gtrsim N N_1^{\al}} P_{N_1}\tilde{z} \  P_{N'_1}\tilde{w})\|_{L^2_t L^2_x}
\\&\lesssim  T^{0+} \sum_{N'_1\sim N_1 \gg 1 \gtrsim N} N N^{\fr12} N_1^{-r} \|P_{N'_1}\tilde{w}\|_{L^\infty_{T} H^r_{x}} \max\Big\{N_1^{\theta}(NN_1^\al)^{-1},
N_1^{\beta}(NN_1^\al)^{-\fr23}\Big\} \|P_{N_1}\tilde{z}\|_{Y^0}
\\&\lesssim   T^{0+} \|w\|_{L^\infty_{T} H^r_{x}}\|z\|_{Y^0_{T}},
\end{align*}
where we use the fact that $\theta-\al<r$ and $\beta-\fr23\al<r$.
Finally the contribution of the last region can be bounded by
\begin{align*}
& T^{0+} \sum_{N'_1\sim N_1 \gg 1 \gtrsim N} N\|Q_{\ll N N_1^{\al}} P_N(Q_{\ll N N_1^{\al}} P_{N_1}\tilde{z} \  Q_{\sim N N_1^{\al}} P_{N'_1}\tilde{w})\|_{X^{0,-\fr13+}}
\\&\lesssim  T^{0+}  \sum_{N'_1\sim N_1 \gg 1 \gtrsim N} N \|P_N(Q_{\ll N N_1^{\al}} P_{N_1}\tilde{z} \  Q_{\sim N N_1^{\al}} P_{N'_1}\tilde{w})\|_{L^2_t L^2_x}
\\&\lesssim  T^{0+} \sum_{N'_1\sim N_1 \gg 1 \gtrsim N} N N^{\fr12} N_1^{-r} \|P_{N_1}\tilde{z}\|_{L^\infty_{t} L^2_{x}} \max\Big\{N_1^{\theta}(NN_1^\al)^{-1},
N_1^{\beta}(NN_1^\al)^{-\fr23}\Big\} \|P_{N'_1}\tilde{w}\|_{Y^r}
\\&\lesssim   T^{0+}  \|w\|_{Y^r_{T}}\|z\|_{L^\infty_{T} L^2_{x}},
\end{align*}
where we use the fact $\theta-\al<r$ and $\beta-\fr23\al<r$ again.

In the case $1 \ll N \lesssim N'_1\sim N_1$, we have
\begin{align*}
\|\pa_x P_N &(P_{\gtrsim N}z P_{\gtrsim N}w)\|_{F^{r,-\fr12}_T} \lesssim   T^{0+} \sum_{N'_1\sim N_1\gtrsim N}N \Nn^{r-\beta} \|\eta(\frac{\cdot}{2T})P_N(\tilde{z}_{N_1}\tilde{w}_{N'_1})\|_{X^{0,-\fr13+}}
\\&=  T^{0+} \sum_{N'_1\sim N_1\gtrsim N} N \Nn^{r-\beta}
\sup_{\|\varphi\|_{X^{0,\fr13-}} \leq 1\atop \text{Supp} \, \varphi\subset [-2T,2T]} \Bigl| \int_{\R}\int_{\R} P_N(\tilde{z}_{N_1}\tilde{w}_{N'_1}) P_N\varphi dxdt \Bigr|\, .
\end{align*}
Note that, in view of the resonance relation \eqref{resonance} and on the time support of $\varphi$,
\begin{align}
\int_{\R}\int_{\R} P_N(\tilde{z}_{N_1}\tilde{w}_{N'_1}) P_N\varphi dxdt
&=\int_{-2T}^{2T}\int_{\R} Q_{\gtrsim NN_1^{\al}}\tilde{z}_{N_1}\tilde{w}_{N'_1} P_N\varphi dxdt\nonumber \\
&\quad+\int_{-2T}^{2T}\int_{\R} Q_{\ll NN_1^{\al}}\tilde{z}_{N_1}Q_{\gtrsim NN_1^{\al}}\tilde{w}_{N'_1} P_N\varphi dxdt
\nonumber \\
&\qquad+\int_{\R}\int_{\R} Q_{\ll NN_1^{\al}}\tilde{z}_{N_1}Q_{\ll NN_1^{\al}}\tilde{w}_{N'_1} Q_{\sim NN_1^{\al}}P_N\varphi dxdt.\label{s5est}
\end{align}
To estimate  the first term, we  separate the contributions of $ N_1\sim N_1'\gg N $ and $ N_1\sim N_1'\sim N$. When
$ N_1\sim N_1'\gg N $, \eqref{tr4} of Lemma \ref{le3-2}  lead, for $\|\varphi\|_{X^{0,\fr13-}} \leq 1$ with $ \text{Supp} \, \varphi\subset [-2T,2T]$, to
\begin{align}
N^{1+r-\beta}&\Bigl|\int_{-2T}^{2T}\int_{\R} Q_{\gtrsim NN_1^{\al}}\tilde{z}_{N_1}\tilde{w}_{N'_1} P_N\varphi dxdt\Bigr|
\lesssim N^{1+r-\beta}\|Q_{\gtrsim NN_1^{\al}}\tilde{z}_{N_1}\|_{L^2_t L^2_x}
\|\tilde{w}_{N'_1} P_N\varphi\|_{L^2_{2T} L^2_x}\nonumber
\\& \lesssim N^{1+r-\beta} N_1^{\fr{5-4\al}{9}-} N^{\fr16+} N_1^{-r} \|\tilde{w}\|_{Y^{r}} (1+\|\tilde{z}\|_{Y^0}) \|P_N\varphi\|_{X^{0,\fr13-}} 
\nonumber\\& \qquad\qquad \times \max\Big\{N_1^{\theta}(NN_1^\al)^{-1},
N_1^{\beta}(NN_1^\al)^{-\fr23}\Big\} \|\tilde{z}_{N_1}\|_{Y^0}\nonumber 
\\& \lesssim  N^{\frac76+r-\beta}  N_1^{-r+\fr{5-4\al}{9}+} \|w\|_{Y^{r}_T} (1+\|z\|_{Y^0_T}) \max\Big\{N_1^{\theta}(NN_1^\al)^{-1},
N_1^{\beta}(NN_1^\al)^{-\fr23}\Big\} \|z_{N_1}\|_{Y_T^{0}}, \label{tutu1}
\end{align}
We notice that the sum over $ N_1\gg N $ is finite  for $\theta-r-\alpha<0 $ and $ \beta-r -\fr23 \alpha<0 $ since $\frac{5-4\al}{9}<0 $ for $ \alpha>\frac54 $. Then the sum over $ N\gg 1 $ is finite  whenever
\begin{equation}\label{tutu2}
\theta-\beta<\frac{26\alpha-13}{18} \quad \text{and} \quad \alpha>\frac{19}{20} \; .
\end{equation}
Now when $ N_1\sim N_1'\sim N$,  by \eqref{tr3} of Lemma \ref{le3-2}, we get
\begin{align*}
&N^{1+r-\beta}\Bigl|\int_{-2T}^{2T}\int_{\R} Q_{\gtrsim NN_1^{\al}}\tilde{z}_{N_1}\tilde{w}_{N'_1} P_N\varphi dxdt\Bigr|
\lesssim N^{1+r-\beta}\|Q_{\gtrsim NN_1^{\al}}\tilde{z}_{N_1}\|_{L^2_t L^2_x}
\|\tilde{w}_{N'_1} P_N\varphi\|_{L^2_{2T} L^2_x}
\\& \lesssim N^{1+r-\beta} N^{\fr{4-2\al}{9}-} N^{\fr16+} N_1^{-r} \|\tilde{w}\|_{Y^{r}} (1+\|\tilde{z}\|_{Y^0}) \|P_N\varphi\|_{X^{0,\fr13-}}
\\&\qquad\qquad \times \max\Big\{N_1^{\theta}(NN_1^\al)^{-1},
N_1^{\beta}(NN_1^\al)^{-\fr23}\Big\} \|\tilde{z}_{N_1}\|_{Y^0}
\\& \lesssim N^{-\beta+\fr{29-4\al}{18}+}  \|w\|_{Y^{r}_T} (1+\|z\|_{Y^0_T}) \max\Big\{N^{\theta-\alpha-1},
N^{\beta-\fr23 \alpha-\fr23}\Big\} \|z_{N_1}\|_{Y_T^{0}}
\end{align*}
that can be summed over $ N\gg1 $ whenever 
\begin{equation}\label{tutu22}
\theta-\beta<\frac{22\alpha-11}{18} \quad \text{and} \quad \alpha>\frac{17}{16} \; .
\end{equation}
The second term can be treated in the same way as the first one. When
$ N_1\sim N_1'\gg N $, Lemma \ref{le2-1} and \eqref{tr2} of Lemma \ref{le3-1}  lead,  for $\|\varphi\|_{X^{0,\fr13-}} \leq 1$ with $ \text{Supp} \, \varphi\subset [-2T,2T]$,  to
\begin{align*}
&N^{1+r-\beta}\Bigl|\int_{-2T}^{2T}\int_{\R} Q_{\ll NN_1^{\al}}\tilde{z}_{N_1}Q_{\gtrsim NN_1^{\al}}\tilde{w}_{N'_1} P_N\varphi dxdt\Bigr|
\lesssim N^{1+r-\beta} \|Q_{\gtrsim NN_1^{\al}}\tilde{w}_{N'_1}\|_{L^2_t L^2_x}
\|Q_{\ll NN_1^{\al}}\tilde{z}_{N_1} P_N\varphi\|_{L^2_{2T} L^2_x}
\\& \lesssim N^{1+r-\beta} N_1^{\fr{5-4\al}{9}-} N^{\fr16+}  \|\tilde{z}\|_{L^\infty_t L^2_x} (1+\|\tilde{z}\|_{L^\infty_t L^2_x}) \|P_N\varphi\|_{X^{0,\fr13-}} \max\Big\{N_1^{\theta}(NN_1^\al)^{-1},
N_1^{\beta}(NN_1^\al)^{-\fr23}\Big\} \|\tilde{w}_{N_1}\|_{Y^0}
\\& \lesssim N^{\frac76+r-\beta} N_1^{-r+\fr{5-4\al}{9}+}  \|z\|_{L^\infty_T L^2_x} (1+\|z\|_{L^2_T L^2_x}) \max\Big\{N_1^{\theta}(NN_1^\al)^{-1},
N_1^{\beta}(NN_1^\al)^{-\fr23}\Big\} \|w_{N_1}\|_{Y_T^{r}},
\end{align*}
whereas when $ N_1\sim N_1'\sim N$,  by Lemma \ref{le2-1} and \eqref{tr1} of Lemma \ref{le3-1}, we get
\begin{align*}
&N^{1+r-\beta}\Bigl|\int_{-2T}^{2T}\int_{\R} Q_{\ll NN_1^{\al}}\tilde{z}_{N_1}Q_{\gtrsim NN_1^{\al}}\tilde{w}_{N'_1} P_N\varphi dxdt\Bigr|
\\& \lesssim N^{1+r-\beta} \|Q_{\gtrsim NN_1^{\al}}\tilde{w}_{N'_1}\|_{L^2_t L^2_x}
\|Q_{\ll NN_1^{\al}}\tilde{z}_{N_1} P_N\varphi\|_{L^2_{2T} L^2_x}
\\& \lesssim N^{1+r-\beta} N^{\fr{4-2\al}{9}-} N^{\fr16+} N_1^{-r} \|z\|_{L^\infty_T L^2_x} (1+\|z\|_{L^2_T L^2_x}) 
 \|P_N\varphi\|_{X^{0,\fr13-}} \\
 & \qquad\qquad \times  \max\Big\{N_1^{\theta}(NN_1^\al)^{-1},
N_1^{\beta}(NN_1^\al)^{-\fr23}\Big\} \|w_{N_1}\|_{Y_T^{r}}
\\& \lesssim   N^{-\beta+\frac{29-4\alpha}{18}}  \max\Big\{N^{\theta-\alpha-1},
N^{\beta-\fr23 \alpha-\fr23}\Big\}  \|z\|_{L^\infty_T L^2_x} (1+\|z\|_{L^2_T L^2_x}) \|w_{N}\|_{Y_T^{r}}.
\end{align*}
To treat the  third term of \eqref{s5est}, we first rewrite
$Q_{\ll NN_1^{\al}}\tilde{z}_{N_1}$ as
$$
Q_{\ll NN_1^{\al}}\tilde{z}_{N_1}=\tilde{z}_{N_1}-Q_{\gtrsim NN_1^{\al}}\tilde{z}_{N_1}\, .
$$
In view of the time support of $ \tilde{w} $, the contribution of $\tilde{w}_{N_1}$ in the third term of \eqref{s5est} is estimated  as follows
\begin{align*}
N^{1+r-\beta}\Bigl|\int_{\R}\int_{\R} \tilde{z}_{N_1}&Q_{\ll NN_1^{\al}}\tilde{w}_{N'_1} Q_{\sim NN_1^{\al}}P_N\varphi dxdt\Bigr| \\
& \lesssim N^{1+r-\beta}\|\tilde{z}_{N_1} \|_{L^2_{[-2,2]} L^\infty_x} \|Q_{\ll NN_1^{\al}}\tilde{w}_{N'_1}\|_{L^\infty_t L^2_x}
\|Q_{\sim NN_1^{\al}}P_N\varphi\|_{L^2_t L^2_x}
\\
& \lesssim N^{1+r-\beta} N_1^{\fr{2-\al}{3}}  \|\tilde{z}\|_{L^\infty_t L^2_x} (1+\|\tilde{z}\|_{L^2_t L^2_x}) \|\tilde{w}_{N_1}\|_{L^\infty_t L^2_x} (NN_1^\al)^{-\fr13+} \|P_N\varphi\|_{X^{0,\fr13-}} \\
& \lesssim N^{\fr23+r-\beta} N_1^{\fr{2-2\al}{3}-r+} \|z\|_{L^\infty_T L^2_x} (1+\|z\|_{L^2_T L^2_x}) \|w_{N_1}\|_{L^\infty_T H^r_x} \, .
\end{align*}
that is summable for 
\begin{equation}\label{tutu222}
r>\frac{2-2\alpha}{3} \quad \text{and} \quad \beta>\frac{4-2\alpha}{3} \; .
\end{equation}
Finally the contribution of $Q_{\gtrsim NN_1^{\al}} \tilde{u}_{N_1}$ in the third term of \eqref{s5est}
\begin{align*}
N^{1+r-\beta}\Bigl|\int_{\R}\int_{\R} & Q_{\gtrsim NN_1^{\al}}\tilde{z}_{N_1}  Q_{\ll NN_1^{\al}}\tilde{w}_{N'_1} Q_{\sim NN_1^{\al}}P_N\varphi dxdt\Bigr| \\
& \lesssim N^{1+r-\beta} \| Q_{\gtrsim NN_1^{\al}} \tilde{z}_{N_1} \|_{L^2_t L^2_x}\|Q_{\ll NN_1^{\al}}\tilde{w}_{N_1'}\|_{L^\infty_t L^2_x}
\|Q_{\sim NN_1^{\al}}P_N\varphi\|_{L^2_t L^\infty_x}
\\
& \lesssim N^{1+r-\beta}\max\Big\{N_1^{\theta}(NN_1^\al)^{-1},
N_1^{\beta}(NN_1^\al)^{-\fr23}\Big\} \\
& \qquad \qquad \times \|\tilde{z}_{N_1}\|_{Y^{0}} \|\tilde{w}_{N_1}\|_{L^\infty_t L^2_x} N^{\fr12}(NN_1^\al)^{-\fr13+} \|P_N\varphi\|_{X^{0,\fr13-}} \\
& \lesssim  N ^{\fr76+r-\beta} N_1^{-\frac{\al}{3}-r+} \|w_{N_1}\|_{L^\infty_T H^r_x} \max\Big\{N_1^{\theta}(NN_1^\al)^{-1},
N_1^{\beta}(NN_1^\al)^{-\fr23}\Big\} \|z_{N_1}\|_{Y_T^{0}}\, .
\end{align*}
that is acceptable under the same conditions \eqref{tutu2}  as for  \eqref{tutu1} since $ \frac{5-4\alpha}{9}>-\frac{\alpha}{3} $ as soon as $ \alpha<5 $.

Gathering the above results and in particular the conditions \eqref{tutu2}, \eqref{tutu22} and \eqref{tutu222}, we obtain that for  $1 \ll N \lesssim N'_1\sim N_1$,
\begin{align*}
\sum_{N'_1\sim N_1\gtrsim N \gg 1} & \|\pa_x P_N(\tilde{z}_{N_1}\tilde{w}_{N'_1})\|_{F^{0,-\fr12}_T}
\lesssim  T^{0+}\|w\|_{Y^{r}_T} \|z\|_{Y^0_T}(1+\|z\|_{Y^0_T}^2),
\end{align*}
under the following conditions
\begin{equation*}
\begin{cases}
\theta-r-\al<0 \\
\beta-r-\fr23\al<0 \\
\theta-\beta<\frac{22\alpha-11}{18} \\
r>\frac{2-2\alpha}{3},
\end{cases}
\end{equation*}
where we used that $22\alpha-11<26\alpha-13 $ as soon as $\alpha>\frac12 $.
Therefore \eqref{estpropyr} holds under  the conditions \eqref{condprop51} and the proof of the proposition is complete.
\end{proof}

\begin{prop}\label{hr}
Let $0<T\le 2,\  \frac{37}{26}<\al\leq2,\ r<0$, and $u,v\in L^\infty_T L^2$ be a solution to \eqref{gKdV} associated with an initial datum $u_0,v_0 \in L^2(\R)$ such that $u_0-v_0 \in \bar{L}^2(\R)$. Then 
we get
\begin{equation}\label{prou1}
\|w\|_{L^\infty_T \bar{H}^r}^2 \lesssim \| w_0\|_{L^\infty_T \bar{H}^r}^2 + T^{0+}\|w\|_{Z^{r}_T}^2 \|z\|_{Y^0_T} (1+\|z\|_{Y^0_T}^2).
\end{equation}
provided 
\begin{equation}\label{coco2}
(\theta,\beta) \in  \Bigl]\frac{11-4\alpha}{6}, \frac{5-\alpha}{3}\Bigr[ \times\Bigl]  \frac{4-2\alpha}{3},\min(\alpha-1,\frac29 (5-\alpha))\Bigr[ \;,
\end{equation}
and 
\begin{equation}\label{coco3}
\begin{cases}
\theta-2r<\alpha \\
\beta-\frac{4}{3} r<\fr23\al \\
 \theta+r < \frac{10\alpha-11}{6}   \\
 \beta+r <\frac{8\alpha-11}{6}\\
-\alpha/5<r<\frac{3\al-5}{3} \; .
\end{cases}
\end{equation}
\end{prop}
\begin{rema}
It is worth noticing that the admissible set in \eqref{coco2} is not empty as soon as $ \alpha>7/5 $.
\end{rema} 
\begin{rema}\label{rem4}
It is easy to check  that \eqref{coco2}-\eqref{coco3} ensure, in particular, that the conditions \eqref{eq92}, namely,
$$
r>\max\Bigl(\fr{4-5\al}6, \theta-\alpha,\beta-\frac23 \alpha\Bigr), \quad \theta<\frac{2+5\alpha}{6} 
\quad\text{and} \quad \beta<\alpha/2 \; 
$$
that are shared by   Propositions \ref{Cw2}-\ref{bl2}  and Lemmas \ref{Cwv}-\ref{le3-2} in the Appendix are fulfilled. Indeed, 
  clearly  $0>2r>\theta-\alpha\Rightarrow r>\theta-\alpha $ and  $0>\frac{4}{3}r>\beta-\frac{2}{3} \alpha\Rightarrow r>\beta-
\frac{2}{3}\alpha $. Moreover, it is  easy to check that  for $\alpha \in [7/5,2] $, $\fr{4-5\al}6 \le -\alpha/5 $, $\frac{2+5\alpha}{6}\ge \frac{5-\alpha}{3} $ and 
$ \alpha/2 \ge \alpha-1 $ which proves the claim.
\end{rema} 
\begin{proof}
Recall that the difference $w=u-v$ satisfies \eqref{diff} with $z=u+v$. In this proof, we will make a constant use of Remark \ref{rem4} above. Applying the operator $P_N$ with $N>0$ dyadic to \eqref{diff}, taking the $L^2$ scalar product with $P_N w$ and integrating on $]0,t[$, we obtain
\begin{align*}
\|P_N w\|_{L^\infty_T \bar{H}^r}^2 &\lesssim \|P_N w_0\|_{L^\infty_T \bar{H}^r}^2 + \langle N^{-1} \rangle^2 \Nn^{2r} \sup_{t\in]0,T[}\Big|\int_0^t\int_{\R}P_N(zw)\pa_x P_N w dxdt'\Big|.
\end{align*}
Thus it remains to estimate
\begin{align*}
J:=\sum_{N>0} \langle N^{-2} \rangle \Nn^{2r} \sup_{t\in]0,T[}\Big|\int_0^t\int_{\R}P_N(zw)\pa_x P_N w dxdt'\Big|.
\end{align*}
We work with the  extensions $\tilde{z}=\rho_T(z)$ and $\tilde{w}=\rho_T(w) $ of $z$ and $w$ supported in time in $[-2,2]$ defined in \eqref{defrho}. To simplify the notation we drop the tilde in the sequel. As in \eqref{intep}-\eqref{intep1}, we eventually get 
\begin{align*}
J&\lesssim\sum_{N>0} \sum_{\tilde{N}\ll N}  \langle N^{-1} \rangle^2 \Nn^{2r} \tilde{N} \sup_{t\in]0,T[}\Big|\int_0^t \!\!\int_{\R}\Lambda_{a_1}(z_{\tilde{N}},w_{\sim N}) w_N dxdt'\Big| 
\\ &\quad +\sum_{N>0} \sum_{\tilde{N}\ll N}  \langle N^{-1} \rangle^2 \Nn^{2r} N \sup_{t\in]0,T[}\Big|\int_0^t\int_{\R}\Lambda_{a_2}(z_{\sim N} , w_{\tilde{N}})  w_N dxdt'\Big|
\\&\quad \quad +  \sum_{N>0} \sum_{N_1\gtrsim N} \langle N^{-1} \rangle^2 \Nn^{2r} N \sup_{t\in]0,T[}\Big|\int_0^t\int_{\R}\Lambda_{a_2}(z_{N_1} w_{\sim N_1}) w_N dxdt'\Big|,
\end{align*}
where $a_1$ is defined in \eqref{a1} and  $a_2(\xi,\xi_1)=\frac{\xi}{N}\phi_{N}(\xi)$. Therefore, it suffices to estimate
\begin{align*}
J&\lesssim\sum_{N>0} \sum_{0<\tilde{N}\ll N}  \langle N^{-1} \rangle^2 \Nn^{2r} \tilde{N}  \sup_{t\in]0,T[}\Big|I_t(z_{\tilde{N}},w_{\sim N}, w_N) \Big|
\\ &\quad +\sum_{N>0} \sum_{0<\tilde{N}\ll N}  \langle N^{-1} \rangle^2 \Nn^{2r} N \sup_{t\in]0,T[}\Big|I_t(w_{\tilde{N}}, z_{\sim N},   w_N)\Big|
\\&\quad \quad +  \sum_{N>0} \sum_{N_1\gtrsim N} \langle N^{-1} \rangle^2 \Nn^{2r} N \sup_{t\in]0,T[}\Big|I_t(w_{N_1}, z_{\sim N_1}, w_N)\Big|,
\\& =:J_1+J_2+J_3.
\end{align*}
\emph{Estimates for $J_1$}.
The contribution of the sum over $0<\tilde{N}\lesssim 1$ is easily estimated by
\begin{align*}
\sum_{0<\tilde{N}\lesssim 1} \sum_{N\gg\tilde{N}}  \langle N^{-1} \rangle^2 \Nn^{2r} \ \tilde{N} \tilde{N}^{\fr12}\|z_{\tilde{N}}\|_{L^\infty_t L^2_x} \|w_N\|_{L^2_{tx}}^2\lesssim  T \|z\|_{L^\infty_T L^2_x} \|w\|_{L^\infty_T \bar{H}^r_x}^2.
\end{align*}
It remains to estimate the sum over $\tilde{N}\gg 1$.
We set $R=N^{-(0+)}(\tilde{N} N^\al)$ so that $R\ll \tilde{N} N^\al$ since $N\gg \tilde{N}\gg 1$. We split $I_t$ as
\begin{align*}
I_t(z_{\tilde{N}},w_{\sim N},w_{N})=I_\infty(1^{\text{high}}_{t,R}z_{\tilde{N}},w_{\sim N},w_{N})+I_\infty(1^{\text{low}}_{t,R}z_{\tilde{N}},w_{\sim N},w_{N})=:I^{\text{high}}_t+I^{\text{low}}_t.
\end{align*}
The contribution of $I^{\text{high}}_t$ is estimated, thanks to Lemma \ref{le2-2} and H\"{o}lder's inequality, by
\begin{align*}
N^{2r} \tilde{N} |I_t^{\text{high}}|&\lesssim N^{2r} \tilde{N}^{\fr32} \|1^{\text{high}}_{t,R}\|_{L^1} \|z_{\tilde{N}}\|_{L^\infty_t L^2_x} \|w_{N}\|_{L^\infty_t L^2_x}^2 \\
& \lesssim  T^{0+} \tilde{N}^{\fr32} N^{0+}(\tilde{N} N^{\al-})^{-1+} \|z_{\tilde{N}}\|_{L^\infty_t L^2_x} \|w_{N}\|_{L^\infty_t H^r_x}^2 \\
&\lesssim T^{0+} \tilde{N}^{\frac12+}N^{-\al+0+}\|z_{\tilde{N}}\|_{L^\infty_t L^2_x} \|w_{N}\|_{L^\infty_t H^r_x}^2,
\end{align*}
that is acceptable for $\alpha>1/2 $. 
To evaluate the contribution of $I_t^{\text{low}}$, we use the decomposition
\begin{align*}
I_\infty(1^{\text{low}}_{t,R} z_{\tilde{N}},w_{\sim N},w_{N})&=I_\infty(1^{\text{low}}_{t,R} z_{\tilde{N}},Q_{\gtrsim \tilde{N} N^\al}w_{\sim N},w_{N})
\\& \quad +I_\infty(1^{\text{low}}_{t,R} z_{\tilde{N}},Q_{\ll \tilde{N} N^\al} w_{\sim N},Q_{\gtrsim \tilde{N} N^\al} w_{N})
\\& \quad\quad +I_\infty(1^{\text{low}}_{t,R} Q_{\sim \tilde{N} N^\al}z_{\tilde{N}},Q_{\ll \tilde{N} N^\al}w_{\sim N},Q_{\ll \tilde{N} N^\al}w_{N})
\\&=: I^{1,\text{low}}_t+I^{2,\text{low}}_t+I^{3,\text{low}}_t.
\end{align*}
For the first term, by Lemma \ref{le2-2}, Lemma \ref{Cu2} and \eqref{triv} we have
\begin{align*}
N^{2r} \tilde{N} |I^{1,\text{low}}_t| & \lesssim  N^{2r} \tilde{N}\|w_{N} \|_{L^\infty_t L^2_x} \|z_{\tilde{N}}\|_{L^2_t L^\infty_x}
\|Q_{\gtrsim \tilde{N} N^{\al}}w_{\sim N}\|_{L^{2+}_t L^2_x} \|1^{\text{low}}_{t,R}\|_{L^{\infty-}} \\
& \lesssim  T^{0+} (1+\|z\|_{L^\infty_T L^2_x}) N^{2r} \tilde{N} \tilde{N}^{\frac{2-\alpha}{3}}\|w_{N} \|_{L^\infty_t L^2_x} \|z\|_{L^\infty_t L^2_x}
\|Q_{\gtrsim \tilde{N} N^{\al}}w_{\sim N}\|_{L^{2+}_t L^2_x}
\\ &\lesssim  T^{0+}  (1+\|z\|_{L^\infty_T L^2_x})  \tilde{N} \tilde{N}^{\frac{2-\alpha}{3}}
\|w\|_{L^\infty_T H^r_x} \|z\|_{L^\infty_T L^2_x}  \\
& \qquad\qquad \times \max\Big\{N^{\theta}(\tilde{N} N^{\al})^{-1},
N^{\beta}(\tilde{N} N^{\al})^{-\fr23}\Big\}^{1-}\|w_{\sim N}\|_{Y_T^{r}}.
\end{align*}
For the second term, we proceed in the same way with Lemma \ref{le2-1} in hands to get
\begin{align*}
N^{2r} \tilde{N}  |I^{2,\text{low}}_t|
 & \lesssim N^{2r} \tilde{N}  \|Q_{\gtrsim \tilde{N} N^{\al}}w_{ N}\|_{L^{2+}_t L^2_x}
\|Q_{\ll \tilde{N}}  w_{\sim N}\|_{L^\infty_t L^2_x} \|  z_{\tilde{N}}\|_{L^2_T L^\infty_x}\|1^{\text{low}}_{t,R}\|_{L^{\infty-}}
\\
& \lesssim T^{0+} \tilde{N} \tilde{N}^{\frac{2-\alpha}{3}}
(1+\|z\|_{L^\infty_T L^2_x})  \|w\|_{L^\infty_T H^r_x}\|z\|_{L^\infty_T L^2_x}  \\
& \qquad\qquad \times\max\Big\{N^{\theta}(\tilde{N} N^{\al})^{-1},
N^{\beta}(\tilde{N} N^{\al})^{-\fr23}\Big\}^{1-} \|w_{N}\|_{Y_T^{r}}.
\end{align*}
To treat the  third term $I^{3,\text{low}}_t$, we rewrite
$Q_{\ll \tilde{N}N^{\al}}w_{\sim N}$ as
$$
Q_{\ll \tilde{N}N^{\al}}w_{\sim N}=w_{\sim N}-Q_{\gtrsim \tilde{N}N^{\al}}w_{\sim N}\, .
$$
With  Proposition \ref{Cw2} in hand (recall that according to Remark \ref{rem4}, the conditions \eqref{eq92} are satisfied), the contribution of $w_{\sim N}$ in the third term is estimated  as above  by 
\begin{align*}
N^{2r} \tilde{N}  |I_\infty(1^{\text{low}}_{t,R} Q_{\sim \tilde{N} N^\al}z_{\tilde{N}}, & w_{\sim N},Q_{\ll \tilde{N} N^\al}w_{N})| \\
& \lesssim  N^{2r} \tilde{N} \|Q_{\sim \tilde{N} N^\al}z_{\tilde{N}}\|_{L^{2+}_t L^2_x}
\|w_{\sim N}\|_{L^2_t L^\infty_x} \|Q_{\ll \tilde{N} N^\al}w_N\|_{L^\infty_t L^2_x}\|1^{\text{low}}_{t,R}\|_{L^{\infty-}}
\\& \lesssim T^{0+}\tilde{N}   N^{\fr{2-\al}{3}}  
\|w\|_{Y^{r}_T} (1+\|z\|_{Y^0_T}) \|w_{N}\|_{L^\infty_t H^r_x} \\
& \qquad\qquad \times \max\Big\{\Nt^{\theta}(\tilde{N} N^{\al})^{-1},
\Nt^\beta (\tilde{N} N^{\al})^{-\fr23}\Big\}^{1-} \|z_{\tilde{N}}\|_{Y_T^{0}},
\end{align*}
whereas, making use of Lemma \ref{first}, the  contribution of $Q_{\gtrsim \tilde{N}N^{\al}}w_{\sim N}$  is estimated  by
\begin{align*}
&\quad N^{2r} \tilde{N}  |I_\infty(1^{\text{low}}_{t,R} 
Q_{\sim \tilde{N} N^\al}z_{\tilde{N}},Q_{\gtrsim \tilde{N} N^\al}w_{\sim N},Q_{\ll \tilde{N} N^\al}w_N)|
\\&\lesssim T^{0+}N^{2r}\tilde{N}
\tilde{N}^{\fr12} \Nt^{\frac32+} (\tilde{N} N^\alpha)^{-1}
\max\Big\{N^{\theta}(\tilde{N} N^\al)^{-1},N^{\beta}(\tilde{N} N^\al)^{-\fr23}\Big\}^{1-} \\
& \qquad\qquad \times\|w_N\|_{L^\infty_t L^2_x}
\|z_{\tilde{N}}\|_{X^{-(\frac{3}{2}+),1}} \|w_{\sim N}\|_{Y^{0}}
\\&\lesssim T^{0+} \tilde{N}^{\frac12} \Nt^{\frac32+} N^{-\alpha}
\max\Big\{N^{\theta}(\tilde{N} N^\al)^{-1},N^{\beta}(\tilde{N} N^\al)^{-\fr23}\Big\}^{1-} \|w\|_{L^\infty_t H^r_x}
\|z\|_{Y^0} \|w\|_{Y^{r}}
\\&\lesssim T^{0+} \tilde{N}  N^{-\alpha+1+}
\max\Big\{N^{\theta}(\tilde{N} N^\al)^{-1},N^{\beta}(\tilde{N} N^\al)^{-\fr23}\Big\}^{1-} \|w\|_{L^\infty_t H^r_x}
\|z\|_{Y^0} \|w\|_{Y^{r}} \; ,
\end{align*}
where we use the fact that $Y^0 \hookrightarrow {X^{-(\frac{3}{2}+),1}}$ since $\theta<\frac{5-\al}{3}<\fr32$. 
Since $1<\alpha\le 2 $ we deduce that  
\begin{align*}
\sum_{\tilde{N}\gg 1} &\sum_{N\gg \tilde{N}}  \langle N^{-1} \rangle^2  \Nn^{2r} \tilde{N}  \sup_{t\in]0,T[}\Big|I_t(z_{\tilde{N}},w_{\sim N}, w_N) \Big|\\
  & \lesssim T^{0+} \sum_{N\gg 1} \sum_{0<\tilde{N}\ll N}  N^{\fr{2-\al}{3}} \tilde{N} \max\Big\{N^{\theta}(\tilde{N} N^\al)^{-1},N^{\beta}(\tilde{N} N^\al)^{-\fr23}\Big\}^{1-}
\|w\|_{Y^{r}_T}^2  \|z\|_{Y^{0}_T} (1+\|z\|_{Y^{0}_T}) \\
& \lesssim T^{0+} \sum_{N\gg 1}    \max\Big\{N^{\theta+\frac{2-4\alpha}{3}+},N^{\beta+1-\alpha+}\Big\}
\|w\|_{Y^{r}_T}^2  \|z\|_{Y^{0}_T} (1+\|z\|_{Y^{0}_T})\\
& \lesssim T^{0+} \|w\|_{Y^{r}_T}^2  \|z\|_{Y^{0}_T} (1+\|z\|_{Y^{0}_T})
\end{align*}
where in the last step we used that  $\theta<\frac{4\alpha-2}{3} $ and $\beta<\alpha-1 $.

\emph{Estimates for $J_2$}.
The contribution of the sum over $N\lesssim 1$ is easily estimated by
\begin{align*}
 \sum_{0<N\lesssim 1} \sum_{0<\tilde{N}\ll N}  \langle N^{-1}\rangle^2 \Nn^{2r} N \tilde{N}^{\fr12}\|z_{N}\|_{L^2_T L^2_x} \|w_N\|_{L^2_T L^2_x}\|w_{\tilde{N}}\|_{L^\infty_T L^2_x}\lesssim T  \|z\|_{L^\infty_t L^2_x} \|w\|_{L^\infty_t \bar{H}^r_x}^2.
\end{align*}
The contribution of the sum over $ N\gg 1 $ and $ 0<\tilde{N} \lesssim N^{-\frac{2}{3}(r+1+)} \lesssim 1$ (since $r>-\frac{\alpha}{5}\geq -1$) is estimated as follows 
\begin{align*}
 \sum_{N\gg 1}   \sum_{0<\tilde{N}\lesssim  N^{-\frac{2}{3}(r+1+)} }  &  N^{2r} N \tilde{N}^{\fr12}\|z_{N}\|_{L^2_T L^2_x} \|w_N\|_{L^2_T L^2_x}\|w_{\tilde{N}}\|_{L^\infty_T L^2_x}  \\
  &   \lesssim   \sum_{N\gg 1}   \sum_{0<\tilde{N}\lesssim   N^{-\frac{2}{3}(r+1+)} } 
  N^{1+r} \tilde{N}^{\fr32} \|z_{N}\|_{L^2_T L^2_x} \|w_N\|_{L^2_T H^r_x}\|w_{\tilde{N}}\|_{L^\infty_T \bar{H}^r_x}\\
  & \lesssim T \, \|z\|_{L^\infty_t L^2_x} \|w\|_{L^\infty_t \bar{H}^r_x}^2.
 \end{align*}
 It thus remains to estimate the sum over $ N\gg 1 $ and $ \tilde{N} \gg N^{-\frac{2}{3}(r+1+)} $.
 We 
set $R=N^{-(0+)} (\tilde{N} N^\al)$ so that $ R\ll \tilde{N} N^\al $. Note also that since $ r<0 $, it holds $\tilde{N} N^\al\gg N^{\alpha-\frac23} \gg N^{\frac13} $ and thus $ R\gg 1$.
We split $I_t$ as
\begin{align*}
I_t(w_{\tilde{N}},z_{\sim N},w_{N})=I_\infty(1^{\text{high}}_{t,R}w_{\tilde{N}},z_{\sim N},w_{N})+I_\infty(1^{\text{low}}_{t,R}w_{\tilde{N}},z_{\sim N},w_{N})=:I^{\text{high}}_t+I^{\text{low}}_t.
\end{align*}
Thanks to Lemma \ref{le2-2} and H\"{o}lder's inequality, the contribution of $I^{\text{high}}_t$ is estimated, for $\tilde{N}\gg 1$, by
\begin{align*}
N^{1+2r} |I_t^{\text{high}}|&\lesssim N^{1+2r} \tilde{N}^{\fr12} \|1^{\text{high}}_{t,R}\|_{L^1} \|w_{\tilde{N}}\|_{L^\infty_t L^2_x} \|w_{N}\|_{L^\infty_t L^2_x}\|z_{N}\|_{L^\infty_t L^2_x} \\
& \lesssim  T^{0+} \tilde{N}^{\fr12} N^{0+}(\tilde{N} N^{\al-})^{-1+} 
\tilde{N}^{-r} N^{1+r} \|w_{\tilde{N}}\|_{L^\infty_t H^r_x} \|w_{N}\|_{L^\infty_t H^r_x}\|z_{N}\|_{L^\infty_t L^2_x} \\
&\lesssim T^{0+} N^{1+r-\alpha+}  \tilde{N}^{-\frac12-r+}
 \|w\|_{L^\infty_t H^r_x}^2 \|z\|_{L^\infty_t L^2_x},
\end{align*}
that is acceptable since $\alpha>1$ and $r<0 $ ensures that $1+r-\alpha<0 $ and $ \frac12-\alpha<0 $. 
In the same way, for  $\tilde{N}\lesssim 1$, we obtain
\begin{align*}
N^{1+2r} |I_t^{\text{high}}|&\lesssim N^{1+2r} \tilde{N}^{\fr12} \tilde{N} \|1^{\text{high}}_{t,R}\|_{L^1} \|w_{\tilde{N}}\|_{L^\infty_t \bar{H}^0_x} \|w_{N}\|_{L^\infty_t L^2_x}\|z_{N}\|_{L^\infty_t L^2_x} \\
& \lesssim  T^{0+} \tilde{N}^{\fr32} N^{0+}(\tilde{N} N^{\al-})^{-1+} 
 N^{1+r} \|w_{\tilde{N}}\|_{L^\infty_t \bar{H}^r_x} \|w_{N}\|_{L^\infty_t H^r_x}\|z_{N}\|_{L^\infty_t L^2_x} \\
&\lesssim T^{0+}  \tilde{N}^{\fr12+}  N^{1+r-\alpha+} 
 \|w\|_{L^\infty_t \bar{H}^r_x}^2 \|z\|_{L^\infty_t L^2_x}
\end{align*}
that is also acceptable for the same reasons.
To evaluate the contribution $I_t^{\text{low}}$, we use the decomposition
\begin{align*}
I_\infty(1^{\text{low}}_{t,R} w_{\tilde{N}},z_{\sim N},w_{N})&=I_\infty(1^{\text{low}}_{t,R} w_{\tilde{N}},  Q_{\gtrsim \tilde{N}N^{\al}} z_{\sim N},w_{N})
\\& \quad +I_\infty(1^{\text{low}}_{t,R}w_{\tilde{N}},Q_{\ll \tilde{N}N^{\al}}z_{\sim N},Q_{\gtrsim \tilde{N}N^{\al}}w_{N})
\\& \quad\quad +I_\infty(1^{\text{low}}_{t,R} Q_{\sim \tilde{N}N^{\al}} w_{\tilde{N}},Q_{\ll \tilde{N}N^{\al}}z_{\sim N},Q_{\ll \tilde{N}N^{\al}}w_{N})
\\&=: I^{1,\text{low}}_t+I^{2,\text{low}}_t+I^{3,\text{low}}_t.
\end{align*}
In the case of $N^{-\fr23(r+1)-}\lesssim \tilde{N} \lesssim 1$, we use \eqref{Est11} of Proposition \ref{bl2}. This is justified since $r<0 $ forces $-\fr23(r+1)-\ge -2/3 $.
\begin{align}
N^{1+2r} |I^{1,\text{low}}_t|
&\lesssim N^{1+2r} \|w_{\tilde{N}} w_N\|_{L^2_t L^2_x}
\|Q_{\gtrsim \tilde{N}N^{\al}} z_{\sim N}\|_{L^{2+}_t L^2_x}  \|1^{\text{low}}_{t,R}\|_{L^{\infty-}_t}
\nonumber\\& \lesssim T^{0+} N^{1+r} \tilde{N}N^{\fr{5-4\al}{6}} (1+\|z\|_{Y^0_T}^2) 
\|w\|_{Z^r_T}^2\\
& \qquad\qquad \times
 \max\Big\{N^{\theta}(\tilde{N}N^\al)^{-1},
N^{\beta}(\tilde{N}N^\al)^{-\fr23}\Big\}^{1-} \|z_{\sim N}\|_{Y_T^{0}}
\nonumber\\& \lesssim T^{0+} \tilde{N}^{0+}\max\Big\{N^{\theta+r+\frac{11-10\alpha}{6}}, N^{\beta+r+\frac{11-8\alpha}{6}}\Bigr)
\|w\|_{Z_T^{r}}^2 \|z\|_{Y^0_T} (1+\|z\|_{Y^0_T}^2),\label{za1} 
\end{align}
that is acceptable whenever 
\begin{equation}\label{az1}
 \theta+r < \frac{10\alpha-11}{6} \quad \text{and} \quad  \beta+r <\frac{8\alpha-11}{6} .
\end{equation}
In the case of $\tilde{N}\gg 1$, by \eqref{Est1} of Proposition \ref{bl2} we have 
\begin{align}
N^{1+2r} &|I^{1,\text{low}}_t|
\lesssim N^{1+2r}\|w_{\tilde{N}} w_N\|_{L^2_t L^2_x}
\|Q_{\gtrsim \tilde{N}N^{\al}} z_{\sim N}\|_{L^{2+}_t L^2_x}  \|1^{\text{low}}_{t,R}\|_{L^{\infty-}_t}\nonumber 
\\& \lesssim T^{0+} N^{1+r}\tilde{N}^{-r} N^{\fr{5-4\al}{6}} (1+\|z\|_{Y^0_T}^2) \|w_{\tilde{N}}\|_{Z^r_T}
 \|w_{N}\|_{Z^r_T} \\
 & \qquad\qquad \times\max\Big\{N^{\theta}(\tilde{N}N^\al)^{-1},
N^{\beta}(\tilde{N}N^\al)^{-\fr23}\Big\}^{1-} \|z_{\sim N}\|_{Y_T^{0}}
\nonumber \\& \lesssim T^{0+} \max\Big\{ \tilde{N}^{-r-1+}N^{\theta+r+\frac{11-10\alpha}{6}}, \tilde{N}^{-r-\frac{2}{3}+} N^{\beta+r+\frac{11-8\alpha}{6}}\Bigr)
\|w\|_{Z_T^{r}}^2 \|z\|_{Y^0_T} (1+\|z\|_{Y^0_T}^2), \label{za2}
\end{align}
that is acceptable under the same conditions as above since $ -\frac12<-\fr{\al}5<r<0$.\\
To treat the  second term $I^{2,\text{low}}_t$, we rewrite
$Q_{\ll \tilde{N}N^{\al}}z_{\sim N}$ as
\begin{equation} \label{hhg}
Q_{\ll \tilde{N}N^{\al}}z_{\sim N}=z_{\sim N}-Q_{\gtrsim \tilde{N}N^{\al}}z_{\sim N}\, .
\end{equation}
The contribution of $z_{\sim N}$ in $I^{2,\text{low}}_t$   can be estimated in the same way as $I^{1,\text{low}}_t$  under the same conditions.
In the case of $N^{-\fr23(r+1)-}\lesssim \tilde{N} \lesssim 1$, we use \eqref{Est22} of Proposition \ref{bl2} to get 
\begin{align*}
N^{1+2r} |I_\infty(1^{\text{low}}_{t,R} & w_{\tilde{N}},z_{\sim N},Q_{\gtrsim \tilde{N}N^{\al}}w_{N})|
\lesssim N^{1+2r} \|w_{\tilde{N}} z_{\sim N} \|_{L^2_t L^2_x}
\|Q_{\gtrsim \tilde{N}N^{\al}} w_{N}\|_{L^{2+}_t L^2_x}  \|1^{\text{low}}_{t,R}\|_{L^{\infty-}_t}
\\& \lesssim T^{0+} N^{1+r} \tilde{N}N^{\fr{5-4\al}{6}} (1+\|z\|_{Y^0_T}^2) \|w\|_{Z^r_T} 
\|z\|_{Y^0_T} \\
& \qquad\qquad \times
 \max\Big\{N^{\theta}(\tilde{N}N^\al)^{-1},
N^{\beta}(\tilde{N}N^\al)^{-\fr23}\Big\}^{1-} \|w_{N}\|_{Z_T^{r}}
\\& \lesssim T^{0+}  \tilde{N}^{0+}\max\Big\{N^{\theta+r+\frac{11-10\alpha}{6}}, N^{\beta+r+\frac{11-8\alpha}{6}}\Bigr)
\|w\|_{Z_T^{r}}^2 \|z\|_{Y^0_T} (1+\|z\|_{Y^0_T}^2),
\end{align*}
whereas in the case of $\tilde{N}\gg 1$, by \eqref{Est2} of Proposition \ref{bl2} we have 
\begin{align*}
N^{1+2r} &|I_\infty(1^{\text{low}}_{t,R}  w_{\tilde{N}},z_{\sim N},Q_{\gtrsim \tilde{N}N^{\al}}w_{N})|
\lesssim N^{1+2r} \|w_{\tilde{N}} z_{\sim N} \|_{L^2_t L^2_x}
\|Q_{\gtrsim \tilde{N}N^{\al}} w_{N}\|_{L^{2+}_t L^2_x}  \|1^{\text{low}}_{t,R}\|_{L^{\infty-}_t}
\\& \lesssim T^{0+} N^{1+r} \tilde{N}^{-r}N^{\fr{5-4\al}{6}} (1+\|z\|_{Y^0_T}^2) \|w\|_{Z^r_T} \|z\|_{Y^0_T} \\
& \qquad\qquad \times
 \max\Big\{N^{\theta}(\tilde{N}N^\al)^{-1},
N^{\beta}(\tilde{N}N^\al)^{-\fr23}\Big\}^{1-} \|w_{N}\|_{Z_T^{r}}
\\
& \lesssim T^{0+}\max\Big\{ \tilde{N}^{-r-1+}N^{\theta+r+\frac{11-10\alpha}{6}}, \tilde{N}^{-r-\frac{2}{3}+} N^{\beta+r+\frac{11-8\alpha}{6}}\Bigr)
\|w\|_{Z_T^{r}}^2 \|z\|_{Y^0_T} (1+\|z\|_{Y^0_T}^2).
\end{align*}
Now the contribution of $Q_{\gtrsim \tilde{N}N^{\al}}z_{\sim N}$ in $I^{2,\text{low}}_t$ can be estimated as follows :
 In the case $0<\tilde{N}\lesssim N^{\fr{5-4\al}{3}}\ll 1$, we have by Bernstein inequality
\begin{align*}
N^{1+2r}  & |I_\infty(1^{\text{low}}_{t,R}     w_{\tilde{N}},Q_{\gtrsim \tilde{N}N^{\al}} z_{\sim N},Q_{\gtrsim \tilde{N}N^{\al}}w_{N})| \\
& \lesssim  T^{\frac12 +} N^{1+2r}
\tilde{N}^{\fr12}\max\Big\{N^{\theta}(\tilde{N} N^\al)^{-1},N^{\beta}(\tilde{N} N^\al)^{-\fr23}\Big\}^{1-} \|z_{\sim N}\|_{Y^{0}} \|w_{\tilde{N}}\|_{L^\infty_t L^2_x} \| w_{N}\|_{L^\infty_t L^2_x} \\
& \lesssim  T^{\frac12 +} N^{1+r}
\tilde{N}^{\fr12} \tilde{N} \max\Big\{N^{\theta}(\tilde{N} N^\al)^{-1},N^{\beta}(\tilde{N} N^\al)^{-\fr23}\Big\}^{1-} \|z_{\sim N}\|_{Y^{0}} \|w_{\tilde{N}}\|_{L^\infty_t \bar{H}^r_x} \| w_{N}\|_{L^\infty_t H^r_x}
\\&\lesssim T^{\frac12 +} N^{1+r}
N^{\fr{5-4\al}{6}}  \tilde{N}
\max\Big\{N^{\theta}(\tilde{N} N^\al)^{-1},N^{\beta}(\tilde{N} N^\al)^{-\fr23}\Big\}^{1-} \|w\|_{Y^{r}_T} \|w\|_{Z^{r}_T}  \|z\|_{Y^0_T} \\
& \lesssim T^{0+}  \tilde{N}^{0+}\max\Big\{N^{\theta+r+\frac{11-10\alpha}{6}}, N^{\beta+r+\frac{11-8\alpha}{6}}\Bigr)
\|w\|_{Z_T^{r}}^2 \|z\|_{Y^0_T},
\end{align*}
that is acceptable whenever \eqref{az1} is satisfied.
In the case of $N^{\fr{5-4\al}{3}}\lesssim\tilde{N}\lesssim 1$, we  get 
\begin{align}
N^{1+2r} & |I_\infty(1^{\text{low}}_{t,R} w_{\tilde{N}}, Q_{\gtrsim \tilde{N} N^\al}z_{\sim N},Q_{\gtrsim \tilde{N} N^\al}w_N)|
\nonumber\\
&\lesssim T^{0+} N^{1+r}\tilde{N}^{\fr12} \tilde{N}
\max\Big\{N^{\theta}(\tilde{N} N^\al)^{-1},N^{\beta}(\tilde{N} N^\al)^{-\fr23}\Big\}^{2-} \|z_{\sim N}\|_{Y^0_T} \| w_{\tilde{N}}\|_{L^\infty_t \bar{H}^r_x} \|w_{N}\|_{Y^r_T} \nonumber \\
&\lesssim T^{0+} 
\max\Big\{N^{\theta-\alpha}\tilde{N}^{-\fr12},N^{\beta-\fr23 \al}\tilde{N}^{-\fr16}  \Big\} N^{1+r}
\tilde{N} \nonumber\\
& \qquad\qquad \times\max\Big\{N^{\theta}(\tilde{N} N^\al)^{-1},N^{\beta}(\tilde{N} N^\al)^{-\fr23}\Big\}^{1-} \|z_{\sim N}\|_{Y^0_T} \| w_{\tilde{N}}\|_{L^\infty_t \bar{H}^r_x} \|w_{N}\|_{Y^r_T}  \nonumber
 \end{align}
 At this stage we notice that for $N^{\fr{5-4\al}{3}}\lesssim\tilde{N}\lesssim 1$
 $$
 \max\Big\{N^{\theta-\alpha}\tilde{N}^{-\fr12},N^{\beta-\fr23 \al}\tilde{N}^{-\fr16}  \Big\} \le N^{\fr{5-4\al}{6}}
 $$
 whenever 
 \begin{equation}\label{az2}
 \theta <\frac{5-\alpha}{3} \quad\text{and}\quad  \beta <\frac{2}{9} (5-\alpha) \, .
 \end{equation}
 Under this condition we thus get 
 \begin{align*}
 N^{1+2r} & |I_\infty(1^{\text{low}}_{t,R} w_{\tilde{N}}, Q_{\gtrsim \tilde{N} N^\al}z_{\sim N},Q_{\gtrsim \tilde{N} N^\al}w_N)| \\
&\lesssim T^{0+} 
N^{\fr{5-4\al}{6}}  N^{1+r}
\tilde{N} \max\Big\{N^{\theta}(\tilde{N} N^\al)^{-1},N^{\beta}(\tilde{N} N^\al)^{-\fr23}\Big\}^{1-}\|w\|_{Z_T^{r}}^2 \|z\|_{Y^0_T}, 
 \end{align*}
 that is acceptable as above. 
Finally for $ \tilde{N} \gg 1$ we proceed similarly to get that under the same conditions \eqref{az2} it holds 
\begin{align}
N^{1+2r} & |I_\infty(1^{\text{low}}_{t,R} w_{\tilde{N}}, Q_{\gtrsim \tilde{N} N^\al}z_{\sim N},Q_{\sim \tilde{N} N^\al}w_N)|
\nonumber\\
&\lesssim T^{0+} N^{1+r}\tilde{N}^{\fr12}\tilde{N}^{-r}
\max\Big\{N^{\theta}(\tilde{N} N^\al)^{-1},N^{\beta}(\tilde{N} N^\al)^{-\fr23}\Big\}^{2-} \|z_{\sim N}\|_{Y^0_T} \| w_{\tilde{N}}\|_{L^\infty_t H^r_x} \|w_{N}\|_{Y^r_T} \nonumber\\
&\lesssim T^{0+} 
\max\Big\{N^{\theta-\alpha}\tilde{N}^{-\fr12},N^{\beta-\fr23 \al}\tilde{N}^{-\fr16}  \Big\} N^{1+r}\tilde{N}^{-r}\\
& \qquad\qquad \times
 \max\Big\{N^{\theta}(\tilde{N} N^\al)^{-1},N^{\beta}(\tilde{N} N^\al)^{-\fr23}\Big\}^{1-} \|z_{\sim N}\|_{Y^0_T} \| w_{\tilde{N}}\|_{L^\infty_t \bar{H}^r_x} \|w_{N}\|_{Y^r_T} \nonumber\\
&\lesssim T^{0+} 
N^{\fr{5-4\al}{6}}  N^{1+r}
\tilde{N}^{-r} \max\Big\{N^{\theta}(\tilde{N} N^\al)^{-1},N^{\beta}(\tilde{N} N^\al)^{-\fr23}\Big\}^{1-}
 \|w\|_{Z_T^{r}}^2 \|z\|_{Y^0_T} \label{s5re1}
  \end{align}
  that is acceptable as in \eqref{za2}. Therefore, these contributions are acceptable whenever  \eqref{az1} and \eqref{az2}  are satisfied.

To treat the  third term $I^{3,\text{low}}_t$, we  separated three cases. 
In the case  $0<\tilde{N}\lesssim N^{-\fr23(r+1)-}\ll 1$,
\begin{align*}
N^{1+2r} |I^{3,\text{low}}_t|
&\lesssim N^{1+2r}\tilde{N}^{\fr12}\|w_{\tilde{N}}\|_{L^\infty_t L^2_x}
\|z_{\sim N}\|_{L^\infty_T L^2_x} \|P_N w\|_{L^\infty_T L^2_x} \|1^{\text{low}}_{t,R}\|_{L^1}
\\& \lesssim T  \tilde{N}^{\fr32}  N^{1+r}
\|z\|_{L^\infty_T L^2_x} \|w_{N}\|_{L^\infty_T H^r_x} \|w_{\tilde{N}}\|_{L^\infty_T \bar{H}^r_x}
\\& \lesssim  T  N^{0-}
\|z\|_{L^\infty_T L^2_x} \|w\|_{Z^r_T}^2,
\end{align*}
that is acceptable.
In the case of $N^{-\fr23(r+1)-}\lesssim \tilde{N} \lesssim 1$, we have 
\begin{align*}
N^{1+2r} |I^{3,\text{low}}_t|
&\lesssim N^{1+2r} \tilde{N}^{\fr12}\|Q_{\sim \tilde{N}N^{\al}}w_{\tilde{N}}\|_{L^{2+}_t L^2_x}
\|z_{\sim N}\|_{L^\infty_T L^2_x} \|P_N w\|_{L^\infty_T L^2_x} \|1^{\text{low}}_{t,R}\|_{L^{2-}}
\\& \lesssim T^{\fr12+} \tilde{N}^{\fr12} N^{1+r}
\|z\|_{L^\infty_T L^2_x} \|w_{N}\|_{L^\infty_T H^r_x} \max\Big\{(\tilde{N}N^\al)^{-1},
(\tilde{N}N^\al)^{-\fr23}\Big\}^{1-} \|w_{\tilde{N}}\|_{Y_T^{r}}
\\& \lesssim T^{0+}
\max\Big\{N^{-\al+\fr43 (r+1)+},
N^{-\fr23\al+\fr{10}{9}(r+1)+}\Big\} \|w_{\tilde{N}}\|_{Y_T^{r}} \|z\|_{L^\infty_T L^2_x} \|w_{N}\|_{L^\infty_T H^r_x}
\end{align*}
that is acceptable whenever $1+r < \frac{3}{5} \alpha \Leftrightarrow r<\frac{3\alpha-5}{5}$.
In the case $\tilde{N}\gg 1$, we rewrite $Q_{\ll \tilde{N} N^\alpha} z_{\sim N}$ as in \eqref{hhg}. The contribution of $ z_{\sim N} $ is then estimated thanks to Lemma \ref{Cu2} by  
\begin{align*}
&N^{1+2r} |I^{3,\text{low}}_t|
\lesssim N^{1+2r}  \|Q_{\sim \tilde{N}N^{\al}}w_{\tilde{N}}\|_{L^{2+}_t L^2_x}
\|z_{\sim N} \|_{L^2_t L^\infty_x} \|P_N w\|_{L^\infty_T L^2_x} \|1^{\text{low}}_{t,R}\|_{L^{\infty-}}
\\& \lesssim T^{0+} N^{\fr{2-\al}{3}}  N^{1+r}
\|z\|_{L^\infty_T L^2_x} (1+\|z\|_{Y^0_T}) \|w_{N}\|_{L^\infty_T H^r_x} \\
& \qquad\qquad \times\max\Big\{\tilde{N}^{\theta-r}(\tilde{N}N^\al)^{-1},
\tilde{N}^{\beta-r}(\tilde{N}N^\al)^{-\fr23}\Big\}^{1-} \|w_{\tilde{N}}\|_{Y_T^{r}}
\\& \lesssim T^{0+}
\max\Big\{\tilde{N}^{\theta-r-1}N^{r+\fr{5-4\al}{3}+},
\tilde{N}^{\beta-r-\fr23}N^{r+\fr{5-3\al}{3}+}\Big\}
\|z\|_{L^\infty_T L^2_x}  \|w_{N}\|_{L^\infty_T H^r_x} \|w_{\tilde{N}}\|_{Y_T^{r}} (1+\|z\|_{Y^0_T}),
\end{align*}
that is acceptable if $\theta<\frac{4\alpha-2}{3} $, $\beta<\alpha-1 $ and 
\begin{equation}\label{condr}
r  <\frac{3\alpha-5}{3}\; .
\end{equation}
Since  $ r<0 $, it is worth noticing that \eqref{condr} implies the condition $r<\frac{3\alpha-5}{5} $ obtained in the case  $N^{-\fr23(r+1)-}\lesssim \tilde{N} \lesssim 1$. 
For the contribution of $Q_{\gtrsim \tilde{N} N^\alpha} z_{\sim N}$ we proceed similarly as in \eqref{s5re1} to get, under condition \eqref{az2}, that 
\begin{align*}
N^{1+2r} & |I_\infty(1^{\text{low}}_{t,R} Q_{\sim \tilde{N} N^\al}w_{\tilde{N}}, Q_{\gtrsim \tilde{N} N^\al}z_{\sim N},Q_{\ll \tilde{N} N^\al}w_N)|
\\
&\lesssim T^{0+} N^{1+r}\tilde{N}^{\fr12} \tilde{N}^{-r}
\max\Big\{\tilde{N}^{\theta}(\tilde{N} N^\al)^{-1},\tilde{N}^{\beta}(\tilde{N} N^\al)^{-\fr23}\Big\} \\
& \qquad\qquad \times
\max\Big\{N^{\theta}(\tilde{N} N^\al)^{-1},N^{\beta}(\tilde{N} N^\al)^{-\fr23}\Big\}^{1-} \|z_{\sim N}\|_{Y^0_T} \| w_{\tilde{N}}\|_{Y^r_T}\|w_{N}\| _{L^\infty_t H^r_x}  \\
&\lesssim T^{0+} 
N^{\fr{5-4\al}{6}}  N^{1+r}
\tilde{N}^{-r} \max\Big\{N^{\theta}(\tilde{N} N^\al)^{-1},N^{\beta}(\tilde{N} N^\al)^{-\fr23}\Big\}^{1-}
 \|w\|_{Z_T^{r}}^2 \|z\|_{Y^0_T} 
      \end{align*}
which is acceptable as in \eqref{s5re1}. 
Gathering the  above results we infer that the contribution of   $J_2$ is acceptable under the additional conditions \eqref{az1}, \eqref{az2} and \eqref{condr}. 

\emph{Estimates for $J_3$}.
We separate the contributions of $  N_1 \lesssim 1$, $N \lesssim 1 \ll N_1$ and $1 \ll N \lesssim N_1$. We first notice that $\langle N^{-1} \rangle N\sim \Nn$, the contribution of the sum over $ N_1 \lesssim 1$ is easily estimated thanks to Cauchy-Schwarz by
\begin{align}\label{J3-1}
\sum_{0<N_1\lesssim 1} \sum_{0<N\lesssim N_1}  N^{\fr32} \langle N^{-1} \rangle^{2} \|w_N\|_{L^\infty_T H^r_x}\|z_{N_1}\|_{L^2_T L^2_x} \|w_{N_1}\|_{L^2_T H^r_x}\lesssim T  \|z\|_{L^\infty_t L^2_x} \|w\|_{L^\infty_t H^r_x}\|w\|_{L^\infty_t \bar{H}^r_x}.
\end{align}

For the contribution of  $N \lesssim 1 \ll N_1$,  the sum over $N<N_1^{2r-}$ and $N_1 \gg 1$, we proceed in the same way to get
\begin{align}\label{J3-2}
&\sum_{N_1 \gg 1} \sum_{0<N<N_1^{2r-}} \langle N^{-1} \rangle^2  N |I_t(w_N,w_{N_1},z_{\sim N_1})|
\nonumber\\&\lesssim \sum_{N_1 \gg 1}\sum_{0<N<N_1^{2r-}} \langle N^{-1} \rangle^2  N N^{\fr12}  \|w_N\|_{L^\infty_t L^2_x} \|w_{N_1}\|_{L^2_t L^2_x} \|z_{\sim N_1}\|_{L^2_t L^2_x}
\nonumber\\
&\lesssim \sum_{N_1 \gg 1}\sum_{0<N<N_1^{2r-}} N_1^{-r} N^{\fr12}  \|w_N\|_{L^\infty_t \bar{H}^r_x} \|w_{N_1}\|_{L^2_t H^r_x} \|z_{\sim N_1}\|_{L^2_t L^2_x}
\nonumber\\
&\lesssim T  \|z\|_{L^\infty_t L^2_x} \|w\|_{L^\infty_t H^r_x}\|w\|_{L^\infty_t \bar{H}^r_x}.
\end{align}
It remains to estimate the sum over $1\gtrsim N\geq N_1^{2r-}$ and $N_1 \gg 1$. We set $R=N_1^{-(0+)}(N N_1^\al)$ so that 
 $R\ll N {N_1}^\al$ since $N_1\gg 1$. Note also that $R\gg 1 $ since $r >-\al/5>-\alpha/2 $. We split $I_t$ as
\begin{align*}
I_t(w_N,w_{N_1},z_{\sim N_1})=I_\infty(1^{\text{high}}_{t,R}w_N,w_{N_1},z_{\sim N_1})+I_\infty(1^{\text{low}}_{t,R}w_N,w_{N_1},z_{\sim N_1})=:I^{\text{high}}_t+I^{\text{low}}_t.
\end{align*}
The contribution of $I^{\text{high}}_t$ is estimated, thanks to Lemma \ref{le2-2} and H\"{o}lder's inequality, by
\begin{align*}
&N^{-1}|I_t^{\text{high}}|\lesssim N^{-1} N^{\fr12}\|1^{\text{high}}_{t,R}\|_{L^1}  \|w_{N}\|_{L^\infty_t L^2_x} \|w_{N_1}\|_{L^\infty_t L^2_x} \|z_{\sim N_1}\|_{L^\infty_t L^2_x}
\\&\lesssim T^{0+} \, N^{\fr12}N_1^{0+}(N N_1^{\al-})^{-1+} N_1^{-r}\|w_{N}\|_{L^\infty_t \bar{H}^r_x} \|w_{N_1}\|_{L^\infty_t H^r_x} \|z_{\sim N_1}\|_{L^\infty_t L^2_x}
\\&\lesssim T^{0+} \, N_1^{-2r-\al+}\|w_{N}\|_{L^\infty_t \bar{H}^r_x} \|w_{N_1}\|_{L^\infty_t H^r_x} \|z_{\sim N_1}\|_{L^\infty_t L^2_x},
\end{align*}
which is acceptable since $r >-\alpha/2 $.
To evaluate the contribution $I_t^{\text{low}}$, we use the decomposition
\begin{align*}
I_\infty(1^{\text{low}}_{t,R} w_N,w_{N_1},z_{\sim N_1})&=I_\infty(Q_{\gtrsim N N_1^\al}(1^{\text{low}}_{t,R} w_N),w_{N_1},z_{\sim N_1})
\\& \quad +I_\infty(Q_{\ll N N_1^\al}(1^{\text{low}}_{t,R} w_N),Q_{\gtrsim N N_1^\al}w_{N_1},z_{\sim N_1})
\\& \quad\quad +I_\infty(Q_{\ll N N_1^\al}(1^{\text{low}}_{t,R} w_N),Q_{\ll N N_1^\al}w_{N_1},Q_{\sim N N_1^\al}z_{\sim N_1})
\\&=: I^{1,\text{low}}_t+I^{2,\text{low}}_t+I^{3,\text{low}}_t.
\end{align*}
Then
\begin{align*}
N^{-1} |I^{1,\text{low}}_t|
& \lesssim T^\frac12 N^{-1}  N^{\fr12} \max\Big\{\langle N\rangle^{\theta}(N N_1^\al)^{-1},\langle N\rangle^{\beta}(N N_1^\al)^{-\fr23}\Big\} \|w_N\|_{Y^{0}} \|w_{N_1}\|_{L^\infty_t L^2_x} \|z_{\sim N_1}\|_{L^\infty_t L^2_x} \\
& \lesssim  T^\frac12 N^{-\fr12} N_1^{-r} \max\Big\{(N N_1^\al)^{-1},(N N_1^\al)^{-\fr23}\Big\} \|w_N\|_{Y^{r}} \|w_{N_1}\|_{L^\infty_t H^r_x} \|z_{\sim N_1}\|_{L^\infty_t L^2_x}.
\end{align*}
For the estimates on $I_t^{2,\text{low}}$ and $I_t^{3,\text{low}}$, we benefit from the low frequency weight in the $ \bar{H}_x^r$-norm to get \footnote{It is worth noticing that this is these estimates that give   rise to  the conditions $c_4 $ and $ c_6 $ in \eqref{coco4}  and  which, combined with the condition $ \theta >\frac{11-4\alpha}{6} $ that already appears in Proposition \ref{h0}, lead to  our  restriction $ \alpha>\frac{55}{38}$ (see Subsection \ref{subsect6}).}
\begin{align*}
N^{-1} |I^{2,\text{low}}_t|
& \lesssim T^\frac12 N^{-1}  N^{\fr12} \max\Big\{\langle N\rangle^{\theta}(N N_1^\al)^{-1},\langle N\rangle^{\beta}(N N_1^\al)^{-\fr23}\Big\} \|w_{N_1}\|_{Y^{0}} \|w_{N}\|_{L^\infty_t L^2_x} \|z_{\sim N_1}\|_{L^\infty_t L^2_x} \\
&\lesssim   T^\frac12 N^{\fr12}  N_1^{-r} \max\Big\{N_1^{\theta}(N N_1^\al)^{-1},N_1^{\beta}(N N_1^\al)^{-\fr23}\Big\} \|w_{N_1}\|_{Y^{r}} \|w_N\|_{L^\infty_t  \bar{H}_x^r} \|z_{\sim N_1}\|_{L^\infty_t L^2_x}  .
\end{align*}
and similarly
\begin{align*}
N^{-1} | I^{3,\text{low}}_t|
&\lesssim   T^\frac12 N^{\fr12}  N_1^{-r} \max\Big\{N_1^{\theta}(N N_1^\al)^{-1},N_1^{\beta}(N N_1^\al)^{-\fr23}\Big\} \|z_{\sim N_1}\|_{Y^{0}} \|w_N\|_{L^\infty_t  \bar{H}_x^r} \|w_{N_1}\|_{L^\infty_t H^r_x}.
\end{align*}
Therefore, the contribution of the sum over $1\gtrsim N\geq N_1^{2r-}$ and $N_1 \gg 1$, can be bounded as follows
\begin{align}\label{J3-3}
&\sum_{N_1 \gg 1} \sum_{1\ge N\geq N_1^{2r-}} \langle N^{-1} \rangle^2  \langle N \rangle^{2r} N |I_t(w_N,w_{N_1},z_{\sim N_1})|
\nonumber\\
&\lesssim \sum_{N_1 \gg 1}\sum_{1\ge N\geq N_1^{2r-}} T^\frac12  N_1^{-r}N^{-\fr12}   \max\Big\{N_1^{\theta} N_1^{-\alpha} ,N_1^{\beta}N^{\frac13} N_1^{-\fr23 \alpha}, (N N_1^\al)^{-1},(N N_1^\al)^{-\fr23}\Big\} \|z\|_{Y^{0}} \|w\|_{Z^{r}}^2
\nonumber\\
&\lesssim \sum_{N_1 \gg 1}  T^\frac12 \max\Big\{N_1^{\theta-\alpha-2r}  ,N_1^{\beta-\frac23 \alpha-\frac43 r}, N_1^{-\alpha-4r},
N_1^{-\frac23 \alpha-\frac{10}{3}r}\Big\} N_1^{0+}\|z\|_{Y^{0}} \|w\|_{Z^{r}}^2
\nonumber\\
&\lesssim  T^\frac12\|z\|_{Y^{0}} \|w\|_{Z^{r}}^2,
\end{align}
under the additional conditions 
\begin{equation}\label{condi}
 \theta-\al-2r<0, \quad \beta-\frac23 \alpha-\frac{4}{3} r<0\quad \text{and} \quad  -5r<\alpha .
 \end{equation}

In the case $1 \ll N \lesssim N_1$,
we set $R={N_1}^{-(0+)}(N {N_1}^\al)$ so that   $R\ll N {N_1}^\al$ since $N_1\gg 1$.
We split $I_t$ as
\begin{align*}
I_t(w_{N_1},z_{\sim N_1},w_{N})=I_\infty(1^{\text{high}}_{t,R}w_{N_1},z_{\sim N_1},w_{N})+I_\infty(1^{\text{low}}_{t,R}w_{N_1},z_{\sim N_1},w_{N})=:I^{\text{high}}_t+I^{\text{low}}_t.
\end{align*}
The contribution of $I^{\text{high}}_t$ is estimated, thanks to Lemma \ref{le2-2} and H\"{o}lder's inequality, by
\begin{align*}
N^{1+2r} |I_t^{\text{high}}|&\lesssim N^{1+2r} N^{\fr12}\|1^{\text{high}}_{t,R}\|_{L^1} \|z_{N_1}\|_{L^\infty_t L^2_x} \|w_{N}\|_{L^\infty_t L^2_x} \|w_{N_1}\|_{L^\infty_t L^2_x}
\\&\lesssim T^{0+} N^{1+2r} N^{\fr12}N_1^{0+}(N N_1^{\al-})^{-1+} \|z_{N_1}\|_{L^\infty_t L^2_x} \|w_{N}\|_{L^\infty_t L^2_x} \|w_{N_1}\|_{L^\infty_t L^2_x} \\
&\lesssim  T^{0+} N^{r+\frac12} N_1^{-r-\alpha+}  
 \|z_{N_1}\|_{L^\infty_t L^2_x} \|w_{N}\|_{L^\infty_t H^r_x} \|w_{N_1}\|_{L^\infty_t H^r_x}
\end{align*}
that is acceptable since $ r>-1/2$ and $ \alpha>1/2$.
To evaluate the contribution $I_t^{\text{low}}$, we use the decomposition
\begin{align*}
I_\infty(1^{\text{low}}_{t,R} w_{N_1},z_{\sim N_1},w_{N})&=I_\infty(1^{\text{low}}_{t,R} Q_{\gtrsim N N_1^{\al}}w_{N},z_{\sim N_1},w_{N_1})
\\& \quad +I_\infty(1^{\text{low}}_{t,R} Q_{\ll N N_1^{\al}}w_{N},Q_{\gtrsim N N_1^{\al}}z_{\sim N_1},w_{N_1})
\\& \quad\quad +I_\infty(1^{\text{low}}_{t,R} Q_{\ll N N_1^{\al}}w_{N},Q_{\ll N N_1^{\al}}z_{\sim N_1},Q_{\sim N N_1^{\al}}w_{N_1})
\\&=: I^{1,\text{low}}_t+I^{2,\text{low}}_t+I^{3,\text{low}}_t.
\end{align*}
For the first term, by Lemma \ref{Cu2} we have 
\begin{align*}
&N^{1+2r} |I^{1,\text{low}}_t|
\lesssim N^{1+2r}  \|Q_{\gtrsim N N_1^{\al}}w_{N}\|_{L^{2+}_t L^2_x}
\|z_{\sim N_1}\|_{L^2_T L^\infty_x} \| w_{N_1}\|_{L^\infty_T L^2_x} \|1^{\text{low}}_{t,R}\|_{L^{\infty-}}
\\& \lesssim T^{0+} N_1^{\fr{2-\al}{3}} N^{1+2r}
\|z\|_{L^\infty_T L^2_x} (1+\|z\|_{L^2_T L^2_x}) \|w_{N_1}\|_{L^\infty_T L^2_x}  \max\Big\{N^{\theta}(N N_1^{\al})^{-1},
N^{\beta}(N N_1^{\al})^{-\fr23}\Big\}^{1-} \|w_{N}\|_{Y_T^{0}}
\\& \lesssim T^{0+} N_1^{\fr{2-\al}{3}} N_1^{-r} N^{1+r}
\|z\|_{L^\infty_T L^2_x} (1+\|z\|_{L^2_T L^2_x}) \|w_{N_1}\|_{L^\infty_T H^r_x}  \max\Big\{N^{\theta}(N N_1^{\al})^{-1},
N^{\beta}(N N_1^{\al})^{-\fr23}\Big\}^{1-}\|w_{N}\|_{Y_T^{r}}.
\end{align*}
To bound $I^{2,\text{low}}_t$,  we rewrite 
$Q_{\ll N N_1^{\al}}w_{ N}$ as
\begin{equation}\label{de1}
Q_{\ll N N_1^{\al}}w_{N}=w_{N}-Q_{\gtrsim N N_1^{\al}}w_{N}\, .
\end{equation}
To estimate the contribution of $w_{N}$, we  separate the contributions of $ N_1\gg N $ and $ N_1\sim N$. When
$ N_1\gg N $, Proposition \ref{bl2} leads to
\begin{align*}
&\quad  N^{1+2r} |I_\infty(1^{\text{low}}_{t,R}w_N,Q_{\gtrsim N N_1^\al}z_{\sim N_1},w_{N_1})|
\\
&\lesssim  N^{1+2r} \|Q_{\gtrsim N N_1^{\al}}z_{\sim N_1}\|_{L^{2+}_t L^2_x}
\|w_{N} w_{N_1}\|_{L^2_t L^2_x} 
\|1^{\text{low}}_{t,R}\|_{L^{\infty-}}
\\& \lesssim   T^{0+} N_1^{\fr{5-4\al}{6}} N_1^{-r} N^{1+r} \max\Big\{N_1^{\theta}(N N_1^{\al})^{-1},
N_1^{\beta}(N N_1^{\al})^{-\fr23}\Big\}^{1-}
\|w\|_{Y^{r}_T}^2 (1+\|z\|_{Y^0_T}^2)   \|z_{N_1}\|_{Y_T^{0}}
\\& \lesssim   T^{0+} N^{\fr{2-\al}{3}} N_1^{-r} N^{1+r} \max\Big\{N_1^{\theta}(N N_1^{\al})^{-1},
N_1^{\beta}(N N_1^{\al})^{-\fr23}\Big\}^{1-}
\|w\|_{Y^{r}_T}^2 (1+\|z\|_{Y^0_T}^2)  \|z_{N_1}\|_{Y_T^{0}},
\end{align*}
whereas when $ N_1\sim N$, making use of Proposition \ref{Cw2} we get
\begin{align*}
	&\quad  N^{1+2r} |I_\infty(1^{\text{low}}_{t,R}w_N,Q_{\gtrsim N N_1^\al}z_{\sim N_1},w_{N_1})|
	\\
	&\lesssim  N^{1+2r} \|Q_{\gtrsim N N_1^{\al}}z_{\sim N_1}\|_{L^{2+}_t L^2_x}
	\|w_{N} \|_{L^2_t L^\infty_x} \|w_{N_1}\|_{L^\infty_t L^2_x}
	\|1^{\text{low}}_{t,R}\|_{L^{\infty-}}
	\\& \lesssim   T^{0+} N^{\fr{2-\al}{3}} N_1^{-r} N^{1+r} \max\Big\{N_1^{\theta}(N N_1^{\al})^{-1},
	N_1^{\beta}(N N_1^{\al})^{-\fr23}\Big\}^{1-}
	\|w\|_{Y^{r}_T} (1+\|z\|_{Y^0_T}) \|w_{N_1}\|_{L^\infty_t H^r_x}  \|z_{N_1}\|_{Y_T^{0}}.
\end{align*}
On the other hand, proceeding as in \eqref{s5re1}, the contribution of $Q_{\gtrsim NN_1^{\al}}w_{N}$  can be bounded as follows
\begin{align*}
&\quad   N^{1+2r} |I_\infty(1^{\text{low}}_{t,R} Q_{\gtrsim N N_1^\al}w_N,Q_{\gtrsim N N_1^\al}z_{\sim N_1},w_{N_1})|
\\&\lesssim  T^{0+} 
N^{\frac{3}{2}+r} N_1^{-r}  \max\Big\{N^{\theta}(N N_1^\al)^{-1},N^{\beta}(N N_1^\al)^{-\fr23}\Big\}^{1-} \\
&\qquad\qquad \times 
 \max\Big\{N_1^{\theta}(N N_1^\al)^{-1},N_1^{\beta}(N N_1^\al)^{-\fr23}\Big\} \|w_{N}\|_{Y^{r}} \|z_{N_1}\|_{Y^0} \| w_{N_1}\|_{L^\infty_t H^r_x}
\\
&\lesssim  T^{0+} N_1^{\fr{2-\al}{3}} N_1^{-r} N^{1+r}\max\Big\{N^{\theta}(N N_1^{\al})^{-1},
N^{\beta}(N N_1^{\al})^{-\fr23}\Big\}^{1-}
 \|w_{N}\|_{Y^{r}} \|z_{N_1}\|_{Y^0} \| w_{N_1}\|_{L^\infty_t H^r_x}
\end{align*}
provided $ \theta\le \frac{2\alpha+2}{3} $ and $\beta\le \frac{2+\alpha}{3} $ that is verified for $ (\theta,\beta) $ satisfying \eqref{coco} and $ \alpha\le 2$.
To bound $I^{3,\text{low}}_t$,  we  use again the decomposition \eqref{de1}. 
To estimate the contribution of $w_{N}$, we  separate the contributions of $ N_1\gg N $ and $ N_1\sim N$ again. Then the contribution of  $w_{N}$ can be bounded as above by 
\begin{align*}
&\quad N^{1+2r} |I_\infty(1^{\text{low}}_{t,R}w_N,Q_{\ll N N_1^\al}z_{\sim N_1},Q_{\sim N N_1^{\al}} w_{N_1})|
\\
&\lesssim  N^{1+2r} \|Q_{\sim N N_1^{\al}}w_{\sim N_1}\|_{L^{2+}_t L^2_x}
\|w_{N} \|_{L^2_t L^\infty_x} \|z_{N_1}\|_{L^\infty_t L^2_x}
\|1^{\text{low}}_{t,R}\|_{L^{\infty-}}
\\& \lesssim  T^{0+} N^{\fr{2-\al}{3}}  N_1^{-r} 
N^{1+r} \max\Big\{N_1^{\theta}(N N_1^{\al})^{-1},
N_1^{\beta}(N N_1^{\al})^{-\fr23}\Big\}^{1-} 
\|w\|_{Y^{r}_T}^2 (1+\|z\|_{Y^0_T})  \|z_{N_1}\|_{Y_T^{0}}.
\end{align*}
whereas the contribution of  $Q_{\gtrsim NN_1^{\al}}w_{N}$  can be bounded as above by 
\begin{align*}
& \quad  N^{1+2r}  |I_\infty(1^{\text{low}}_{t,R} Q_{\gtrsim N N_1^\al}w_N,Q_{\ll N N_1^\al}z_{\sim N_1},Q_{\sim N N_1^\al}w_{N_1})|
\\
&\lesssim  T^{0+} 
N^{\frac{3}{2}+r} N_1^{-r}  \max\Big\{N^{\theta}(N N_1^\al)^{-1},N^{\beta}(N N_1^\al)^{-\fr23}\Big\}^{1-} \\
&\qquad\qquad \times 
 \max\Big\{N_1^{\theta}(N N_1^\al)^{-1},N_1^{\beta}(N N_1^\al)^{-\fr23}\Big\} \|w_{N}\|_{Y^{r}} \|z_{N_1}\|_{L^\infty_t L^2_x} \| w_{N_1}\|_{Y^r}
\\
&\lesssim  T^{0+} N_1^{\fr{2-\al}{3}} N_1^{-r} N^{1+r}\max\Big\{N^{\theta}(N N_1^{\al})^{-1},
N^{\beta}(N N_1^{\al})^{-\fr23}\Big\}^{1-}
 \|w_{N}\|_{Y^{r}} \|z_{N_1}\|_{L^\infty_t L^2_x} \| w_{N_1}\|_{Y^r}
\end{align*}
as soon as \eqref{coco} is satisfied. 
Gathering the above estimates we infer that for $ \alpha\le 2 $ and $(\theta,\beta) $ satisfying \eqref{coco},  the contribution $1 \ll N \lesssim N_1$ in $J_3$ can be controlled by
\begin{align}\label{J3-4}
J_3 &\lesssim  T^{0+}\sum_{N_1\gtrsim N \gg 1} N^{1+r} N_1^{-r}  \|w\|_{Y^{r}_T}^2 \|z\|_{Y^0_T} (1+\|z\|_{Y^0_T}^2) \\
&\qquad\times\Bigl( N^{\fr{2-\al}{3}} 
\max\Big\{N_1^{\theta}(N N_1^{\al})^{-1},
N_1^{\beta}(N N_1^{\al})^{-\fr23}\Big\}^{1-} +N_1^{\fr{2-\al}{3}}\max\Big\{N^{\theta}(N N_1^{\al})^{-1},
N^{\beta}(N N_1^{\al})^{-\fr23}\Big\}^{1-}\Bigr) 
\nonumber\\
 &\lesssim T^{0+} \|w\|_{Y^{r}_T}^2 \|z\|_{Y^0_T} (1+\|z\|_{Y^0_T}^2),
\end{align}
 provided $\alpha>1/2 $, $\theta<\frac{4\alpha-2}{3}$,  $\beta <\alpha-1$, $ \theta -\alpha-r<0$ and  $\beta-\frac23 \alpha -r <0$.
This ensures that the contribution of  $ J_3 $ is acceptable under the additional condition \eqref{condi}.

In conclusion, \eqref{prou1} holds for $\alpha\in ]\frac{37}{26},2] $, $r\in ]-\frac12,0[ $  and $ (\theta,\beta)$ satisfying \eqref{coco}, provided  \eqref{az1}, \eqref{az2}, \eqref{condr} and \eqref{condi} are satisfied. Thus, we obtain the conditions \eqref{coco2}-\eqref{coco3}.
\end{proof}

\section{Unconditional local well-posedness in $L^2(\R)$} \label{Sect6}
\subsection{Choice of $\theta,\beta $ and $ r$.} \label{subsect6}
In this subsection we  check that for $ \alpha\in ]\frac{55}{38} ,2] $ we can find   $\theta,\beta $ and $ r$ such that the conditions in  Proposition \ref{bl1}, Proposition \ref{bl2}, Proposition \ref{y0}, Proposition \ref{h0}, Proposition \ref{yr} and Proposition \ref{hr} are satisfied. For this we have to take $ \alpha\in ]\frac{37}{26},2] $ with $\theta,\beta $ and $ r$ satisfying 
\begin{equation}\label{coco4}
\begin{cases}
c_1 : \frac{11-4\alpha}{6}<\theta<\frac{5-\alpha}{3}  \\
c_2 :\frac{4-2\alpha}{3} <\beta<\min(\alpha-1,\frac29 (5-\alpha)) \\
c_3 :\max(-\alpha/5,\frac{2-2\alpha}{3})<r<\frac{3\al-5}{3} \; .
\end{cases}
\text{ and }
\begin{cases}
c_4 :\theta-2r<\alpha \\
c_5 :\beta-\frac{4}{3} r<\fr23\al \\
c_6 : \theta+r < \frac{10\alpha-11}{6}   \\
c_7 : \beta+r <\frac{8\alpha-11}{6}\\
c_8 : \theta-\beta<\frac{22\alpha-11}{18}\; .\\
\end{cases}
\end{equation}
First we notice that $c_4$ and $c_6 $ ensure that $ \theta<\frac19 (13\alpha-11) $ which combined with $c_1 $ forces $ \alpha >\frac{55}{38}>\frac{37}{26} $. Hence, we have to reduce the range of admissible $\alpha $ to $ ]\frac{55}{38} ,2]$.

Second, we notice that if we take $ \beta=\theta-\frac\alpha3 $ then $c_4\Rightarrow c_5 $, $ c_6\Rightarrow c_7 $ and $c_8 $ is fulfilled as soon as $ \alpha>\frac{11}{16}$. 

Now we  tackle the case $ \alpha\in [\frac32,2] $. In that case we will set 
\begin{equation}\label{choice1}
 \theta=\frac{4-\alpha}{3}+ , \quad \beta=\theta-\frac\alpha3 =\frac{4-2\alpha}{3}+\text{ and } r=\frac{\alpha-3}{6}  \; .
 \end{equation}
  We notice that $ \frac{11-4\alpha}{6} \le \frac{4-\alpha}{3} $ provided $ \alpha\ge \frac32 $ and thus $c_1$ is fulfilled. It is  also straightforward to check that $c_2$, $c_3 $ and $ c_4 $ are fulfilled provided $ \alpha>7/5$. Finally, $c_6$ is fulfilled provided $\alpha>\frac{16}{11} $ that ensures for $ \alpha\in [\frac32,2] $, all the conditions \eqref{coco4} for this choice of $(\theta,\beta,r) $. 

Finally in the case $\alpha\in ]\frac{55}{38},\frac32[ $, we set 
\begin{equation}\label{choice2}
 \theta=\frac{11-4\alpha}{6}+ , \quad \beta=\theta-\frac\alpha3=\frac{11-6\alpha}{6}+\text{ and } r=\frac{13\alpha-22}{11}  \; .
 \end{equation}
 Then  we can check that $ c_1$ is fulfilled provided $ \alpha>1/2 $, $c_2 $ is fulfilled provided $\alpha\in ]\frac{17}{12},\frac32[ $  whereas  $ c_3$, $ c_4 $ and $ c_6 $  are  fulfilled provided $ \alpha\in ]\frac{55}{38},\frac{11}{6}[$. Since $\frac{17}{12}<\frac{55}{38}$, this ensures for $ \alpha\in   ]\frac{55}{38},\frac32[ $, all the conditions \eqref{coco4} for this choice of $(\theta,\beta,r) $.

\subsection{Unconditional uniqueness}
Let $u_0\in L^2(\R)$  and let $u,v$ be two solutions to the Cauchy problem \eqref{gKdV} associated with the same initial datum  $u_0 $ that belong to $L^\infty_T L^2(\R)$ for some $ T>0$. According to Proposition \ref{y0} and Proposition \ref{h0}, setting $ T'=\min(1,T)$, we get that $u,v\in Y^0_{T'}$. Then applying Proposition \ref{yr} and Proposition \ref{hr}, we obtain that $ u-v\in  Z^r_{T'} $ and satisfies 
$$
\|u-v\|_{Z^r_{T'}} \lesssim T^{0+} C\, \|u-v\|_{Z^r_{T'}}\; ,
$$
for some $C= C(\|u\|_{L^\infty_T L^2_x}+\|v\|_{L^\infty_T L^2_x})>0$. 
Thus we have $u\equiv v$ on $[0,T^{''}]$ for some $ T^{''}=T^{''}(\|u\|_{L^\infty_T L^2_x}+\|v\|_{L^\infty_T L^2_x})>0$ and reiterating this argument a finite number of times it follows that $u\equiv v$ on $[0,T]$.

\subsection{Local existence and strong continuity in $L^2(\R)$}
%
Let $u_0\in H^\infty(\R) $.
It is well known (see \cite{ABFS89}) that \eqref{gKdV} is locally well-posed in $H^s(\R)$ for $s>\fr32$ with a minimal time of existence $\tilde{T}=\tilde{T}(\|u_0\|_{H^{\fr32+}})>0$.
So, there exists a solution $u\in C([-\tilde{T},\tilde{T}];H^\infty(\R))$ of \eqref{gKdV} emanating from the initial data $u_0$ with $\tilde{T}=\tilde{T}(\|u_0\|_{H^{\fr32+}})$. According to Proposition \ref{y0} and Proposition \ref{h0}, taking $s=0$ in \eqref{eqy0}, \eqref{eqh0} and $\omega\equiv 1$ in \eqref{eqh0},  it follows from a direct continuity argument that there exists $T=T(\|u_0\|_{L^2})>0 $ such that $\|u\|_{L^\infty_{T'}L^2}\lesssim \|u_0\|_{L^2}$ with $T'=\min\{T,\tilde{T}\}$. Using again \eqref{eqh0} with $\omega\equiv 1$, we deduce that $\|u\|_{L^\infty_{T'}H^s}\lesssim \|u_0\|_{H^s}$ for any $s\geq 0$. Iterating this argument, the above locally well-posed result then ensures that $u\in C([0,T];H^\infty(\R))$ with $\|u\|_{L^\infty_{T}H^s}\lesssim \|u_0\|_{H^s}$ for any $s\geq 0$.

Now let $u_0\in L^2(\R)$. Setting $u_{0,n}=P_{\leq n}u_0$, then the sequence $\{u_{0,n}\}_{n\geq 1}$ belongs to $H^\infty(\R)$ and
converges to $u_0$ in $L^2(\R)$.
According to the argument above, we deduce that the sequence $\{u_n\}_{n\geq 1}$ of solutions to \eqref{gKdV} emanating from  $\{u_{0,n}\}_{n\geq 1}$  satisfies $\{u_n\}_{n\geq 1}\subset C([-T,T];H^\infty(\R))$ with $\|u_n\|_{L^\infty_T L^2} \lesssim \|u_{0,n}\|_{L^2}\lesssim\|u_0\|_{L^2}$.
Note that $u_{0,n_1}-u_{0,n_2}\in \bar{H}^0(\R)$ for any $n_1,\ n_2\geq 1$. Therefore, we infer from Proposition \ref{yr} and Proposition \ref{hr} that $\{u_n\}_{n\geq 1}$ is a Cauchy sequence in $C([-T,T];H^{0-}(\R))$. Now,  we recall that, following \cite{KT03}, for any sequence $\{u_{0,k}\}_{k\geq 1}$ converging to $u_0$ in $L^2(\R)$, there exists a dyadic sequence $\{\omega_N\}$ satisfying Definition \ref{def1} for some $1<\delta<1+$ such that
\begin{align}\label{envelop}
\|u_0\|_{L^2_\omega}<\infty,\quad \sup_{k\geq 1}\|u_{0,k}\|_{L^2_\omega}<\infty \quad \text{and} \quad \omega_N \rightarrow +\infty.
\end{align}
Applying Proposition \ref{h0} with this sequence $\{\omega_N\}$, we obtain that the emanating solutions satisfy
\begin{align*}
\sup_{n\geq 1}\|u_{n}\|_{L^\infty(]0,T[;L^2_\omega)}\lesssim\sup_{n\geq 1}\|u_{0,n}\|_{L^2_\omega}<\infty,
\end{align*}
which ensures that $\{u_{n}\}_{n\geq 1}$ is actually a Cauchy sequence in $C([0,T];L^2(\R))$ and that its limit $ u \in C([0,T];L^2(\R))$ is the solution of \eqref{gKdV} emanating from $u_0 $.

\subsection{Continuity with respect to initial data}
We make use again of the frequency envelope $\{\omega_N\}$. Let $\{u_{0,k}\}_{k\geq 1} \subseteq L^2(\R)$, with $\sup_{k\geq 1}\|u_{0,k}\|_{L^2}\leq 2\|u_{0}\|_{L^2}$, that converges to $u_0$ in $L^2(\R)$.
Then there exists $\{\omega_N\}$ satisfying \eqref{envelop}. Therefore, applying Proposition \ref{h0} with this sequence $\{\omega_N\}$ and Proposition \ref{y0}, we infer that the sequence of solution $\{u_k\}_{k\geq 1}$ emanating from $\{u_{0,k}\}_{k\geq 1}$ is bounded in $L^\infty_T L^2_{\omega}$ for some $ T=T(\|u_0\|_{L^2})>0 $. More precisely, we have 
\begin{align}\label{con1}
\|u\|_{L^\infty([-T,T];L^2_\omega)}+\sup_{k\geq 1}\|u_{k}\|_{L^\infty([0,T];L^2_\omega)}\lesssim\sup_{k\geq 1}\|u_{0,k}\|_{L^2_\omega}<\infty.
\end{align}
Let $u^M$ and $u_k^M$ be the solution to \eqref{gKdV} with initial data $P_{\leq M}u_0$ and $P_{\leq M}u_{0,k}$.

By the triangle inequality, we obtain
\begin{align}\label{con2}
&\|u-u_k\|_{L^\infty_T L^2}\lesssim \|P_{\leq N}(u-u_k)\|_{L^\infty_T L^2}+ \|P_{\geq N} u\|_{L^\infty_T L^2}+\|P_{\geq N} u_k\|_{L^\infty_T L^2},
\end{align}
and
\begin{align}\label{con3}
\|P_{\leq N}(u-u_k)\|_{L^\infty_T L^2} &\lesssim
\|P_{\leq N}(u-u^M)\|_{L^\infty_T L^2}+\|P_{\leq N}(u_k-u_k^M)\|_{L^\infty_T L^2}+\|P_{\leq N}(u^M-u_k^M)\|_{L^\infty_T L^2}
\nonumber\\ &\lesssim N^{-r}\|u-u^M\|_{L^\infty_T H^r}+ N^{-r}\|u_k-u_k^M\|_{L^\infty_T H^r}+\|P_{\leq N}(u^M-u_k^M)\|_{L^\infty_T L^2}.
\end{align}
Fix $\epsilon>0$, by \eqref{con1}, there exists a $N>1$ such that
\begin{align}\label{con4}
\|P_{\geq N} u\|_{L^\infty_T L^2}+\sup_{k\geq 1} \|P_{\geq N} u_k\|_{L^\infty_T L^2}\lesssim \epsilon.
\end{align}
According to Proposition \ref{y0} and Proposition \ref{h0}, we know that $u^M, u_k^M \in Y^0_T$. Note that $u_{0}-P_{\leq M} u_{0},\ u_{0,k}-P_{\leq M} u_{0,k} \in \bar{L}^2(\R)$, by \eqref{con1}, Proposition \ref{yr} and Proposition \ref{hr}, for fix $N>1$, there exists a $M>1$ such that
\begin{align}\label{con5}
\|u-u^M\|_{L^\infty_T H^r}+ &\|u_k-u_k^M\|_{L^\infty_T H^r}
\lesssim \|u_0-P_{\leq M}u_0\|_{\bar{H}^r}+\|u_{0,k}-P_{\leq M}u_{0,k}\|_{\bar{H}^r}
\nonumber\\&\lesssim \|P_{>M}u_0\|_{L^2}+\sup_{k\geq 1}\|P_{>M}u_{0,k}\|_{L^2} \lesssim N^{r} \epsilon.
\end{align}
Finally, note that $P_{\leq M}u_0$ and $P_{\leq M}u_{0,k}$ belong to $H^\infty(\R)$ and $P_{\leq M}u_{0,k} \to P_{\leq M}u_0$ in any $ H^\theta(\R) $ with $\theta\ge 0 $. Therefore, according to the local well-posedness  result in  \cite{ABFS89} (or \cite{MV15}) , the emanating solutions satisfies $ u_k^M \to u^M $ in $C([0,T]; L^2(\R)) $ . In particular, there exists $ K\in \N $ such that for $ k\ge K $, 
\begin{align}\label{con6}
\|P_{\leq N}(u^M-u_k^M)\|_{L^\infty_T L^2} \lesssim \|u^M-u_k^M\|_{L^\infty_T L^2} \lesssim \epsilon.
\end{align}
By combining \eqref{con2}-\eqref{con6}, we deduce that for $ k\ge K $ we have
\begin{align*}
\|u-u_k\|_{L^\infty_T L^2}\lesssim \epsilon,
\end{align*}
which implies the continuity with respect to initial data of the flow-map.
\subsection{Global well-posedness}
In view of Hypothesis \ref{Hy1} on $ L_{\alpha+1} $, it is direct to check that the $L^2$-norm is conserved for smooth solutions to \eqref{gKdV}. The continuity with respect to initial data,  then ensures that this conservation law is shared by our local solution $ u\in C([-T,T]; L^2(\R)) $ with $ T=T(\|u_0\|_{L^2})>0 $. Therefore this solution can be extended for all $ T>0 $ and thus $u\in C(\R;L^2(\R)) $.

\section{Appendix}\label{Sect7}
In this last section we state and prove the refined linear and bilinear Strichartz estimates involving the difference $w=u-v$ of two solutions to \eqref{gKdV} which we made use in Section \ref{Sect5}. The main difference with the refined linear and bilinear Strichartz estimates on the solution $u $ is that, due to the negative Sobolev  regularity of $ w$ , we will need $ w$  and $ z=u+v$ to belong to some  Bourgain's space.
\begin{prop}\label{Cw2}
Let $\al\in [\fr75,2], 0<T\le 2, r<0$   and let $(w,z)\in Y^r_T \times Y^0_T $ satisfying
\begin{equation}\label{Eqw}
\pa_t w+L_{\al+1} w+\fr12 \pa_x(w z)=0
\end{equation}
on $]0,T[\times \R$ with $L_{\al+1}$ satisfying Hypothesis \ref{Hy1}. Then

For $N\geq 1$, we have
\begin{align*}
\|w_{N}\|_{L^2_T L^\infty_x}
\lesssim T^{\fr14} N^{\fr{2-\al}{3}} N^{-r}
\|w\|_{Y^{r}_T} (1+\|z\|_{Y^0_T}). 
\end{align*}
whenever the following conditions are fulfilled :\\
\begin{equation}\label{eq92}
r>\max\Bigl(\fr{4-5\al}6, \theta-\alpha,\beta-\frac23 \alpha\Bigr), \quad \theta<\frac{2+5\alpha}{6} 
\quad\text{and} \quad \beta<\alpha/2 \; .
\end{equation}
\end{prop}
\begin{proof}
We chop the time interval $[0,T]$ into small pieces of length $\sim N^{-\delta}T$, $\de=\fr{5-\al}{3}$, i.e., we define $\{I_{j,N}\}_{j\in J_{N}}$ with $\# J_{N}\sim N^\delta$ so that $\bigcup_{j\in J_{N}}=[0,T],\ |I_{j,N}|\sim N^{-\delta}T$. Then we notice that the first Strichartz estimate in  Lemma \ref{str} ensures that, for $v\in X^{0,\frac12,1}$, 
$$
\|P_N v\|_{L^4_T L^\infty_x} \lesssim N^{\frac{1-\alpha}{4}} \|P_N v \|_{X^{0,\frac12,1}} \; .
$$
Then according to the second  Strichartz estimate in Lemma \ref{str}  , we have for any $t_j\in I_{j,N} $, 
\begin{align}\label{w2}
	\Bigl\|\int_{t_j}^t e^{(t-t')L_{\al+1}}\pa_x &P_{N}(w z) dt'\Bigr\|_{L^4_{I_{j,N}} L^\infty_{x}}
	\lesssim  N^\frac{1-\alpha}{4} \Bigl( |I_{j,N}|^\frac34  N^\frac{1-\alpha}{4} N \| P_{\ll N}w P_{\sim N}z \|_{L^\infty_{I_{j,N}} L^1_{x}} \nonumber\\
	&+   N^\frac{1-\alpha}{4}  |I_{j,N}|^\frac34 N \|P_{\sim N}wP_{\ll N}z \|_{L^\infty_{I_{j,N}} L^1_{x}} + N  \Bigl\| \int_{t_j}^t e^{(t-t')L_{\alpha+1}} P_{N} (P_{\gtrsim N}w P_{\gtrsim N}z)\Bigr\|_{X^{0,\frac12,1}_{I_{j,N}}} \Bigr)
\end{align}
Now, similarly as in the proof of Lemma \ref{Cu2}, we have to bound the expression under the parentheses by $ \langle N\rangle^{-r} $. It is worth noticing that our choice of $ \delta $ leads to $ N^\frac{1-\alpha}{4}  |I_{j,N}|^\frac34 N\sim T^{\frac34}$.
We  estimate the terms in the above  right-hand side  one by one. It is easy to see that
\begin{align}
	 |I_{j,N}|^\frac34  N^\frac{1-\alpha}{4} N& \|P_{\sim N}w P_{\ll N}z\|_{L^\infty_{I_{j,N}}L^1_x}
	\lesssim \|P_{\sim N}w\|_{L^\infty_{I_{j,N}}L^2_x} \|P_{\ll N}z\|_{L^\infty_{I_{j,N}}L^2_x}
	\nonumber
	\\&\lesssim N^{-r}\|P_N w\|_{L^\infty_{I_{j,N}}H^r} \|z\|_{L^\infty_{T}L^2_x}
	\lesssim N^{-r}\|w\|_{L^\infty_{T}H^r} \|z\|_{L^\infty_{T}L^2_x}, \label{est1} 
\end{align}
and for $r\leq 0$
\begin{align}
	 |I_{j,N}|^\frac34  N^\frac{1-\alpha}{4} N&\|P_{\ll N}w P_{\sim N}z\|_{L^\infty_{I_{j,N}}L^1_x}
	\lesssim \|P_{\ll N}w\|_{L^\infty_{I_{j,N}}L^2_x} \|P_{\sim N}z\|_{L^\infty_{I_{j,N}}L^2_x}
	\nonumber
	\\&\lesssim N^{-r}\|w\|_{L^\infty_{I_{j,N}}H^r} \|z\|_{L^\infty_{T}L^2_x}
	\lesssim N^{-r}\|w\|_{L^\infty_{T}H^r} \|z\|_{L^\infty_{T}L^2_x}.\label{est2}
\end{align}
Now we deal with the last term:
\begin{align*}
	N  \Bigl\| \int_{t_j}^t e^{(t-t')L_{\alpha+1}} P_{N} &  (P_{\gtrsim N}w P_{\gtrsim N}z)\Bigr\|_{X^{0,\frac12,1}_{I_{j,N}}} \\
	& \lesssim N\sum_{N_1\gtrsim N}
	\Bigl\| \int_{t_j}^t e^{(t-t')L_{\alpha+1}} \sum_{L,L_1,L_2} 
	Q_L P_{N}(Q_{L_1} P_{N_1}w Q_{L_2}P_{N_1}z) dt'\Bigr\|_{X^{0,\fr12,1}_{I_{j,N}}}.
\end{align*}
Making use of  \eqref{Bour1} and  Lemma \ref{Cu2}, the contribution of the sum over $L\gtrsim N N_1^{\al}$ can be controlled by
\begin{align*}
	&\sum_{N_1\gtrsim N}\|\int_{t_j}^t e^{(t-t')L_{\al+1}} Q_{\gtrsim N N_1^{\al}} P_{N}(P_{N_1}w P_{N_1}z) dt'\|_{X^{0,\fr12,1}_{I_{j,N}}}
	\\&\lesssim\sum_{N_1\gtrsim N}\|Q_{\gtrsim N N_1^{\al}} P_{N}(P_{N_1}w P_{N_1}z) \|_{X^{0,-\fr12,1}_{I_{j,N}}}
	\\&\lesssim \sum_{N_1\gtrsim N}(N N_1^{\al})^{-\fr12}\| P_{N}(P_{N_1}w P_{N_1}z)\|_{L^2_{I_{j,N}}L^2_x}
	\\&\lesssim \sum_{N_1\gtrsim N}(N N_1^{\al})^{-\fr12}  N^\frac{1}{2} |I_{j,N}|^\frac{1}{2}\|P_{N_1}w\|_{L^\infty_{T}L^2_x} \|P_{N_1}z\|_{L^\infty_{T}L^2_x}
	\\&\lesssim \sum_{N_1\gtrsim N} N_1^{-\fr\al2} N_1^{-r} N^{\frac{\alpha-5}{6}} \|w\|_{L^\infty_{T}H^r}
	\|z\|_{L^\infty_{T}L^2_x} 
	\\&\lesssim N^{- \fr{5+2\al}6} N^{-r}\|w\|_{L^\infty_{T}H^r}\|z\|_{L^\infty_{T}L^2_x},
\end{align*}
where we use  that $-r-\fr{\al}{2}<0$.
The contribution of the region where $L\ll N N_1^{\al}$ and $L_1\gtrsim N N_1^{\al}$ is estimated thanks to \eqref{Bour1}  by
\begin{align*}
	&\sum_{N_1\gtrsim N}\|\int_{t_j}^t e^{(t-t')L_{\al+1}} Q_{\ll N N_1^{\al}} P_{N}(Q_{\gtrsim N N_1^{\al}} P_{N_1}w P_{N_1}z) dt'\|_{X^{0,\fr12,1}_{I_{j,N}}}
	\\&\lesssim\sum_{N_1\gtrsim N} |I_{j,N}|^{\fr12} \|Q_{\ll N N_1^{\al}} P_{N}(Q_{\gtrsim N N_1^{\al}} P_{N_1}w P_{N_1}z) \|_{L^2_{t,x}}
	\\&\lesssim\sum_{N_1\gtrsim N} T^{\fr12} N^{-\fr{\delta}{2}} N^{\fr12}
	\|Q_{\gtrsim N N_1^{\al}} P_{N_1}w\|_{L^2_{t,x}} \|P_{N_1}z\|_{L^\infty_t L^2_{x}}
	\\&\lesssim\sum_{N_1\gtrsim N} T^{\fr12} N^{-\fr{\delta}{2}} N^{\fr12}
	\max\Big\{N_1^{\theta-r}(N N_1^{\al})^{-1}
	,\ N_1^{\beta-r}(N N_1^{\al})^{-\fr23}\Big\} \|w\|_{Y^{r}}
	\|z\|_{L^\infty_t L^2_{x}}
	\\&\lesssim T^{\fr12} N^{-\fr{\delta}{2}}(N^{\theta-\al-\fr12}+N^{\beta-\fr23\al-\fr16}) N^{-r}
	\|w\|_{Y^{r}_T}
	\|z\|_{L^\infty_T L^2_{x}},
\end{align*}
where we use the fact that  $\theta-r-\al<0$ and $\beta-r-\fr23\al<0$. Finally the contribution of the last region can be bounded by
\begin{align*}
	&\sum_{N_1\gtrsim N}\|\int_{t_j}^t e^{(t-t')L_{\al+1}} Q_{\ll N N_1^{\al}} P_{N}(Q_{\ll N N_1^{\al}} P_{N_1}w Q_{\sim N N_1^{\al}} P_{N_1}z) dt'\|_{X^{0,\fr12,1}_{I_{j,N}}}
	\\&\lesssim\sum_{N_1\gtrsim N} |I_{j,N}|^{\fr12} \|Q_{\ll N N_1^{\al}} P_{N}(Q_{\ll N N_1^{\al}} P_{N_1}w Q_{\sim N N_1^{\al}} P_{N_1}z) \|_{L^2_{t,x}}
	\\&\lesssim\sum_{N_1\gtrsim N} T^{\fr12} N^{-\fr{\delta}{2}} N^{\fr12}
	\|P_{N_1}w\|_{L^\infty_t L^2_{x}} \|Q_{\sim N N_1^{\al}} P_{N_1}z\|_{L^2_{t,x}}
	\\&\lesssim\sum_{N_1\gtrsim N} T^{\fr12} N^{-\fr{\delta}{2}} N^{\fr12} N_1^{-r}\|w\|_{L^\infty_t H^r_{x}}
	\max\Big\{N_1^{\theta}(N N_1^{\al})^{-1}
	,\ N_1^{\beta}(N N_1^{\al})^{-\fr23}\Big\} \|z\|_{Y^{0}}
	\\&\lesssim T^{\fr12} N^{-\fr{\delta}{2}}(N^{\theta-\al-\fr12}+N^{\beta-\fr23\al-\fr16}) N^{-r}
	\|w\|_{L^\infty_T H^r_{x}}\|z\|_{Y^{0}_T},
\end{align*}
where we use the fact that $\theta-r-\al<0$ and $\beta-r-\fr23\al<0$ again.
Hence,
\begin{align}
	&N \Bigl\| \int_{t_j}^t e^{(t-t')L_{\al+1}}  \ P_{N}(P_{\gtrsim N}w P_{\gtrsim N}z) dt'\Bigr\|_{X^{0,\fr12,1}_{I_{j,N}}}
	\lesssim N N^{- \fr{5+2\al}6} N^{-r}\|w\|_{L^\infty_{T}H^r}\|z\|_{L^\infty_T L^2_{x}}\nonumber
	\\&+T^{\fr12} N^{-\fr{\delta}{2}} N (N^{\theta-\al-\fr12}
	+N^{\beta-\fr23\al-\fr16}) N^{-r}
	(\|w\|_{Y^{r}_T}
	\|z\|_{L^\infty_T L^2_{x}}+\|w\|_{L^\infty_T H^r_{x}}\|z\|_{Y^{0}_T})\nonumber
	\\
	&\lesssim N^{-r}(\|w\|_{L^\infty_{T}H^r}\|z\|_{L^\infty_{T}L^2_x}+\|w\|_{Y^{r}_T}
	\|z\|_{L^\infty_T L^2_{x}}+\|w\|_{L^\infty_T H^r_{x}}\|z\|_{Y^{0}_T})\nonumber
	\\&\lesssim N^{-r}(\|w\|_{Y^{r}_T}
	\|z\|_{L^\infty_T L^2_{x}}+\|w\|_{L^\infty_T H^r_{x}}\|z\|_{Y^{0}_T}),\label{est3}
\end{align}
where we use the fact that $\alpha\ge 1/2 $, $-\fr{5-\al}{6}+\theta-\al+\fr12\leq 0$ and $-\fr{5-\al}{6}+\beta-\fr23\al+\fr56\leq 0$ that is satisfied as soon as $ \alpha\ge \frac75$, $\theta\le \frac{2+5\alpha}{6} $ and $ \beta \le \alpha/2 $.

Therefore, 
\begin{align*}
	\|w_{N}\|_{L^2_T L^\infty_x}^2&=\sum_{j\in J_{N}}\|w_{N}\|_{L^2({I_{j,N}};L^\infty_x)}^2
	\\& \lesssim \sum_{j\in J_{N}} |I_{j,N}|^{\fr12} \|w_{N}\|_{L^4({I_{j,N}};L^\infty_x)}^2
	\\& \lesssim \sum_{j\in J_{N}} |I_{j,N}|^{\fr12}
	\Bigl( N^{\fr{1-\al}{2}}\|w_{N}\|_{L^\infty({I_{j,N}};L^2_x)}^2
	+N^{\fr{1-\al}{2}}\Bigl\|\int_{t_j}^t e^{(t-t')L_{\al+1}}\pa_x P_{N}(w z) dt'\Bigr\|_{L^4({I_{j,N}};L^\infty_x)}^2 \Bigr)
	\\& \lesssim \sum_{j\in J_{N}} |I_{j,N}|^{\fr12}
	\Bigl( N^{\fr{1-\al}{2}}\|w_{N}\|_{L^\infty({I_{j,N}};L^2_x)}^2
	+N^{\fr{1-\al}{2}} N^{-2r}(\|w\|_{Y^{r}_T}
	\|z\|_{L^\infty_T L^2_{x}}+\|w\|_{L^\infty_T H^r_{x}}\|z\|_{Y^{0}_T})^2 \Bigr)
	\\& \lesssim  (TN^{-\de})^{\fr12}
	N^\de N^{\fr{1-\al}{2}}
	\Bigl(N^{-2r}\|w_{N}\|_{L^\infty_T H^r_{x}}^2
	+N^{-2r}(\|w\|_{Y^{r}_T}
	\|z\|_{L^\infty_T L^2_{x}}+\|w\|_{L^\infty_T H^r_{x}}\|z\|_{Y^{0}_T})^2 \Bigr)
	\\& \lesssim T^{\fr12} N^{\fr{4-2\al}{3}} N^{-2r} \|w\|_{Y^{r}_T}^2 (1+\|z\|_{Y^0_T}^2),
\end{align*}
which leads to the desired result.
\end{proof}

\begin{prop}\label{bl1}
	Let $a\in L^\infty(\R^2)$ be such that $\|a\|_{L^\infty}\lesssim 1, \al\in [\fr75,2], N_1\gtrsim N_2, \, N_1 \geq 1, 0<T\le 2, r<0$ and let $(w_k,z_k)\in Y^r_T \times Y^0_T $ satisfying 
	\begin{equation}\label{Eqwk}
		\pa_t w_k+L_{\al+1} w_k+\fr12 \pa_x(w_k z_k)=0
	\end{equation}
	on $]0,T[\times \R$ for $k=1,2$, with $L_{\al+1}$ satisfying Hypothesis \ref{Hy1}. Then
	\begin{align}
		\|\Lambda_a(P_{N_1}w_1 ,P_{\lesssim N_2}w_2)\|_{L^2_{T,x}} 
		&\lesssim T^{\fr14}N_1^{\fr{2-\al}{3}}N_1^{-r} \langle N_2\rangle^{-r}\nonumber\\
		&\quad\quad \quad \times \prod_{i=1}^2 \Bigl(\|w_i\|_{L^\infty_{T} H^r_{x}}+\|w_i\|_{Y^{r}_T}
		\|z_i\|_{L^\infty_T L^2_{x}}+\|w_i\|_{L^\infty_T H^r_{x}}\|z_i\|_{Y^{0}_T}\Bigr)\label{re1}
	\end{align}
	and
	\begin{align}
		\|\Lambda_a(P_{N_1}w_1  ,P_{\lesssim N_2}z_2)\|_{L^2_{T,x}}
		&\lesssim T^{\fr14} N_1^{\fr{2-\al}{3}}N_1^{-r} (\|z_2\|_{L^\infty_T L^2_{x}}+\|z_2\|_{L^\infty_T L^2_{x}}^2)\nonumber \\
		&\quad\quad \times  (\|P_{N_1}w_1\|_{L^\infty_{T} H^r_{x}}+\|w_1\|_{Y^{r}_T}
		\|z_1\|_{L^\infty_T L^2_{x}}+\|w_1\|_{L^\infty_T H^r_{x}}\|z_1\|_{Y^{0}_T})\; .
		\label{re2}
	\end{align}
	whenever the conditions \eqref{eq92} are fulfilled.
\end{prop}
\begin{proof}
	We take $\delta=\frac{5-\alpha}{3} $ and we chop the time interval $[0,T]$ into small pieces of length $\sim N_1^{-\delta}T$, i.e., we define $\{I_{j}\}_{j\in J_{N_1}}$ with $\# J_{N_1}\sim N_1^\delta$ so that $\bigcup_{j\in J_{N_1}}=[0,T],\ |I_{j}|\sim N_1^{-\delta}T$. By frequency localization considerations, we decompose $P_{N}(w_kz_k) $ as 
	\begin{align*}
		P_{N}(w_kz_k)=P_{N_j} \Bigl( P_{\sim N_j}w_k P_{\ll N_j}z_k+P_{\ll N_j}w_k P_{\sim N_j}z_k+P_{\gtrsim N_j}w_k P_{\gtrsim N_j}z_k \Bigr).
	\end{align*}

	Then applying Propositions \ref{prov}-\ref{pro}, Christ-Kiselev arguments (see Corollaries 3.10 and 3.11 in \cite{MT3}), Bernstein inequality and Holder inequality in time   we have
	\begin{align}
		&\|\Lambda_a(P_{N_1}w_1,P_{\lesssim N_2}w_2)\|_{L^2({I_{j,N_1}};L^2)}^2\nonumber
		\\
		& \lesssim T^{\fr12}N_1^{-\fr{\al+\delta-1}{2}}  \Bigl( \|P_{N_1}w_1\|_{L^\infty_{I_{j}} L^2_{x}}^2+N_1^\frac{1-\alpha}{2}N_1^2 |I_{j}|^{\fr32} \| P_{N_1}(P_{\sim N_1}w_1 P_{\ll N_1}z_1+P_{\ll N_1}w_1 P_{\sim N_1}z_1 )\|_{L^\infty_{I_{j}} L^1_{x}}^2 \nonumber\\
		&\hspace*{3cm} + N_1^2 \Bigl\| \int_{t_j}^t e^{(t-t')L_{\alpha+1}} P_{N_1} (P_{\gtrsim N_1}w_1 P_{\gtrsim N_1}z_1)\Bigr\|_{X^{\frac12,0,1}_{I_{j}}}^2\Bigr)\nonumber \\
	&\quad \quad \quad  \times	 \Bigl( \|P_{\lesssim N_2}w_2\|_{L^\infty_{I_{j}} L^2_{x}}^2+\sum_{0<\tilde{N}_2\lesssim N_2} \tilde{N}_2^\frac{1-\alpha}{2}\tilde{N}_2^2 |I_{j}|^{\fr32} \| P_{\tilde{N}_2}(P_{\sim \tilde{N}_2}w_2 P_{\ll \tilde{N}_2}z_2+P_{\ll \tilde{N}_2}w_2 P_{\sim \tilde{N}_2}z_2 )\|_{L^\infty_{I_{j}} L^1_{x}}^2 \nonumber\\
		&\hspace*{3cm} + \tilde{N}_2^2 \Bigl\| \int_{t_j}^t e^{(t-t')L_{\alpha+1}} P_{\tilde{N}_2} (P_{\gtrsim \tilde{N}_2}w_2 P_{\gtrsim \tilde{N}_2}z_2)\Bigr\|_{X^{\frac12,0,1}_{I_{j}}}^2\Bigr)\nonumber \\
		& \lesssim T^{\fr12}N_1^{-\fr{\al+\delta-1}{2}} \Bigl( \|P_{N_1}w_1\|_{L^\infty_{I_{j}} L^2_{x}}^2+ \| P_{N_1}(P_{\sim N_1}w_1 P_{\ll N_1}z_1\|_{L^\infty_{I_{j}} L^1_{x}}^2+\| P_{\ll N_j}w_1 P_{\sim N_1}z_1 )\|_{L^\infty_{I_{j}} L^1_{x}}^2 \nonumber\\
		&\hspace*{3cm} + N_1^2 \Bigl\| \int_{t_j}^t e^{(t-t')L_{\alpha+1}} P_{N_1} (P_{\gtrsim N_1}w_k P_{\gtrsim N_1}z_k)\Bigr\|_{X^{\frac12,0,1}_{I_{j}}}^2\Bigr)\nonumber \\
		& \quad \quad \quad \times \Bigl[ \|P_{\lesssim N_2}w_2\|_{L^\infty_{I_{j}} L^2_{x}}^2+
		\sup_{0<\tilde{N}_2\lesssim N_2}  \Bigl(\| P_{\tilde{N}_2}(P_{\sim \tilde{N}_2}w_2 P_{\ll \tilde{N}_2}z_2\|_{L^\infty_{I_{j}} L^1_{x}}^2+\| P_{\ll \tilde{N}_2}w_2 P_{\sim \tilde{N}_2}z_2 )\|_{L^\infty_{I_{j}} L^1_{x}}^2  \nonumber\\
		&\hspace*{3cm} + N_2^2 \Bigl\| \int_{t_j}^t e^{(t-t')L_{\alpha+1}} P_{\tilde{N}_2} (P_{\gtrsim \tilde{N}_2}w_2 P_{\gtrsim \tilde{N}_2}z_2)\Bigr\|_{X^{\frac12,0,1}_{I_{j}}}^2\Bigr) \Bigr]\label{est0}
	\end{align}
	Then under the conditions \eqref{eq92} it follows from \eqref{est1}, \eqref{est2} and \eqref{est3} that 
	\begin{align}
		\|\Lambda_a(P_{N_1}w_1&,P_{N_2}w_2)\|_{L^2({I_{j}};L^2)}^2 \nonumber \\
		& \lesssim  T^{\fr12}N_1^{-\fr{\al+\delta-1}{2}} N_1^{-2r} \langle N_2\rangle^{-2r} \prod_{i=1}^2 \Bigl( \|w_i\|_{L^\infty_{I_{j}} L^2_{x}}^2+ \|w_i\|_{Y^{r}_T}
		\|z_i\|_{L^\infty_T L^2_{x}}+\|w_i\|_{L^\infty_T H^r_{x}}\|z_i\|_{Y^{0}_T})^2\Bigr)\, .
	\end{align}
	This yields \eqref{re1} by resumming over the small intervals $I_{j} $, $ j\in J_{N_1}$.  \eqref{re2} can be proven in the same way by estimating the contribution of $P_{\lesssim N_2} z $ as in Proposition \ref{sti}.
\end{proof}
\begin{prop}\label{bl2}
	Let $a\in L^\infty(\R^2)$ be such that $\|a\|_{L^\infty}\lesssim 1, \al\in [\fr75,2], N_1\gg N_2, N_1\geq 1, 0<T\le 2, r<0$  and let $(w_k,z_k)\in Z^r_T \times Y^0_T $ satisfying
	\begin{equation*}
		\pa_t w_k+L_{\al+1} w_k+\fr12 \pa_x(w_k z_k)=0
	\end{equation*}
	on $]0,T[\times \R$ for $k=1,2$, with $L_{\al+1}$ satisfying Hypothesis \ref{Hy1}. Then, assuming that the conditions \eqref{eq92} are fulfilled, it holds 
	\begin{align}
		\|\Lambda_a(P_{N_1}w_1,P_{\lesssim N_2}w_2)\|_{L^2_{T,x}}
		&\lesssim N_1^{\fr{5-4\al}{6}}N_1^{-r}  \langle N_2\rangle^{-r}(\|w_1\|_{L^\infty_{T} H^r_{x}}+\|w_1\|_{Y^{r}_T}
		\|z_1\|_{L^\infty_T L^2_{x}}+\|w_1\|_{L^\infty_T H^r_{x}}\|z_1\|_{Y^{0}_T})\nonumber \\
		&\quad\times(\|w_2\|_{L^\infty_{T} H^r_{x}}+\|w_2\|_{Z^{r}_T}
		\|z_2\|_{L^\infty_T L^2_{x}}+\|w_2\|_{L^\infty_T H^r_{x}}\|z_2\|_{Y^{0}_T}),\label{Est1}
	\end{align}
	\begin{align}
		\|\Lambda_a(P_{N_1}z_1,P_{\lesssim N_2}w_2)\|_{L^2_{T,x}}
		&\lesssim N_1^{\fr{5-4\al}{6}} \langle N_2\rangle^{-r} (\|w_2\|_{L^\infty_{T} H^r_{x}}+\|w_2\|_{Z^{r}_T}
		\|z_2\|_{L^\infty_T L^2_{x}}+\|w_2\|_{L^\infty_T H^r_{x}}\|z_2\|_{Y^{0}_T})\nonumber\\
		&\quad\times(\|z_1\|_{L^\infty_T L^2_{x}}+\|z_1\|_{L^\infty_T L^2_{x}}^2).\label{Est2}
	\end{align}
	and
	\begin{align}
		\|\Lambda_a(P_{N_1}w_1,P_{\lesssim N_2}z_2)\|_{L^2_{T,x}}
		&\lesssim N_1^{\fr{5-4\al}{6}}N_1^{-r} (\|w_1\|_{L^\infty_{T} H^r_{x}}+\|w_1\|_{Z^{r}_T}
		\|z_1\|_{L^\infty_T L^2_{x}}+\|w_1\|_{L^\infty_T H^r_{x}}\|z_1\|_{Y^{0}_T})\nonumber\\
		&\quad\times(\|z_2\|_{L^\infty_T L^2_{x}}+\|z_2\|_{L^\infty_T L^2_{x}}^2)\, .\label{Est3}
	\end{align}
		Moreover, for $ N_1^{-2/3} \lesssim N_2\lesssim 1 \ll N_1$, 
		\begin{align}
			\|\Lambda_a(P_{N_1}w_1,P_{ N_2}w_2)\|_{L^2_{T,x}}
			&\lesssim N_1^{\fr{5-4\al}{6}}N_1^{-r}  N_2 (\|w_1\|_{L^\infty_{T} H^r_{x}}+\|w_1\|_{Y^{r}_T}
			\|z_1\|_{L^\infty_T L^2_{x}}+\|w_1\|_{L^\infty_T H^r_{x}}\|z_1\|_{Y^{0}_T})\nonumber \\
			&\quad\times(\|w_2\|_{L^\infty_{T} H^r_{x}}+\|w_2\|_{Z^{r}_T}
			\|z_2\|_{L^\infty_T L^2_{x}}+\|w_2\|_{L^\infty_T H^r_{x}}\|z_2\|_{Y^{0}_T}),\label{Est11}
		\end{align}
		and
		\begin{align}
			\|\Lambda_a(P_{N_1}z_1,P_{N_2}w_2)\|_{L^2_{T,x}}
			&\lesssim N_1^{\fr{5-4\al}{6}} N_2 (\|w_2\|_{L^\infty_{T} H^r_{x}}+\|w_2\|_{Z^{r}_T}
			\|z_2\|_{L^\infty_T L^2_{x}}+\|w_2\|_{L^\infty_TH^r_{x}}\|z_2\|_{Y^{0}_T})\nonumber\\
			&\quad\times(\|z_1\|_{L^\infty_T L^2_{x}}+\|z_1\|_{L^\infty_T L^2_{x}}^2).\label{Est22}
		\end{align}
\end{prop}
\begin{proof}
	\eqref{Est1}-\eqref{Est3} can be proven in  exactly the same way as \eqref{re1}-\eqref{re2}  by making use of \eqref{prop2} instead of \eqref{prop1}. It thus just suffices to prove \eqref{Est11}-\eqref{Est22} that concerns some interactions with very low frequencies of the difference $ w$ of two solutions to \eqref{gKdV}.

	Again we take $ \delta=\fr{5-\al}{3} $ and  chop the time interval $[0,T]$ into small pieces of length $\sim N_1^{-\delta}T$, i.e., we define $\{I_{j}\}_{j\in J_{N_1}}$ with $\# J_{N_1}\sim N_1^\delta$ so that $\bigcup_{j\in J_{N_1}}=[0,T],\ |I_{j}|\sim N_1^{-\delta}T$.
	Then applying Proposition \ref{pro} and proceeding as in \eqref{est0} we have
	\begin{align}
		\|\Lambda_a &(P_{N_1}w_1,P_{N_2}w_2)\|_{L^2({I_{j}};L^2)}^2 \nonumber \\
		&\lesssim N_1^{-\al}  \Bigl( \|P_{N_1}w_1\|_{L^\infty_{I_{j}} L^2_{x}}^2+ \| P_{N_1}(P_{\sim N_1}w_1 P_{\ll N_1}z_1)\|_{L^\infty_{I_{j}} L^1_{x}}^2+\| P_{N_1}(P_{\ll N_1}w_1P_{\sim N_1}z_1)\|_{L^\infty_{I_{j}} L^1_{x}}^2 \nonumber\\
		&\hspace*{3cm} + N_1^2 \Bigl\| \int_{t_j}^t e^{(t-t')L_{\alpha+1}} P_{N_j} (P_{\gtrsim N_1}w_1 P_{\gtrsim N_1}z_1)\Bigr\|_{X^{\frac12,0,1}_{I_{j}}}^2\Bigr) \nonumber \\
		&\times \Bigl(  \|P_{N_2}w_2\|_{L^\infty_{I_{j}} L^2_{x}}^2+ N_2^2 \| P_{N_2}(P_{\sim N_2}w_2 
		P_{\lesssim  N_2}z_2)\|_{L^1_{I_{j}} L^2_{x}}^2+N_2^2 \| P_{N_2}(P_{\lesssim  N_2}w_2 P_{\sim N_2}z_2)\|_{L^1_{I_{j}} L^2_{x}}^2 \nonumber\\
		& + N_2^2 \Bigl\| P_{N_2} (\sum_{\tilde{N}_2\in \Theta_{N_2}} P_{\tilde{N}_2} w_2 P_{\sim\tilde{N}_2}z_2)\Bigr\|_{L^1_{I_{j}} L^2_{x}}^2
		+N_2^2 \Bigl\| \int_{t_j}^t e^{(t-t')L_{\alpha+1}} P_{N_2} (\hspace*{-4mm}\sum_{\tilde{N}_2\gg1\atop \tilde{N}_2^{-2r}>N_2^{-1}}
		P_{\tilde{N}_2} w_2 P_{\sim\tilde{N}_2} z_2)\Bigr\|_{X^{\frac12,0,1}_{I_{j}}}^2\Bigr) \label{est00} 
	\end{align}
	where 
	$$
	\Theta_{N_2}=\Bigl\{\tilde{N}_2\in 2^\Z, N_2\ll \tilde{N}_2\lesssim 1 \text{ or } (\tilde{N}_2\gg 1 \text{ and } 
	\tilde{N}_2^{-2r} <N_2^{-1}) \Bigr\} \; .
	$$
	In view of \eqref{est1}-\eqref{est3}, we just need to estimate the contribution of the terms in the last parenthesis in the above right-hand side member. First, since $N_2\lesssim 1$,  it is straigtforward to check that 
	\begin{align*}
		&\|P_{\sim N_2}w_2 P_{\lesssim N_2}z_2\|_{L^1_{I_{j}}L^2_x}
		\lesssim  N_2^{\fr12}\|P_{\sim N_2}w_2\|_{L^\infty_{I_{j}}L^2_x} \|P_{\lesssim N_2}z_2\|_{L^\infty_{I_{j}}L^2_x}
		\lesssim \|w_2\|_{L^\infty_{T} H^r} \|z_2\|_{L^\infty_{T}L^2_x},
	\end{align*}
	and
	\begin{align*}
		&\|P_{\lesssim N_2}w_2 P_{\sim N_2}z_2\|_{L^1_{I_{j,N_1}}L^2_x}
		\lesssim N_2^{\fr12} \|P_{\lesssim N_2}w_2\|_{L^\infty_{I_{j}}L^2_x} \|P_{\sim N_2}z_2\|_{L^\infty_{I_{j}}L^2_x}
		\lesssim  \|w_2\|_{L^\infty_{T} H^r} \|z_2\|_{L^\infty_{T}L^2_x}.
	\end{align*}
	Now we deal with the contribution of the sum over $ \tilde{N}_2\in \Theta_{N_2} $. 
	For  the contribution of the sum over $\tilde{N}_2\lesssim 1$ we have
	\begin{align*}
		\sum_{N_2\ll\tilde{N}_2\lesssim 1} \|P_{N_2}(P_{\tilde{N}_2}w_2 P_{\tilde{N}_2}z_2) \|_{L^1_{I_{j}}L^2_x}&\lesssim \sum_{N_2\ll\tilde{N}_2\lesssim 1}  N_2^{\fr12}  \|P_{\tilde{N}_2}w_2\|_{L^2_{I_{j}} L^2_x} \|P_{\tilde{N}_2}z_2\|_{L^2_{I_{j}} L^2_x}
		\\&\lesssim \|w_2\|_{L^2_{I_{j}} H^{r}_x} \|z_2\|_{L^2_{I_{j}} L^2_x},
	\end{align*}
	whereas for the contribution of the sum over  ($\tilde{N}_2 \gg 1$ and $ \tilde{N}_2^{-2r}<N_2^{-1}$)  we have 
	\begin{align*}
		\sum_{\tilde{N}_2 \gg 1,\ \tilde{N}_2^{2r}>N_2} \|P_{N_2}(P_{\tilde{N}_2}w_2 P_{\tilde{N}_2}z_2) \|_{L^1_{I_{j}}L^2_x}
		&\lesssim \sum_{\tilde{N}_2 \gg 1,\ \tilde{N}_2^{-2r}< N_2^{-1}}  N_2^{\fr12}  \|P_{\tilde{N}_2}w_2\|_{L^2_{I_{j}} L^2_x} \|P_{\tilde{N}_2}z_2\|_{L^2_{I_{j}} L^2_x}
		\\&\lesssim \sum_{\tilde{N}_2 \gg 1,\ \tilde{N}_2^{-2r}< N_2^{-1}}
		N_2^{\fr12} \langle\tilde{N}_2\rangle^{-r} 
		\|P_{\tilde{N}_2}w_2\|_{L^2_{I_{j}} H^{r}_x} \|P_{\tilde{N}_2}z_2\|_{L^2_{I_{j}} L^2_x}
		\\&\lesssim \sum_{\tilde{N}_2 \gg 1,\ \tilde{N}_2^{-2r}< N_2^{-1}} N_2^{\fr12} \tilde{N}_2^{-r}
		\|P_{\tilde{N}_2}w_2\|_{L^2_{I_{j}} H^{r}_x} \|P_{\tilde{N}_2}z_2\|_{L^2_{I_{j}} L^2_x}
		\\&\lesssim \|w_2\|_{L^2_{I_{j}} H^{r}_x} \|z_2\|_{L^2_{I_{j}} L^2_x}.
	\end{align*}
	It remains to bound the  contribution of the sum over ($\tilde{N}_2 \gg 1$ and $\tilde{N}_2^{-2r}\ge N_2^{-1}$). To do this we first write 
	\begin{align*}
		&\sum_{\tilde{N}_2 \gg 1,\ \tilde{N}_2^{-2r}\ge N_2^{-1}} \|\int_{t_j}^t e^{(t-t')L_{\al+1}} P_{N_2}(P_{\tilde{N}_2}w_2 P_{\tilde{N}_2}z_2) dt'\|_{X^{0,\fr12,1}_{I_{j}}}
		\\&\lesssim  \sum_{\tilde{N}_2 \gg 1,\ \tilde{N}_2^{-2r}\ge N_2^{-1}} \|\sum_{L,L_1,L_2}Q_L P_{N_2}(Q_{L_1} P_{\tilde{N}_2}w_2 Q_{L_2}P_{\tilde{N}_2}z_2) \|_{X^{0,-\fr12,1}_{I_{j}}}.
	\end{align*}
	The contribution of the sum over $L\gtrsim N_2 \tilde{N}_2^{\al}$ can be controlled by
	\begin{align*}
		&\sum_{\tilde{N}_2 \gg 1,\ \tilde{N}_2^{-2r}\ge N_2^{-1}}\|Q_{\gtrsim N_2 \tilde{N}_2^{\al}} P_{N_2}(P_{\tilde{N}_2}w_2 P_{\tilde{N}_2}z_2) \|_{X^{0,-\fr12,1}_{I_{j}}}
		\\&\lesssim \sum_{\tilde{N}_2 \gg 1,\ \tilde{N}_2^{-2r}\ge N_2^{-1}}(N_2 \tilde{N}_2^{\al})^{-\fr12}\|P_{N_2}(P_{\tilde{N}_2}w_2 P_{\tilde{N}_2}z_2)\|_{L^2_{I_{j}}L^2_x}
		\\&\lesssim \sum_{\tilde{N}_2 \gg 1,\ \tilde{N}_2^{-2r}\ge N_2^{-1}}(N_2 \tilde{N}_2^{\al})^{-\fr12}N_2^{\fr12}\|P_{\tilde{N}_2}w_2\|_{L^2_{I_{j}}L^2_x} \|P_{\tilde{N}_2}z_2\|_{L^\infty_{I_{j}}L^2_x}
		\\&\lesssim \sum_{\tilde{N}_2 \gg 1,\ \tilde{N}_2^{-2r}\ge N_2^{-1}}(N_2 \tilde{N}_2^{\al})^{-\fr12} N_2^{\fr12}\tilde{N}_2^{-r}
		\|P_{\tilde{N}_2}w_2\|_{L^2_{I_{j}} \bar{H}^{r}_x} \|P_{\tilde{N}_2}z_2\|_{L^\infty_{I_{j}}L^2_x}
		\\&\lesssim \|w_2\|_{L^2_{I_{j}} H^{r}_x} \|z_2\|_{L^\infty_{I_{j}} L^2_x},
	\end{align*}
	where we use the fact that $-\fr12\al-r<0$ that is automatically satisfied for $\alpha\ge 7/5 $ satisfying \eqref{eq92}.
	To bound the contribution of the region where $L\ll N_2 \tilde{N}_2^{\al}$ and $L_1\gtrsim N_2 \tilde{N}_2^{\al}$ we use that
	$N_1^{-2/3}\lesssim N_2 $ and \eqref{Bour1} to get 
	\begin{align*}
		&\sum_{\tilde{N}_2 \gg 1} \sum_{ \tilde{N}_2^{2r}\vee N_1^{-2/3} \le N_2 \lesssim 1}\|Q_{\ll N_2 \tilde{N}_2^{\al}} P_{N_2}(Q_{\gtrsim N_2 \tilde{N}_2^{\al}} P_{\tilde{N}_2}w_2 P_{\tilde{N}_2}z_2)\|_{X^{0,-\fr12,1}_{I_{j}}}
		\\
		& \lesssim \sum_{\tilde{N}_2 \gg 1} \sum_{ \tilde{N}_2^{2r}\vee N_1^{-2/3} \le N_2 \lesssim 1} N_1^\frac{\alpha-5}{6}\|Q_{\ll N_2 \tilde{N}_2^{\al}} P_{N_2}(Q_{\gtrsim N_2 \tilde{N}_2^{\al}} P_{\tilde{N}_2}w_2 P_{\tilde{N}_2}z_2) \|_{L^2_{I_{j}}L^2_x}\\
		&\lesssim\sum_{\tilde{N}_2 \gg 1,0<N_2\lesssim 1}  N_2^{\fr12}N_2^\frac{5-\alpha}{4}
		\max\Big\{\tilde{N}_2^{\theta-r}(N_2 \tilde{N}_2^{\al})^{-1}
		,\ \tilde{N}_2^{\beta-r}(N_2 \tilde{N}_2^{\al})^{-\fr23}\Big\} \|w_2\|_{Y^{r}_T}
		\|P_{\tilde{N}_2}z_2\|_{L^\infty_{I_{j}} L^2_{x}} 
		\\
		&\lesssim\sum_{\tilde{N}_2 \gg 1,0<N_2\lesssim 1}  
		N_2^\frac{3-\alpha}{4}\max\Big\{\tilde{N}_2^{\theta-r-\alpha}
		,\ \tilde{N}_2^{\beta-r-\fr23 \alpha}\Big\} \|w_2\|_{Y^{r}_T}
		\|P_{\tilde{N}_2}z_2\|_{L^\infty_{I_{j}} L^2_{x}}\\
		&\lesssim
		\|w_2\|_{Z^{r}_T}
		\|z_2\|_{L^\infty_{I_{j}} L^2_{x}},
	\end{align*}
	where we use  that \eqref{eq92} forces $\theta-\al-r<0$ and $\beta-\fr23\al- r<0$. Finally the contribution of the last region can be bounded in the same way by 
	\begin{align*}
		&\sum_{\tilde{N}_2 \gg 1} \sum_{ \tilde{N}_2^{2r}\vee N_1^{-2/3} \le N_2 \lesssim 1}\|Q_{\ll N_2 \tilde{N}_2^{\al}} P_{N_2}(Q_{\ll N_2 \tilde{N}_2^{\al}} P_{\tilde{N}_2}w_2 Q_{\sim N_2 \tilde{N}_2^{\al}} P_{\tilde{N}_2}z_2) \|_{X^{0,-\fr12,1}_{I_{j}}}
		\\
		&\lesssim\sum_{\tilde{N}_2 \gg 1} \sum_{ \tilde{N}_2^{2r}\vee N_1^{-2/3} \le N_2 \lesssim 1}N_1^\frac{\alpha-5}{6}\|Q_{\ll N_2 \tilde{N}_2^{\al}} P_{N_2}(Q_{\ll N_2 \tilde{N}_2^{\al}} P_{\tilde{N}_2}w_2 Q_{\sim N_2 \tilde{N}_2^{\al}} P_{\tilde{N}_2}z_2) \|_{L^2_{I_{j}}L^2_x}
		\\
		&\lesssim\sum_{\tilde{N}_2 \gg 1,0<N_2\lesssim 1}  N_2^{\fr12}N_2^\frac{5-\alpha}{4}\tilde{N}_2^{-r}
		\max\Big\{\tilde{N}_2^{\theta}(N_2 \tilde{N}_2^{\al})^{-1}
		,\ \tilde{N}_2^{\beta}(N_2 \tilde{N}_2^{\al})^{-\fr23}\Big\} \|z_2\|_{Y^{0}_T}
		\|P_{\tilde{N}_2}w_2\|_{L^\infty_{I_{j}} H^r_{x}}
		\\&\lesssim
		\|w_2\|_{L^\infty_{I_{j}} H^r_{x}}
		\|z_2\|_{Y^{0}_T},
	\end{align*}
	where we use again that $\theta-\al-r<0$ and $\beta-\fr23\al- r<0$.
	
	Therefore, coming back to \eqref{est00} we finally obtain 
	\begin{align*}
		\|\Lambda_a(P_{N_1}w_1,P_{N_2}w_2)\|_{L^2({I_{j}};L^2)}^2
		&\lesssim N_1^{-\al}N_1^{-2r} N_2^{2}(\|P_{N_1}w_1\|_{L^\infty_{I_{j}} H^r_{x}}^2+\|w_1\|_{Y^{r}_T}^2
		\|z_1\|_{L^\infty_T L^2_{x}}^2+\|w_1\|_{L^\infty_T H^r_{x}}^2\|z_1\|_{Y^{0}_T}^2)\\
		&\quad\times(\|P_{N_2}w_2\|_{L^\infty_{I_{j}} H^r_{x}}^2+\|w_2\|_{Z^{r}_T}^2
		\|z_2\|_{L^\infty_T L^2_{x}}^2+\|w_2\|_{L^\infty_T H^r_{x}}^2\|z_2\|_{Y^{0}_T}^2).
	\end{align*}

	Then,
	\begin{align*}
		&\|\Lambda_a(P_{N_1}w_1,P_{N_2}w_2)\|_{L^2_{T,x}}^2=\sum_{j\in J_{N_1}}
		\|\Lambda_a(P_{N_1}w_1,P_{N_2}w_2)\|_{L^2({I_{j}};L^2)}^2
		\\&\lesssim \sum_{j\in J_{N_1}}N_1^{-\al}N_1^{-2r} N_2^{2}(\|P_{N_1}w_1\|_{L^\infty_{I_{j}} H^r_{x}}^2+\|w_1\|_{Y^{r}_T}^2
		\|z_1\|_{L^\infty_T L^2_{x}}^2+\|w_1\|_{L^\infty_T H^r_{x}}^2\|z_1\|_{Y^{0}_T}^2)\\
		&\quad\times(\|P_{N_2}w_2\|_{L^\infty_{I_{j}} H^r_{x}}^2+\|w_2\|_{Z^{r}_T}^2
		\|z_2\|_{L^\infty_T L^2_{x}}^2+\|w_2\|_{L^\infty_T H^r_{x}}^2\|z_2\|_{Y^{0}_T}^2)
		\\&\lesssim N_1^{\delta-\al}N_1^{-2r} N_2^{2}(\|P_{N_1}w_1\|_{L^\infty_{T} H^r_{x}}^2+\|w_1\|_{Y^{r}_T}^2
		\|z_1\|_{L^\infty_T L^2_{x}}^2+\|w_1\|_{L^\infty_T H^r_{x}}^2\|z_1\|_{Y^{0}_T}^2)\\
		&\quad\times(\|P_{N_2}w_2\|_{L^\infty_{T} H^r_{x}}^2+\|w_2\|_{Z^{r}_T}^2
		\|z_2\|_{L^\infty_T L^2_{x}}^2+\|w_2\|_{L^\infty_T H^r_{x}}^2\|z_2\|_{Y^{0}_T}^2)
		\\&\lesssim  N_1^{\fr{5-4\al}{3}}
		N_1^{-2r} N_2^{2}(\|P_{N_1}w_1\|_{L^\infty_{T} H^r_{x}}^2+\|w_1\|_{Y^{r}_T}^2
		\|z_1\|_{L^\infty_T L^2_{x}}^2+\|w_1\|_{L^\infty_T H^r_{x}}^2\|z_1\|_{Y^{0}_T}^2)\\
		&\quad\times(\|P_{N_2}w_2\|_{L^\infty_{T} H^r_{x}}^2+\|w_2\|_{Z^{r}_T}^2
		\|z_2\|_{L^\infty_T L^2_{x}}^2+\|w_2\|_{L^\infty_T H^r_{x}}^2\|z_2\|_{Y^{0}_T}^2),
	\end{align*}
	that proves \eqref{Est1}.
	\eqref{Est2} can be proven  in the same way and \eqref{Est3}  can be proved in the same way as Proposition \ref{bl0} and Proposition \ref{bl1}. This completes the proof.
\end{proof}

\begin{lemm}\label{Cwv}Let $\alpha\in [\frac{7}{5},2] $, $0<T \le 2$, $v\in X^{0,\fr12,1}_T $and $w \in Y^r_T$ satisfying
	\eqref{Eqw} on $]0,T[\times \R$  with $L_{\al+1}$ satisfying Hypothesis \ref{Hy1}. Then
	for $N_1 \gg N_2\ge 1 $ it holds
	\begin{align}
		\|w_{N_1}v_{N_2}\|_{L^2_T L^2_x} &
		\lesssim  N_1^{\fr{5-4\al}{6}} N_1^{-r} \|w\|_{Y^r_T} (1+\|z\|_{Y^0_T}) \|v_{N_2}\|_{X^{0,\fr12,1}_T},
		\label{zd}
	\end{align}
	whenever the conditions \eqref{eq92} are fulfilled.
\end{lemm}
\begin{proof}
	Similar with the proof of Proposition \ref{bl1}, decomposing $ [0,T] $ in small intervals $I_{j} $ with $ |I_{j}| \sim N_1^{-\frac{5-\alpha}{3}} T $,  we have
	\begin{align*}
		\|w_{N_1}v_{N_2}\|_{L^2_{T,x}}^2&=\sum_{j\in J_{N_1}}\|w_{N_1}v_{N_2}\|_{L^2(I_{j};L^2)}^2
		\lesssim  \sum_{j\in J_{N_1}} N_1^{-\al}\|v_{N_2}\|_{X^{0,\fr12,1}_T}^2 \nonumber \\
		&\times \Bigl( \|P_{N_1}w\|_{L^\infty_{I_{j}} L^2_{x}}^2+ \| P_{N_1}(P_{\sim N_1}w P_{\ll N_1}z\|_{L^\infty_{I_{j}} L^1_{x}}^2+\| P_{\ll N_1}w P_{\sim N_1}z )\|_{L^\infty_{I_{j}} L^1_{x}}^2 \nonumber\\
		&\hspace*{3cm} + N_1^2 \Bigl\| \int_{t_j}^t e^{(t-t')L_{\alpha+1}} P_{N_1} (P_{\gtrsim N_1}w P_{\gtrsim N_1}z)\Bigr\|_{X^{\frac12,0,1}_{I_{j}}}^2\Bigr)\\
		&\lesssim   N_1^{\fr{5-4\al}{3}} N_1^{-2r}\|w\|_{Y^r_T}^2 (1+\|z\|_{Y^0_T})^2\|v_{N_2}\|_{X^{0,\fr12,1}_T}^2,
	\end{align*}
	which leads to the desired result.
\end{proof}

\begin{lemm}\label{le3-2} Let $\alpha\in [\frac{7}{5},2] $,  $0<T<1$, $v\in X^{0,\fr13-}_T $ and $w \in Y^r_T$ satisfying
	\eqref{Eqw} on $]0,T[\times \R$  with $L_{\al+1}$ satisfying Hypothesis \ref{Hy1}. Then
	for $N_1 \geq 1,\ N_1\geq N_2$ it holds
	\begin{align}\label{tr3}
		\|w_{N_1}v_{N_2}\|_{L^2_T L^2_x} &
		\lesssim T^{\fr16-} N_1^{\fr{4-2\al}{9}-} N_2^{\fr16+} N_1^{-r} \|w\|_{Y^r_T} (1+\|z\|_{Y^0_T}) \|v_{N_2}\|_{X_T^{0,\fr13-}}.
	\end{align}
	whenever the conditions \eqref{eq92} are fulfilled.
	Moreover when $N_1 \geq 1,\ N_1 \gg N_2$ it holds
	\begin{align}\label{tr4}
		\|w_{N_1}v_{N_2}\|_{L^2_T L^2_x} &
		\lesssim  N_1^{\fr{5-4\al}{9}-} N_2^{\fr16+} N_1^{-r} \|w\|_{Y^r_T} (1+\|z\|_{Y^0_T}) \|v_{N_2}\|_{X_T^{0,\fr13-}}.
	\end{align}
\end{lemm}
\begin{proof}
	In view of Lemma \ref{Cw2} and the continuous embedding $ X^{0,\frac12+}_T \hookrightarrow L^\infty_T L^2_x $ we have
	\begin{align*}
		\|w_{N_1}v_{N_2}\|_{L^2_T L^2_x} \lesssim \|w_{N_1}\|_{L^2_T L^\infty_x} \|v_{N_2}\|_{L^\infty_T L^2_x}
		\lesssim T^{\fr14} N_1^{\fr{2-\al}{3}} N_1^{-r} \|w\|_{Y^r_T} (1+\|z\|_{Y^0_T}) \|v_{N_2}\|_{X_T^{0,\fr12+}} \, .
	\end{align*}
	Interpolating with the crude estimate
	\begin{align}
		\|w_{N_1}v_{N_2}\|_{L^2_T L^2_x} \lesssim \|w_{N_1}\|_{L^\infty_T L^2_x} \|v_{N_2}\|_{L^2_T L^\infty_x}
		\lesssim N_2^{\fr12} N_1^{-r} \|w\|_{Y^r_T} \|v_{N_2}\|_{L^2_T L^2_x}.
	\end{align}
	we obtain \eqref{tr3}. \eqref{tr4} follows in the same way by interpolating between \eqref{zd} and the above crude estimate.
\end{proof}

\vspace*{1em}
\noindent\textbf{Acknowledgements.} 
 This work was conducted during a visit of the second author at Institut Denis Poisson (IDP) of Université de Tours in France. The second author is deeply grateful to IDP for its kind hospitality. The second author also acknowledges support from the China Scholarship Council and the National Natural Science Foundation of China (No. 12201118).

\end{document}